\documentclass[11pt]{article}
\usepackage[tbtags]{amsmath}
\usepackage{amssymb}
\usepackage{amsthm}
\usepackage[misc]{ifsym}
\usepackage{cases}
\usepackage{xcolor}
\usepackage{mathtools}
\usepackage{mathrsfs}
\usepackage{graphicx}
\usepackage{float}
\usepackage[colorlinks=true, linkcolor=black, citecolor=black, urlcolor=black]{hyperref}
\usepackage{subcaption} 

\numberwithin{equation}{section}
\normalsize

\title{\bf Linear-Quadratic Partially Observed Mean Field Stackelberg Stochastic Differential Game with Applications\thanks{This work is supported by National Key R\&D Program of China (2022YFA1006104), National Natural Science Foundation of China (12471419, 12401577, 12271304), Shandong Provincial Natural Science Foundation (ZR2024ZD35), Postdoctoral Fellowship Program (Grade B) of China Postdoctoral Science Foundation (GZB20230168), China Postdoctoral Science Foundation (2024M750488, 2025T180851), and Shanghai Post-doctoral Excellence Program (2023229).}}
\author{\normalsize Yu Si\thanks{\it School of Mathematics, Shandong University, Jinan 250100, P.R. China, E-mail: 202112003@mail.sdu.edu.cn},\quad Yueyang Zheng\thanks{\it School of Mathematical Sciences, Fudan University, Shanghai, 200433, P.R. China, E-mail: zhengyueyang0106@163.com},\quad Jingtao Shi\thanks{\it Corresponding author. School of Mathematics, Shandong University, Jinan 250100, P.R. China, E-mail: shijingtao@sdu.edu.cn}}

\date{}

\newtheorem{myprob}{Problem}[section]
\newtheorem{mypro}{Proposition}[section]
\newtheorem{mythm}{Theorem}[section]
\newtheorem{mydef}{Definition}[section]
\newtheorem{mylem}{Lemma}[section]
\newtheorem{myassump}{Assumption}[section]
\newtheorem{myremark}{Remark}[section]

\begin{document}
	
	\maketitle
	
	\noindent{\bf Abstract:}\quad This paper is concerned with a linear-quadratic partially observed mean field Stackelberg stochastic differential game, which contains a leader and a large number of followers. Specifically, the followers confront a large-population Nash game subsequent to the leader's initial announcement of his strategy. In turn, the leader optimizes his own cost functional, taking into account the anticipated reactions of the followers. The state equations of both the leader and the followers are general stochastic differential equations, where the drift terms contain both the state average term and the state expectation term. However, the followers' state average terms enter into the drift term of the leader's state equation and the state expectation term of the leader enters into the state equation of the follower, reflecting the mutual influence between the leader and the followers. By utilizing the techniques of state decomposition and backward separation principle, we deduce the open-loop adapted decentralized strategies and feedback decentralized strategies of this leader-followers system, and demonstrate that the decentralized strategies are the corresponding $\varepsilon$-Stackelberg-Nash equilibrium. Finally, we apply the theoretical result to a product planning problem with sticky prices. 
	
	\vspace{2mm}
	
	
	\vspace{2mm}
	
	
	\section{Introduction}
	
	In recent years, the study of dynamic optimization for large-population stochastic population systems has gained significant attention across a wide range of disciplines, including finance, engineering, and biological modeling. Large population systems are characterized by a large number of agents, where the influence of any single individual is negligible, yet their collective behavior exerts a substantial impact on all agents. A key feature of this system is the weak coupling through state averages or empirical distributions in the dynamics or the cost functionals. An efficient approach is to discuss the associated {\it mean field game} (MFG) to determine an approximate equilibrium by analyzing the related limiting behavior. The existing literature mainly offers two approaches to addressing large-population optimization problems. The first approach, introduced by Huang et al. \cite{Huang-Caines-Malhame-2007,Huang-Malhame-Caines-2006}, involves replacing the coupling terms with limiting processes to solve the auxiliary control problem for individual agents. This method formalizes a fixed-point problem by determining the limiting process, known as the fixed-point (or top-down) approach, also referred to as Nash certainty equivalence. The second approach, proposed by Lasry and Lions \cite{Lasry-Lions-2007}, treats the large-population problem as a high-dimensional control problem. It first derives the optimal solution for the high-dimensional control problem and then takes the limit as $N\rightarrow\infty$ to obtain the optimal solution, termed the direct (or bottom-up) approach. The decentralized strategies derived from these two methods for large-population problems can be applied to models with a large but finite number of agents to achieve an $\epsilon$-Nash equilibrium. This means that even if each agent has access to centralized information about all agents, their cost functions would not improve significantly. Compared to the Nash strategies determined under centralized information, these decentralized strategies require only the private information of each agent, offering lower computational and implementation complexity. The interested readers can refer to \cite{Bardi-Priuli-2014, Du-Huang-Wu-2018, Hu-Huang-Li-2018, Hu-Huang-Nie-2018, Huang-Zhou-2020, Li-Li-Wu-2022, Moon-Basar-2017, Wang-Zhang-2017} for {\it linear-quadratic} (LQ) MFGs, and refer to \cite{Bensoussan-Frehse-Yam-2013, Buckdahn-Djehiche-Li-Peng-2009, Carmona-Delarue-2013, Cong-Shi-2024, Huang-2010, Nguyen-Huang-2012} for further analysis of MFGs and related topics.
	
	We note that the aforementioned literature primarily focuses on scenarios where all agents have access to complete information. However, in real-world applications, agents often operate under partial information constraints.   Baghery and \O ksendal \cite{Baghery-Oksendal-2007} gave the sufficient and necessary  maximum principles for the type of partial information control. Huang and Wang \cite{Huang-Wang-2016} explored dynamic optimization problems for large-population systems with partial information. Huang et al. \cite{Huang-Wang-Wu-2016} further investigated backward large-population systems under both full and partial information settings. Wang et al. \cite{Wang-Wu-Xiong-2018} conducted a comprehensive study on optimal control problems for {\it forward-backward stochastic differential equations} (FBSDEs) with incomplete information. Additionally, \c Sen and Caines \cite{Sen-Caines-2016, Sen-Caines-2019} studied MFG problems under partial observation. Bensoussan et al. \cite{Bensoussan-Feng-Huang-2021} analyzed a class of {\it linear-quadratic-Gaussian} (LQG) MFGs with partial observation and common noise. Huang et al. \cite{Huang-Wang-Wang-Xiao-2024, Huang-Wang-Wang-Wang-2023} studied forward-backward stochastic MFGs with partial observation and common noise, as well as backward stochastic MFGs with partial information and common noise, respectively. Li et al. \cite{Li-Nie-Wang-Yan-2024,Li-Nie-Wu-2023} also investigated LQ large-population problems under partial information and the scenario with common noise. Li et al. \cite{Li-Li-Wu-2025} studied a class of linear-quadratic MFG problems with partial observation, where the state equations of all agents are coupled through a control averaging term. By utilizing the fixed-point method, state decomposition and backward separation techniques, they derived the state feedback form of decentralized optimal strategies. Chen et al. \cite{Chen-Du-Wu-2024} investigated a class of linear-quadratic MFG problems with partial observation and convex cost functionals. In their work, the state equations include both state average terms and state expectation terms. They also established the well-posedness of solutions to a conditional mean-field forward-backward stochastic differential equation with filtering. For a comprehensive study on systems with partial observation, we may refer to \cite{Bensoussan-1992,Liptser-Shiryayev-1977} .
	
	The Stackelberg differential game, also known as the leader-follower differential game, arises in situations where certain agents hold dominant positions in the market. von Stackelberg \cite{Stackelberg-1952} introduced the concept of the hierarchical solution. Yong \cite{Yong-2002} delved into a generalized framework of LQ stochastic leader-follower differential games. Shi et al. \cite{Shi-Wang-Xiong-2016} studied a stochastic leader-follower differential game with asymmetric information. Nourian et al. \cite{Nourian-Caines-Malhame-Huang-2012} studied a large-population LQ leader-follower stochastic multi-agent systems and established their ($\varepsilon_1,\varepsilon_2$)-Stackelberg-Nash equilibrium. Wang and Zhang \cite{WZ14} studied discrete-time Stackelberg games for multi-agent systems involving a leader and many followers with infinite horizon tracking-type costs, a set of centralized and distributed strategies are designed, separately. Bensoussan et al. \cite{Benssousan-Chau-Yam-2015,Benssousan-Chau-Yam-2016,Benssousan-Chau-Yam-2017} systematically studied $N$-player games with a leader, proposing a convergence analysis between $N$-player games and mean-field games. Moon and Ba\c{s}ar \cite{Moon-Basar-2018} investigated continuous-time mean field LQ Stackelberg stochastic differential games by the fixed-point method. Lee et al. \cite{LNM19} considered a discrete-time leader-follower decentralized optimal control for a hexarotor group with one leader and large-population heterogeneous followers which are subject to partial information, where they obtained a set of decentralized optimal control for the leader and $N$ follower hexarotors by the mean filed Stackelberg game theory, and then constitutes an $\epsilon$-Stackelebrg equilibrium. Si and Wu \cite{Si-Wu-2021} explored a backward-forward LQ Stackelberg MFG, where the leader's state equation is backward, and the follower's state equation is forward. Feng and Wang \cite{Feng-Wang-2024} investigated an LQ social optimal problem with a leader and a amount of followers.  Wang \cite{Wang-2024} employed a direct method to solve LQ Stackelberg MFGs with a leader and a substantial number of followers. Cong et al. \cite{Cong-Shi-Wang-2024} employed direct method to investigate a class of Stackelberg MFGs and social optimality with an arbitrary number of populations. Very recently, Si and Shi \cite{Si-Shi-2025} researched an LQ mean field Stackelberg stochastic differential game with partial information and common noise.
	
	In this paper, we consider a new class of LQ partially observed mean field Stackelberg stochastic differential games, which contains a leader and a number of followers. The leader first announces his strategy and then each follower optimizes its own cost based on the leader's announcement. At last, the leader chooses his own optimal strategy based on the responses of the followers. Compared with the existing literatures, the contributions of this paper are listed as follows.
	\begin{itemize}
		\item We introduced a general LQ partially observed mean field Stackelberg stochastic differential game, which contains a leader and a large number of followers. The leader and followers can only access information provided by the observation processes. Due to the framework of partial observation, we need to use variational technique and optimal filter technique to obtain the open-loop decentralized optimal strategy and feedback decentralized optimal strategy of the auxiliary limiting problem.
		\item The state equation and cost functional of the leader both contain the state average term of followers. This indicates that the collective behavior of the followers has a significant impact on the leader. Additionally, the leader's state appears in the followers' state equations and cost functionals, meaning that the leader's actions can directly affect the followers' performance measures. Moreover, both the leader's and the followers' state equations contain their own state expectation terms, which enhances both the computational complexity and the generality of the framework.
		\item By utilizing the techniques of state decomposition and backward separation principle, we overcome the difficulty of circular dependency arising from the partial observation framework and obtian optimal feedback decentralized strategies of this leader-followers system, and demonstrate that the decentralized strategies are the corresponding $\varepsilon$-Stackelberg-Nash equilibrium. Furthermore we apply decoupling technique to obtain the solvability of Hamiltonian system.
        \item Theoretical results are applied to a product planning problem with sticky prices.
	\end{itemize}
	
	The rest of this paper is organized as follows. In Section 2, we formulate our problem. In Section 3, we introduce the auxiliary limiting problem and solve the problem of the followers and the leader in turn and derive the decentralized optimal strategies. In Section 4, we prove the decentralized optimal strategies are the $\varepsilon$-Stackelberg-Nash equilibria of the games. Moreover, we apply the theoretical result to a product planning problem with sticky prices in Section 5.   Finally, the conclusion is given in Section 6.
	
	\section{Problem formulation}
	
	Firstly, we introduce some notations that will be used throughout the paper. We consider a finite time interval $[0, T]$ for a fixed $T > 0$.
	Let $\big(\Omega, \mathcal{F},  \left\{\mathcal{F}_t\right\}_{t\geq0},\mathbb{P}\big)$ be a complete filtered probability space, on which
	$\left\{W_0(s),\overline{W}_0(s), W_i(s),\overline{W}_i(s),0 \leq s \leq t, 1 \leq i \leq N\right\}$ is a standard $(2nN+2n)$-dimensional Brownian motion. Then, for any $0 \leq t \leq T$, we define
	$$
	\mathcal{F}_t:=\sigma\left\{W_0(s),\overline{W}_0(s), W_i(s),\overline{W}_i(s),0 \leq s \leq t, 1 \leq i \leq N\right\}\vee \mathcal{N}_{\mathbb{P}},
	$$
	where $\mathcal{N}_{\mathbb{P}}$ is the class of all $\mathbb{P}$-null sets. Let
	$$\mathcal{F}_t^0:=\sigma\left\{W_0(s),\overline{W}_0(s), 0 \leq s \leq t\right\}\vee \mathcal{N}_{\mathbb{P}},$$
	$$\mathcal{F}_t^i:=\sigma\left\{W_i(s),\overline{W}_i(s), 0 \leq s \leq t\right\}\vee \mathcal{N}_{\mathbb{P}} ,\quad 1 \leq i \leq N,$$
	$$\widetilde{\mathcal{F}}_t^{i}:=\sigma\left\{W_i(s),\overline{W}_i(s),W_0(s),\overline{W}_0(s), 0 \leq s \leq t\right\}\vee \mathcal{N}_{\mathbb{P}} ,\quad 1 \leq i \leq N.$$

	Let $\mathbb{R}^n$ be an $n$-dimensional Euclidean space with norm and inner product being defined as $|\cdot|$ and $\langle\cdot, \cdot\rangle$, respectively.
	Next, we introduce four necessary spaces frequently used in this paper. $C([0,T],\mathbb{R}^n)$ is the space of all $\mathbb{R}^n$-valued continuous functions defined on $[0,T]$. $L^2(0,T;\mathbb{R}^n)$ is the space of $\mathbb{R}^n$-valued deterministic function defined on $[0,T]$ satisfying $\int_0^T |f(t)|^2 dt < \infty$. A bounded, measurable function $f(\cdot):[0, T] \rightarrow \mathbb{R}^n$ is denoted as $f(\cdot) \in L^{\infty}(0, T; \mathbb{R}^n)$.
	An $\mathbb{R}^n$-valued, $\mathbb{F}$-adapted stochastic process $f(\cdot): \Omega \times [0, T] \rightarrow \mathbb{R}^n$ satisfying $\mathbb{E} \int_0^T |f(t)|^2 dt < \infty$ is denoted as $f(\cdot) \in L_{\mathbb{F}}^2(0, T; \mathbb{R}^n)$. 
	
	For any random variable or stochastic process $X$ and filtration $\mathcal{H}$, $\mathbb{E}X$ and $\mathbb{E}[X|\mathcal{H}]$ represent the mathematical expectation and conditional mathematical expectation of $X$, respectively. For a given vector or matrix \(M\), let \(M^{\top}\) represent its transpose. We denote the set of symmetric \(n \times n\) matrices (resp. positive semi-definite matrices) with real elements by \(\mathcal{S}^n\) (resp. \(\mathcal{S}_{+}^n\)), and $\|\cdot\|_Q^2 \triangleq\langle Q \cdot, \cdot\rangle$ for any $Q \in \mathcal{S}^n$. If \(M \in \mathcal{S}^n\) is positive (semi) definite, we abbreviate it as \(M > (\geq) 0\). For a positive constant \(k\), if \(M \in \mathcal{S}^n\) and \(M > kI\), we label it as \(M \gg 0\).
	
	Now, let us focus on a large-population system comprised of $N+1$ individual agents, denoted as $\left\{\mathcal{A}_i\right\}_{0 \leq i \leq N}$. The state $x^{u_0,u}_0(\cdot)$ of the leader $\mathcal{A}_0$ is given by the following linear {\it stochastic differential equation} (SDE):
	\begin{equation}\label{leader state}
		\left\{\begin{aligned}
			d x^{u_0,u}_0(t)= &\ \left[A_0(t) x^{u_0,u}_0(t)+B_0(t) u_0(t)+\bar{A}_0(t) \mathbb{E}x^{u_0,u}_0(t)+C_0(t) {x^{u_0,u}}^{(N)}(t)+b_0(t)\right] d t  \\
			&+\sigma_0(t) d W_0(t)+\bar{\sigma}_0(t) d \overline{W}_0(t), \\
			x^{u_0,u}_0(0)= &\ \xi_0,
		\end{aligned}\right.
	\end{equation}
	and the state $x^{u_0,u}_i(\cdot)$ of the follower $\mathcal{A}_i$ is given by the following linear $\mathrm{SDE}$:
	\begin{equation}\label{follower state}
		\left\{\begin{aligned}
			d x^{u_0,u}_i(t)  =&\ \Big[A(t) x^{u_0,u}_i(t)+B(t) u_i(t)+\bar{A}(t) \mathbb{E}x^{u_0,u}_i(t)+C(t) {x^{u_0,u}}^{(N)}(t)\\
			& +F(t) \mathbb{E}x^{u_0,u}_0(t)+b(t)\Big] d t+\sigma(t) d W_i(t)+\bar{\sigma}(t) d \overline{W}_i(t), \\
			x^{u_0,u}_i(0)  =&\ \xi,
		\end{aligned}\right.
	\end{equation}
	where $\xi,\xi_0 \in \mathbb{R}^n$ represent the initial states, $u_0(\cdot)$ and $u_i(\cdot)$ denote control processes of the leader $\mathcal{A}_0$ and follower $\mathcal{A}_i$, respectively. Let $u(\cdot)\equiv(u_1(\cdot),u_2(\cdot),\cdots,u_N(\cdot))$ denotes the control strategy of all followers $\{\mathcal{A}_i\}_{1\leq i \leq N}$. Moreover,  $x^{u_0,u}_0(\cdot)$ and $x^{u_0,u}_i(\cdot)$ denote state  processes of the leader $\mathcal{A}_0$ and follower $\mathcal{A}_i$ corresponding to control strategy $(u_0(\cdot),u(\cdot))$, respectively.  ${x^{u_0,u}}^{(N)}(\cdot) := \frac{1}{N} \sum_{i=1}^N x^{u_0,u}_i(\cdot)$ signifies the average state of the followers corresponding to control strategy $(u_0(\cdot),u(\cdot))$. In addition, leader $\mathcal{A}_0$ can only have assess to a related observation process driven by the following SDE:
	\begin{equation}\label{leader observation}
		\left\{\begin{aligned}
			d Y^{u_0,u}_0(t)  =&\left[H_0(t) x^{u_0,u}_0(t)+\bar{H}_0(t) \mathbb{E}x^{u_0,u}_0(t)+I_0(t) {x^{u_0,u}}^{(N)}(t)+h_0(t)\right] d t +f_0(t) d \overline{W}_0(t) , \\
			Y^{u_0,u}_0(0)  =&\ 0.
		\end{aligned}\right.
	\end{equation}
	The observation process of follower $\mathcal{A}_i$ satisfies
	\begin{equation}\label{follower observation}
		\left\{\begin{aligned}
			d Y^{u_0,u}_i(t)  =&\left[H(t) x^{u_0,u}_i(t)+\bar{H}(t) \mathbb{E}x^{u_0,u}_i(t)+I(t) {x^{u_0,u}}^{(N)}(t)+h(t)\right] d t +f(t) d \overline{W}_i(t) , \\
			Y^{u_0,u}_i(0)  =&\ 0.
		\end{aligned}\right.
	\end{equation}
	The corresponding coefficients are deterministic functions with appropriate dimensions.
	
	Then, let $\mathcal{F}_t^{Y^{u_0,u}_i}$ is the observable information of the agent $\mathcal{A}_i$, $i=0,1,\ldots,N$, where $\mathcal{F}_t^{Y^{u_0,u}_i}=\sigma\left\{Y^{u_0,u}_i(s), 0 \leq s \leq t\right\}$. Next, we define the admissible control set of the follower $\mathcal{A}_i$ by
	$$
	\mathcal{U}^{ad}_i=\left\{u_i(\cdot) \mid u_i(\cdot) \in L_{\mathcal{F}_t^{Y^{u_0,u}_i}}^2\big(0, T ; \mathbb{R}^k\big)\right\}.
	$$
	The admissible control set of the leader $\mathcal{A}_0$ is defined by
	$$
	\mathcal{U}^{ad}_0=\left\{u_0(\cdot) \mid u_0(\cdot) \in L_{\mathcal{F}_t^{Y^{u_0,u}_0}}^2\big(0, T ; \mathbb{R}^k\big)\right\}.
	$$
	For the leader $\mathcal{A}_0$, the cost functional is defined by
	\begin{equation}\label{leader cost}
		\begin{aligned}
			\mathcal{J}_0&\left(u_0(\cdot),u(\cdot)\right)=\frac{1}{2} \mathbb{E} \bigg\{\int_0^T \bigg(\|x^{u_0,u}_0(t)-\Gamma_0  {x^{u_0,u}}^{(N)}(t)-\bar{\Gamma}_0  \mathbb{E}x^{u_0,u}_0(t)-\eta_0(t)\|_{Q_0}^2  \\
			&+\|u_0(t)\|^2_{R_0} +2\left\langle S_0(t)(x^{u_0,u}_0(t)-\Gamma_0(t){x^{u_0,u}}^{(N)}(t)-\bar{\Gamma}_0(t)\mathbb{E}x^{u_0,u}_0(t)),u_0(t)\right\rangle\bigg)  dt\\
			& + \|x^{u_0,u}_0(T)-\Gamma_1 {x^{u_0,u}}^{(N)}(T)-\bar{\Gamma}_1  \mathbb{E}x^{u_0,u}_0(T)-\eta_1\|_{G_0}^2 \bigg\}\,.
		\end{aligned}
	\end{equation}
	Let $u_{-i}(\cdot)\equiv\left(u_1(\cdot), \ldots, u_{i-1}(\cdot), u_{i+1}(\cdot), \ldots, u_N(\cdot)\right)$ be the set of control strategies of followers except for $i$-th follower $\mathcal{A}_i$. For $i=1, \cdots, N$, the cost functional of the $i$-th follower $\mathcal{A}_i$ is defined by
	\begin{equation}\label{follower cost}
		\begin{aligned}
			\mathcal{J}_i&\left(u_i(\cdot), u_{-i}(\cdot),u_0(\cdot)\right)=\frac{1}{2} \mathbb{E} \bigg\{\int_0^T \bigg(\|x^{u_0,u}_i(t)-\Gamma_2  {x^{u_0,u}}^{(N)}(t)-\bar{\Gamma}_2  \mathbb{E}x^{u_0,u}_i(t)-\Gamma_3 x^{u_0,u}_0(t)\\
			&-\bar{\Gamma}_3\mathbb{E}x^{u_0,u}_0(t) -\eta_2(t)\|_Q^2+\|u_i(t)\|^2_R +2\left\langle S(t)(x^{u_0,u}_i(t)-\bar{\Gamma}_2(t)\mathbb{E}x^{u_0,u}_i(t)),u_i(t)\right\rangle\bigg)   dt  \\
			&+\|x^{u_0,u}_i(T)-\Gamma_4  {x^{u_0,u}}^{(N)}(T)-\bar{\Gamma}_4  \mathbb{E}x^{u_0,u}_i(T)-\Gamma_5 x^{u_0,u}_0(T) -\bar{\Gamma}_5\mathbb{E}x^{u_0,u}_0(T)-\eta_4\|_G^2 \bigg\}.
		\end{aligned}
	\end{equation}
	
	Moreover, we introduce the following assumptions of coefficients.
	\begin{myassump}\label{A1}
		The coefficients satisfy the following conditions:
		
		(i) $A_0(\cdot),\bar{A}_0(\cdot),C_0(\cdot), \sigma_0(\cdot), \bar{\sigma}_0(\cdot)$, $A(\cdot),\bar{A}(\cdot), C(\cdot),F(\cdot), \sigma(\cdot), \bar{\sigma}(\cdot)$, $H_0(\cdot),\bar{H}_0(\cdot),I_0(\cdot), f_0(\cdot)$,\\ $H(\cdot), \bar{H}(\cdot), I(\cdot), f(\cdot) \in L^{\infty}\left(0, T ; \mathbb{R}^{n \times n}\right)$, $b_0(\cdot),h_0(\cdot),b(\cdot),h(\cdot)\in L^{\infty}\left(0, T ; \mathbb{R}^{n}\right)$, and $B(\cdot), B_0(\cdot) \in\\ L^\infty\left(0, T; \mathbb{R}^{n \times k}\right) $;
		
		(ii) $Q(\cdot),Q_0(\cdot) \in L^{\infty}\left(0, T ; \mathcal{S}^n\right)$, $R(\cdot),R_0(\cdot) \in L^{\infty}\left(0, T ; \mathcal{S}^k\right)$, $S(\cdot), S_0(\cdot) \in L^\infty \left(0, T; \mathbb{R}^{k \times n}\right) $ $\Gamma_0(\cdot),\bar{\Gamma}_0(\cdot),\Gamma_2(\cdot),\bar{\Gamma}_2(\cdot),\Gamma_3(\cdot),\bar{\Gamma}_3(\cdot) \in L^{\infty}\left(0, T ; \mathbb{R}^{n\times n}\right)$, $\eta_0(\cdot),\eta_2(\cdot) \in L^\infty\left(0, T ; \mathbb{R}^n\right)$, $G,G_0 \in \mathcal{S}^n$,  $\Gamma_1,\bar{\Gamma}_1,\Gamma_4,\bar{\Gamma}_4,\Gamma_5,\bar{\Gamma}_5 \in \mathbb{R}^{n\times n}$, $\eta_1,\eta_4 \in \mathbb{R}^n$ and $R(\cdot),R_0(\cdot)>0$, $Q(\cdot),Q_0(\cdot) \geq 0$, $G, G_0 \geq 0$, $Q(\cdot)-S^\top(\cdot) R^{-1}(\cdot)S(\cdot), Q_0(\cdot)-S_0^\top(\cdot) R_0^{-1}(\cdot)S_0(\cdot) \geq 0$.
	\end{myassump}
	
	We mention that under Assumption \ref{A1}, the system of states (\ref{leader state})-(\ref{follower state}) and the system of observations (\ref{leader observation})-(\ref{follower observation}) admits a unique solution by Yong \cite{Yong-2013}. And then, the cost functionals (\ref{leader cost}), (\ref{follower cost}) are well-defined.
	
	Our {\it LQ mean field Stackelberg stochastic differential game with partial observation}, investigated in this paper, can be stated as follows.
	
	\begin{myprob}\label{problem centralized}
		Finding a set of strategies $\left(u_0^*(\cdot), u_1^*(\cdot), \cdots, u_N^*(\cdot)\right)$ satisfying the following conditions:
		
		(i) given $u_0(\cdot) \in \mathcal{U}_{0}^{ad}$, for $1 \leq i \leq N$, finding the control strategy set $u^*(\cdot;u_0)\equiv(u^*_1(\cdot;u_0), \cdots$,\\ $u^*_N(\cdot;u_0))$ of the followers $\mathcal{A}_1, \cdots, \mathcal{A}_N$, such that
		\begin{equation*}
			\mathcal{J}_i\left(u^*_i(\cdot;u_0), u^*_{-i}(\cdot;u_0),u_0(\cdot)\right)=\inf _{u_i(\cdot) \in\, \mathcal{U}_i^{ad}} \mathcal{J}_i\left(u_i(\cdot), u^*_{-i}(\cdot;u_0),u_0(\cdot)\right);
		\end{equation*}
		
		(ii) finding the control strategy  $u^*_0(\cdot)$ of the leader $\mathcal{A}_0$ such that
		\begin{equation*}
			\mathcal{J}_0\left(u^*_0(\cdot ),u^*(\cdot;u_0^*)\right)=\inf _{u_0(\cdot) \in\, \mathcal{U}_{0}^{ad}} \mathcal{J}_0\left(u_0(\cdot ),u^*(\cdot;u_0)\right).
		\end{equation*}		
	\end{myprob}
	
	We call $\left(u_0^*(\cdot), u_1^*(\cdot), \cdots, u_N^*(\cdot)\right)$ the {\it Stackelberg-Nash equilibrium} of \textbf{Problem \ref{problem centralized}}. Moreover, the corresponding state $\left(x_0^*(\cdot), x_1^*(\cdot), \cdots, x_N^*(\cdot)\right)$ is called the {\it optimal trajectory}.
	
	\section{Limiting Stackelberg-Nash equilibria}
	
	Since the intricate nature arising from the coupling of state-average ${x^{u_0,u}}^{(N)}(\cdot):= \frac{1}{N} \sum_{i=1}^N x^{u_0,u}_i(\cdot)$ and the curse of dimensionality caused by a large number of agents, \textbf{Problem \ref{problem centralized}} becomes challenging to investigate. We shall employ the MFG theory to seek an approximate Stackelberg-Nash equilibrium, which serves as a bridge between the \textbf{Problem \ref{problem centralized}} and the limiting problem  when  the number of agents $N$ approaches infinity. Typically, the state-average is replaced by its frozen limit term for computational convenience.
	
	As $N\rightarrow+\infty$, let us assume that the state-average ${x^{u_0,u}}^{(N)}(\cdot)$ can be approximated by deterministic function $z(\cdot)$ that will be subsequently determined in the later analysis. For the leader $\mathcal{A}_0$, we introduce the following auxiliary state $\bar{x}^{u_0}_0(\cdot)\in L^2_{\mathcal{F}_t^0}(0,T;\mathbb{R}^{n})$ denoting state  processes of the leader $\mathcal{A}_0$ corresponding to control process $u_0(\cdot)$, which satisfies the following linear SDE:
	\begin{equation}\label{leader limiting state}
		\left\{\begin{aligned}
			d \bar{x}^{u_0}_0(t)= &\ \left[A_0(t) \bar{x}^{u_0}_0(t)+B_0(t) u_0(t)+\bar{A}_0(t) \mathbb{E}\bar{x}^{u_0}_0(t)+C_0(t) z(t)+b_0(t)\right] d t  \\
			&+\sigma_0(t) d W_0(t)+\bar{\sigma}_0(t) d \overline{W}_0(t)  , \\
			\bar{x}^{u_0}_0(0)= &\ \xi_0,
		\end{aligned}\right.
	\end{equation}
	and its corresponding observation process satisfies
	\begin{equation}\label{leader limiting observation}
		\left\{\begin{aligned}
			d \bar{Y}^{u_0}_0(t)  =&\left[H_0(t) \bar{x}^{u_0}_0(t)+\bar{H}_0(t) \mathbb{E}\bar{x}^{u_0}_0(t)+I_0(t) z(t)+h_0(t)\right] d t +f_0(t) d \overline{W}_0(t) , \\
			\bar{Y}^{u_0}_0(0)  =&\ 0.
		\end{aligned}\right.
	\end{equation}
	The limiting cost functional becomes
	\begin{equation}\label{leader limiting cost}
		\begin{aligned}
			J_0&\left(u_0(\cdot)\right)=\ \frac{1}{2} \mathbb{E} \bigg\{\int_0^T \bigg(\|\bar{x}^{u_0}_0(t)-\Gamma_0  z(t)-\bar{\Gamma}_0  \mathbb{E}\bar{x}^{u_0}_0(t)-\eta_0(t)\|_{Q_0}^2 \\
			&+\|u_0(t)\|^2_{R_0} +2\left\langle S_0(t)(\bar{x}^{u_0}_0(t)-\Gamma_0(t)z(t)-\bar{\Gamma}_0(t)\mathbb{E}\bar{x}^{u_0}_0(t)),u_0(t)\right\rangle\bigg) dt\\
			& + \|\bar{x}^{u_0}_0(T)-\Gamma_1 z(T)-\bar{\Gamma}_1  \mathbb{E}\bar{x}^{u_0}_0(T)-\eta_1\|_{G_0}^2 \bigg\}\,.
		\end{aligned}
	\end{equation}
	
	For each follower $\mathcal{A}_i$, $i=1,2,\cdots,N$, we introduce the following auxiliary state $\bar{x}^{u_0,u_i}_i(\cdot)\in L^2_{\mathcal{F}_t^i}(0,T;\mathbb{R}^{n})$ denoting state  processes of the follower $\mathcal{A}_i$ corresponding to control strategy $(u_0(\cdot),u_i(\cdot))$, which satisfies the following linear SDE:
	\begin{equation}\label{follower limiting state}
		\left\{\begin{aligned}
			d \bar{x}^{u_0,u_i}_i(t)  =&\Big[A(t) \bar{x}^{u_0,u_i}_i(t)+B(t) u_i(t)+\bar{A}(t) \mathbb{E}\bar{x}^{u_0,u_i}_i(t)+C(t) z(t)\\
			& +F(t) \mathbb{E}\bar{x}^{u_0}_0(t)+b(t)\Big] d t+\sigma(t) d W_i(t)+\bar{\sigma}(t) d \overline{W}_i(t)   , \\
			\bar{x}^{u_0,u_i}_i(0)  =&\ \xi,
		\end{aligned}\right.
	\end{equation}
	and its corresponding observation process satisfies
	\begin{equation}\label{follower limiting observation}
		\left\{\begin{aligned}
			d \bar{Y}^{u_0,u_i}_i(t)  =&\left[H(t) \bar{x}^{u_0,u_i}_i(t)+\bar{H}(t) \mathbb{E}\bar{x}^{u_0,u_i}_i(t)+I(t) z(t)+h(t)\right] d t +f(t) d \overline{W}_i(t) , \\
			\bar{Y}^{u_0,u_i}_i(0)  =&\ 0.
		\end{aligned}\right.
	\end{equation}
	The limiting cost functional becomes
	\begin{equation}\label{follower limiting cost}
		\begin{aligned}
			J_i&\left(u_i(\cdot),u_0(\cdot)\right)=\frac{1}{2} \mathbb{E} \bigg\{\int_0^T \bigg(\|\bar{x}^{u_0,u_i}_i(t)-\Gamma_2  z(t)-\bar{\Gamma}_2  \mathbb{E}\bar{x}^{u_0,u_i}_i(t)-\Gamma_3 \bar{x}_0^{u_0}(t)-\bar{\Gamma}_3\mathbb{E}\bar{x}_0^{u_0}(t)\\
			&-\eta_2(t)\|_Q^2+\|u_i(t)\|^2_R +2\left\langle S(t)(x_i^{u_0,u_i}(t)-\bar{\Gamma}_2(t)\mathbb{E}x_i^{u_0,u_i}(t)),u_i(t)\right\rangle\bigg) dt\\
			&+\|\bar{x}^{u_0,u_i}_i(T)-\Gamma_4  z(T)-\bar{\Gamma}_4  \mathbb{E}\bar{x}^{u_0,u_i}_i(T)-\Gamma_5 \bar{x}_0^{u_0}(T) -\bar{\Gamma}_5\mathbb{E}\bar{x}_0^{u_0}(T)-\eta_4\|_G^2 \bigg\}\,.
		\end{aligned}
	\end{equation}

	For partially observed stochastic control problems, it is natural that the control $u_i(\cdot)$ is related to the observation $\bar{Y}_i(\cdot)$. But the observation $\bar{Y}_i(\cdot)$ depends on the control $u_i(\cdot)$ through the state $\bar{x}^{u_0,u_i}_i(\cdot)$. In other words, there exists a circular dependency between control $u_i(\cdot)$ and observation $\bar{Y}_i(\cdot)$, which poses fundamental difficulties. To overcome this difficulty, we adopt state decomposition technique (see Wang et al. \cite{Wang-Wu-Xiong-2015}) to break the circular dependency.
	
	We introduce the following decomposition for the state-observation system of the leader:
	\begin{equation}\label{leader limiting state decomposition 1}
		\left\{\begin{aligned}
			d \bar{x}_{01}(t)= &\left[A_0(t) \bar{x}_{01}(t)+\bar{A}_0(t) \mathbb{E}\bar{x}_{01}(t)+C_0(t) z_{1}(t)\right] d t  +\sigma_0(t) d W_0(t)+\bar{\sigma}_0(t) d \overline{W}_0(t)   , \\
			\bar{x}_{01}(0)= &\ \xi_0,
		\end{aligned}\right.
	\end{equation}
	\begin{equation}\label{leader limiting observation decomposition 1}
		\left\{\begin{aligned}
			d \bar{Y}_{01}(t)  =&\left[H_0(t) \bar{x}_{01}(t)+\bar{H}_0(t) \mathbb{E}\bar{x}_{01}(t)+I_0(t) z_{1}(t)\right] d t +f_0(t) d \overline{W}_0(t) , \\
			\bar{Y}_{01}(0)  =&\ 0,
		\end{aligned}\right.
	\end{equation}
	and
	\begin{equation}\label{leader limiting state decomposition 2}
		\left\{\begin{aligned}
			d \bar{x}_{02}^{u_0}(t)= &\left[A_0(t) \bar{x}_{02}^{u_0}(t)+B_0(t) u_0(t)+\bar{A}_0(t) \mathbb{E}\bar{x}_{02}^{u_0}(t)+C_0(t) z_{2}(t)+b_0(t)\right] d t   , \\
			\bar{x}_{02}^{u_0}(0)= &\ 0,
		\end{aligned}\right.
	\end{equation}
	\begin{equation}\label{leader limiting observation decomposition 2}
		\left\{\begin{aligned}
			d \bar{Y}_{02}^{u_0}(t)  =&\left[H_0(t) \bar{x}_{02}^{u_0}(t)+\bar{H}_0(t) \mathbb{E}\bar{x}_{02}^{u_0}(t)+I_0(t) z_{2}(t)+h_0(t)\right] d t , \\
			\bar{Y}_{02}^{u_0}(0)  =&0.
		\end{aligned}\right.
	\end{equation}
	For the state-observation system of the followers, we give the following decomposition:
	\begin{equation}\label{follower limiting state decomposition 1}
		\left\{\begin{aligned}
			d \bar{x}_{i1}(t)  =&\left[A(t) \bar{x}_{i1}(t)+\bar{A}(t) \mathbb{E}\bar{x}_{i1}(t)+C(t) z_{1}(t)+F(t) \mathbb{E}\bar{x}_{01}(t)\right] d t\\
			& +\sigma(t) d W_i(t)+\bar{\sigma}(t) d \overline{W}_i(t)  , \\
			\bar{x}_{i1}(0)  =&\ \xi,
		\end{aligned}\right.
	\end{equation}
	\begin{equation}\label{follower limiting observation decomposition 1}
		\left\{\begin{aligned}
			d \bar{Y}_{i1}(t)  =&\left[H(t) \bar{x}_{i1}(t)+\bar{H}(t)\mathbb{E} \bar{x}_{i1}(t)+I(t) z_{1}(t)\right] d t +f(t) d \overline{W}_i(t) , \\
			\bar{Y}_{i1}(0)  =&\ 0,
		\end{aligned}\right.
	\end{equation}
	and
	\begin{equation}\label{follower limiting state decomposition 2}
		\left\{\begin{aligned}
			d \bar{x}_{i2}^{u_0,u_i}(t)  =&\left[A(t) \bar{x}_{i2}^{u_0,u_i}(t)+\bar{A}(t) \mathbb{E}\bar{x}_{i2}^{u_0,u_i}(t)+B(t)u_i(t)+C(t) z_{2}(t)\right.\\
                                          &\quad +F(t) \mathbb{E}\bar{x}_{02}^{u_0,u_i}(t)+b(t)\big] d t , \\
			\bar{x}_{i2}^{u_0,u_i}(0)  =&\ 0,
		\end{aligned}\right.
	\end{equation}
	
	\begin{equation}\label{follower limiting observation decomposition 2}
		\left\{\begin{aligned}
			d \bar{Y}_{i2}^{u_0,u_i}(t)  =&\left[H(t) \bar{x}_{i2}^{u_0,u_i}(t)+\bar{H}(t)\mathbb{E} \bar{x}_{i2}^{u_0,u_i}(t)+I(t) z_{2}(t)+h(t)\right] d t  , \\
			\bar{Y}_{i2}^{u_0,u_i}(0)  =&\ 0.
		\end{aligned}\right.
	\end{equation}
	Then we can verify that $\bar{x}_0^{u_0}(\cdot)=\bar{x}_{01}(\cdot)+\bar{x}_{02}^{u_0}(\cdot)$, $\bar{Y}_0^{u_0}(\cdot)=\bar{Y}_{01}(\cdot)+\bar{Y}_{02}^{u_0}(\cdot)$, $\bar{x}^{u_0,u_i}_i(\cdot)=\bar{x}_{i1}(\cdot)+\bar{x}_{i2}^{u_0,u_i}(\cdot)$, $\bar{Y}_i^{u_0,u_i}(\cdot)=\bar{Y}_{i1}(\cdot)+\bar{Y}_{i2}^{u_0,u_i}(\cdot)$, $z(\cdot)=z_1(\cdot)+z_2(\cdot)$, where $z_1(\cdot)$ and $z_2(\cdot)$ are deterministic function which are respectively approximated by $\bar{x}^{(N)}_{1}(\cdot) \coloneqq \frac{1}{N} \sum_{i=1}^{N}\bar{x}_{i1}(\cdot)$ and $\bar{x}^{(N)^{u_0,u}}_{2}(\cdot) \coloneqq \frac{1}{N} \sum_{i=1}^{N}\bar{x}_{i2}^{u_0,u_i}(\cdot)$.
	
	Let
	\begin{equation}
	\begin{aligned}
	&\mathcal{F}_t^{\bar{Y}_{01}}=\sigma\left\{\bar{Y}_{01}(s), 0 \leq s \leq t\right\},\quad \mathcal{F}_t^{\bar{Y}_{i1}}=\sigma\left\{\bar{Y}_{i1}(s), 0 \leq s \leq t\right\},\\
	&\mathcal{F}_t^{\bar{Y}_{0}^{u_0}}=\sigma\left\{\bar{Y}_{0}^{u_0}(s), 0 \leq s \leq t\right\},\quad \mathcal{F}_t^{\bar{Y}_{i}^{u_0,u_i}}=\sigma\left\{\bar{Y}_{i}^{u_0,u_i}(s), 0 \leq s \leq t\right\}.
	\end{aligned}
	\end{equation}
	Then we define the corresponding decentralized control set of leader by
	\begin{equation}\label{leader admissible control}
		\mathcal{U}^d_0=\left\{u_0(t) \mid u_0(t) \in \mathcal{U}_0 \text{ is $\mathcal{F}_t^{\bar{Y}_{0}^{u_0}}$-adapted} \right\},
	\end{equation}
	where $\mathcal{U}_0=\left\{u_0(t) \mid u_0(t) \text{ is an $\mathcal{F}_t^{\bar{Y}_{01}}$-adapted process and $\mathbb{E}\left[\sup_{0 \leq t \leq T}|u_0(t)|^2\right]<\infty$ }\right\}$. Then the corresponding decentralized control set of followers is defined by
	\begin{equation}\label{follower admissible control}
		\mathcal{U}^d_i=\left\{u_i(t) \mid u_i(t) \in \mathcal{U}_i \text{ is $\mathcal{F}_t^{\bar{Y}_{i}^{u_0,u_i}}$-adapted} \right\},\ i=1,\dots,N
	\end{equation}
	where $\mathcal{U}_i=\left\{u_i(t) \mid u_i(t) \text{ is an $\mathcal{F}_t^{\bar{Y}_{i1}}$-adapted process and $\mathbb{E}\left[\sup_{0 \leq t \leq T}|u_i(t)|^2\right]<\infty$ }\right\}. $
	
	Subsequently, we propose the corresponding limiting large-population  Stackelberg stochastic differential game with partial observation.
	\begin{myprob}\label{problem decentralized}
		Finding a set of strategies $\left(u_0^*(\cdot), u_1^*(\cdot), \cdots, u_N^*(\cdot)\right)$ satisfying the following conditions:
		
		(i)  Given $u_0(\cdot) \in \mathcal{U}_{0}^{d}$, for i-th follower, finding the control strategy set $u_i^*(\cdot;u_0)\in\, \mathcal{U}_i^{d}$ such that
		\begin{equation*}
			J_i\left(u^*_i(\cdot;u_0),u_0(\cdot)\right)=\inf _{u_i(\cdot) \in\, \mathcal{U}_i^{d}} J_i\left(u_i(\cdot),u_0(\cdot)\right);
		\end{equation*}
		
		(ii) finding the control strategy  $u^*_0(\cdot ) \in\, \mathcal{U}_0^{d}$ of the leader $\mathcal{A}_0$ such that
		\begin{equation*}
			J_0\left(u^*_0(\cdot )\right)=\inf _{u_0(\cdot) \in\, \mathcal{U}_0^{d}} J_0\left(u_0(\cdot )\right).
		\end{equation*}
	\end{myprob}
	
	The above $\left(u_0^*(\cdot), u_1^*(\cdot), \cdots, u_N^*(\cdot)\right)$ is called the {\it decentralized Stackelberg equilibrium} of \textbf{Problem \ref{problem decentralized}}.
	
	Based on the definition of the admissible control set, we obtain the following equivalence result of filtrations.
	\begin{mylem}\label{follower lemma1}
		For any $u_i(\cdot) \in\, \mathcal{U}_i^{d}$, $i=1,\ldots,N$, $\mathcal{F}_t^{\bar{Y}_{i}^{u_0,u_i}}=\mathcal{F}_t^{\bar{Y}_{i1}}.$
	\end{mylem}
	
	\begin{proof}
		For any $u_i(\cdot) \in \mathcal{U}_i^d$, since $u_i(\cdot)$ is $\mathcal{F}_t^{\bar{Y}_{i1}}$-adapted, we know that $\bar{x}_{i2}^{u_0,u_i}(\cdot)$ is $\mathcal{F}_t^{\bar{Y}_{i1}}$-adapted by (\ref{follower limiting state decomposition 2}), thus $\bar{Y}_{i2}^{u_0,u_i}(\cdot)$ is also $\mathcal{F}_t^{\bar{Y}_{i1}}$-adapted by (\ref{follower limiting observation decomposition 2}). Then $\bar{Y}_{i}^{u_0,u_i}(\cdot)=\bar{Y}_{i1}(\cdot)+\bar{Y}_{i2}^{u_0,u_i}(\cdot)$ is $\mathcal{F}_t^{\bar{Y}_{i1}}$-adapted, which implies that $\mathcal{F}_t^{\bar{Y}_{i}^{u_0,u_i}} \subseteq \mathcal{F}_t^{\bar{Y}_{i1}}$. According to a similar argument, we obtain $\mathcal{F}_t^{\bar{Y}_{i1}} \subseteq \mathcal{F}_t^{\bar{Y}_{i}^{u_0,u_i}}$ via the equality $\bar{Y}_{i1}(\cdot)=\bar{Y}_{i}^{u_0,u_i}(\cdot)-\bar{Y}_{i2}^{u_0,u_i}(\cdot)$.
	\end{proof}
	
	Applying Lemma \ref{follower lemma1} and noticing the definition of  $\mathcal{U}_i^{d}$, we find that the filtration $\mathcal{F}_t^{\bar{Y}_{i1}}$ does not depend on control $u_i(\cdot)$ because $\bar{x}_{i1}(\cdot)$ and $\bar{Y}_{i1}(\cdot)$ are not related to $u_i(\cdot)$. Therefore, the admissible control set we have defined in (\ref{follower admissible control}) overcomes the difficulty of circular dependency. This approach has been widely used in classical literature, see \cite{Bensoussan-1992,Bensoussan-Feng-Huang-2021,Chen-Du-Wu-2024,Li-Li-Wu-2025,Wang-Wu-Xiong-2015}. Applying Lemma \ref{follower lemma1}, we have the following result.
	\begin{mylem}\label{follower lemma2}
		Let Assumption \ref{A1} hold. Given $u_0(\cdot) \in \mathcal{U}_{0}^{d}$, then we have
		\begin{equation*}
			\inf _{u_i(\cdot) \in\, \mathcal{U}_i^{d}} J_i\left(u_i(\cdot),u_0(\cdot)\right)=\inf _{u_i(\cdot) \in\, \mathcal{U}_i} J_i\left(u_i(\cdot),u_0(\cdot)\right).
		\end{equation*}
	\end{mylem}
	
	The proof can refer to the Lemma 2.3 of \cite{Wang-Wu-Xiong-2015}, so we omit it here. Similarly, Utilizing the definition of the leader's admissible control set, we have the following results.
	\begin{mylem}\label{leader lemma1}
		For any $u_0(\cdot) \in\, \mathcal{U}_0^{d}$, $\mathcal{F}_t^{\bar{Y}_{0}^{u_0}}=\mathcal{F}_t^{\bar{Y}_{01}}.$
	\end{mylem}
	
	\begin{mylem}\label{leader lemma2}
		Let Assumption \ref{A1} hold, then we have
		\begin{equation*}
			\inf _{u_0(\cdot) \in\, \mathcal{U}_0^{d}} J_0\left(u_0(\cdot)\right)=\inf _{u_0(\cdot) \in\, \mathcal{U}_0} J_0\left(u_0(\cdot)\right)\,.
		\end{equation*}
	\end{mylem}
	Below, we will first address the extreme followers' control problem using the backward separation principle, and then proceed to solve the leader's problem.
	
	\subsection{Open-loop decentralized strategies of the followers}
	
	In the remainder of this section, for brevity and without causing ambiguity, we omit the time dependency of some functions and stochastic processes, except for the terminal time.
	
	Given $u_0(\cdot) \in \mathcal{U}_{0}^{d}$ and $z(\cdot)\in L^2(0,T;\mathbb{R}^n)$, equation (\ref{leader limiting state}) admits a unique solution $\bar{x}_0^{u_0}(\cdot)\in \mathcal{F}^0_t$. By applying the standard variational method and dual technique, we can obtain the open-loop decentralized optimal strategies for the sub-problem (i) of \textbf{Problem \ref{problem decentralized}}.
	
	\begin{mythm}\label{open loop thm1 of follower}
		Let Assumptions \ref{A1} hold. For given $u_0(\cdot) \in \mathcal{U}_{0}^{d}$ and $z(\cdot)\in L^2(0,T;\mathbb{R}^n)$, then for $i=1, \cdots, N$, the open-loop decentralized optimal strategy of $i$-th follower $\mathcal{A}_i$ is given by
		\begin{equation}\label{follower open loop optimal control}
			u_i^*=-R^{-1}B^\top \mathbb{E}\Big[\left.p_i\right\rvert \mathcal{F}_t^{\bar{Y}^{u_0,u_i^*}_{i}}\Big]-R^{-1}S\mathbb{E}\Big[\left.{\bar{x}^{u_0,u^*_i}}_i\right\rvert \mathcal{F}_t^{\bar{Y}^{u_0,u^*_i}_{i}}\Big]+R^{-1}S\bar{\Gamma}_2\mathbb{E}\bar{x}^{u_0,u^*_i}_i,
		\end{equation}
		where $\Big(\bar{x}_i^{u_0,u_i^*}(\cdot), p_i(\cdot), q_i(\cdot),\bar{q}_i(\cdot), q_{i0}(\cdot),\bar{q}_{i0}(\cdot)\Big)\in L^2_{\mathcal{F}_t^i}(0,T;\mathbb{R}^n)\times L_{\widetilde{\mathcal{F}}_t^i}^2\left(0, T ; \mathbb{R}^n\right) \times L_{\widetilde{\mathcal{F}}_t^i}^2\left(0, T ; \mathbb{R}^{n\times n}\right) \times L_{\widetilde{\mathcal{F}}_t^i}^2\left(0, T ; \mathbb{R}^{n\times n}\right) \times L_{\widetilde{\mathcal{F}}_t^i}^2\left(0, T ; \mathbb{R}^{n\times n}\right)\times L_{\widetilde{\mathcal{F}}_t^i}^2\left(0, T ; \mathbb{R}^{n\times n}\right)$ satisfies the following stochastic Hamiltonian system:
		\begin{equation}\label{follower Hamiltonian system}
			\left\{\begin{aligned}
				d \bar{x}^{u_0,u_i^*}_i  &= \left[A \bar{x}^{u_0,u_i^*}_i+B u^*_i+\bar{A} \mathbb{E}\bar{x}^{u_0,u_i^*}_i+C z+F \mathbb{E}\bar{x}^{u_0}_0+b\right] d t +\sigma d W_i+\bar{\sigma} d \overline{W}_i \\
				d p_i&=-\left[A^\top p_i+\bar{A}^\top\mathbb{E}p_i +Q\bar{x}_i^{u_0,u_i^*}+\tilde{\Gamma}_2 \mathbb{E}\bar{x}^{u_0,u_i^*}_i +\left(\bar{\Gamma}_2^\top-I_n\right)Q\Gamma_2 z\right.\\
				&\qquad \left.+S^\top u_i^*-\bar{\Gamma}_2^\top S^\top\mathbb{E}u_i^* -Q\Gamma_3 \bar{x}_0^{u_0} +\tilde{\Gamma}_3 \mathbb{E}\bar{x}_0^{u_0}+\left(\bar{\Gamma}_2^\top-I_n\right)Q\eta_2\right] d t\\
                &\quad +q_i d W_i+\bar{q}_i d\overline{W}_i+q_{i0} d W_0+\bar{q}_{i0} d \overline{W}_0 ,\\
				\bar{x}_i^{u_0,u_i^*}(0)&=\xi,\quad p_i(T)=G\bar{x}_i^{u_0,u_i^*}(T)+\tilde{\Gamma}_4 \mathbb{E}\bar{x}^{u_0,u_i^*}_i(T) +\left(\bar{\Gamma}_4^\top-I_n\right)G\Gamma_4 z(T)\\
				&\quad -G\Gamma_5 \bar{x}_0^{u_0}(T) +\tilde{\Gamma}_5 \mathbb{E}\bar{x}_0^{u_0}(T)+\left(\bar{\Gamma}_4^\top-I_n\right)Q\eta_4,
			\end{aligned}\right.
		\end{equation}
		where
		$$
		\begin{aligned}
			\tilde{\Gamma}_2:=\bar{\Gamma}_2^\top Q\bar{\Gamma}_2-Q\bar{\Gamma}_2-\bar{\Gamma}_2^\top Q, \quad \tilde{\Gamma}_3:=\bar{\Gamma}_2^\top Q\left(\Gamma_3+\bar{\Gamma}_3\right)-Q\bar{\Gamma}_3, \\
			\tilde{\Gamma}_4:=\bar{\Gamma}_4^\top G\bar{\Gamma}_4-G\bar{\Gamma}_4-\bar{\Gamma}_4^\top G, \quad \tilde{\Gamma}_5:=\bar{\Gamma}_4^\top G\left(\Gamma_5+\bar{\Gamma}_5\right)-G\bar{\Gamma}_5,
		\end{aligned}
		$$
		and $\bar{Y}^{u_0,u_i^*}_i(\cdot)$ satisfies
		\begin{equation*}
			\left\{\begin{aligned}
				d \bar{Y}^{u_0,u_i^*}_i  &=\ \left[H \bar{x}^{u_0,u_i^*}_i+\bar{H} \mathbb{E}\bar{x}^{u_0,u_i^*}_i+I z+h\right] d t +f d \overline{W}_i , \\
				\bar{Y}^{u_0,u_i^*}_i(0) & =0\,.
			\end{aligned}\right.
		\end{equation*}
	\end{mythm}
	
The proof is postponed to Appendix.
	
	\begin{myremark}
		For simplicity of notation, in the following text, we denote $\hat{\theta}_i:=\mathbb{E}\Big[\left.\theta_i\right\rvert \mathcal{F}_t^{\bar{Y}^{u_0,u_i^*}_i}\Big]$ as the conditional expectation of an arbitrary stochastic process $\theta_i$ with respect to $\mathcal{F}_t^{\bar{Y}^{u_0,u_i^*}_i}$.
	\end{myremark}
	
	Next, we will study the unknown frozen limiting state-average $z(\cdot)$, which is a deterministic function. When $N \rightarrow \infty$, we would like to approximate $x_i^{u_0,u^*}(\cdot)$ by $\bar{x}_i^{u_0,u_i^*}(\cdot)$, thus $\frac{1}{N} \sum_{i=1}^N x^{u_0,u^*}_i(\cdot)$ is approximated by $\frac{1}{N} \sum_{i=1}^N \bar{x}_i^{u_0,u_i^*}(\cdot)$. Since the $\mathcal{F}_t^{W_i, \overline{W}_i}$-adapted process $\bar{x}_i^{u_0,u_i^*}(\cdot)$ and $\mathcal{F}_t^{W_j, \overline{W}_j}$-adapted process $\bar{x}_j^{u_0,u_j^*}(\cdot)$ (for $1\leq i\neq j\leq N$) are independent and identically distributed, we can apply the strong law of large numbers (see Majerek et al. \cite{Majerek-Nowak-Zieba-2005}) to draw a conclusion:
	\begin{equation}\label{SLLN}
		z(\cdot)=\lim _{N \rightarrow \infty} \frac{1}{N} \sum_{i=1}^N \bar{x}_i^{u_0,u_i^*}(\cdot)=\mathbb{E}\bar{x}_i^{u_0,u_i^*}(\cdot).
	\end{equation}
	Then we notice the second equation of (\ref{follower Hamiltonian system}) is coupled with the state of leader $\mathcal{A}_0$, then we arrive at the following {\it coupled control} (CC) system, which is a {\it conditional mean field FBSDE} (CMF-FBSDE) about $\big(\bar{x}_0^{u_0}(\cdot),\bar{x}_i^{u_0,u_i^*}(\cdot), p_i(\cdot), q_i(\cdot),\bar{q}_i(\cdot), q_{i0}(\cdot),\bar{q}_{i0}(\cdot)\big)\in L^2_{\mathcal{F}_t^0}(0,T;\mathbb{R}^n)\times L^2_{\mathcal{F}_t^{i}}(0,T;\mathbb{R}^n)\times L_{\widetilde{\mathcal{F}}_t^i}^2\left(0, T ; \mathbb{R}^n\right) \times L_{\widetilde{\mathcal{F}}_t^i}^2\left(0, T ; \mathbb{R}^{n\times n}\right) \times L_{\widetilde{\mathcal{F}}_t^i}^2\left(0, T ; \mathbb{R}^{n\times n}\right)\times L_{\widetilde{\mathcal{F}}_t^i}^2\left(0, T ; \mathbb{R}^{n\times n}\right)\times L_{\widetilde{\mathcal{F}}_t^i}^2\left(0, T ; \mathbb{R}^{n\times n}\right)$:
	\begin{equation}\label{CC system}
		\left\{\begin{aligned}
			d \bar{x}_0^{u_0} &= \left[A_0 \bar{x}_0^{u_0}+B_0 u_0+\bar{A}_0 \mathbb{E}\bar{x}_0^{u_0}+C_0 \mathbb{E}\bar{x}_i^{u_0,u_i^*}+b_0\right] d t +\sigma_0 d W_0+\bar{\sigma}_0 d \overline{W}_0, \\
			d \bar{x}^{u_0,u_i^*}_i &= \left[A \bar{x}^*_i-BR^{-1}B^\top \hat{p}_i-BR^{-1}S\hat{\bar{x}}_i^{u_0,u_i^*}+\left(\bar{A}+C+BR^{-1}S\bar{\Gamma}_2\right) \mathbb{E}\bar{x}_i^{u_0,u_i^*}\right.\\
			&\quad +F \mathbb{E}\bar{x}_0^{u_0}+b\big] d t+\sigma d W_i+\bar{\sigma} d \overline{W}_i, \\
			d p_i&=\left\{A^\top p_i-S^\top R^{-1}B^\top\hat{p}_i+\left(\bar{A}^\top+\bar{\Gamma}_2^\top S^\top R^{-1}B^\top\right)\mathbb{E}p_i +Q\bar{x}_i^{u_0,u_i^*} \right.\\
			&\quad -S^\top R^{-1}S\hat{\bar{x}}^{u_0,u_i^*}_i+\left[\tilde{\Gamma}_2+S^\top R^{-1}S\bar{\Gamma}_2+\bar{\Gamma}_2^\top S^\top R^{-1}S\left(I_n-\bar{\Gamma}_2\right)\right.\\
			&\left.\left.\quad +\left(\bar{\Gamma}_2^\top-I_n\right)Q\Gamma_2\right] \mathbb{E}\bar{x}_i^{u_0,u_i^*}-Q\Gamma_3 \bar{x}_0^{u_0} +\tilde{\Gamma}_3 \mathbb{E}\bar{x}_0^{u_0} +\left(\bar{\Gamma}_2^\top-I_n\right)Q\eta_2\right\} d t\\
            &\quad +q_i d W_i+\bar{q}_i d\overline{W}_i+q_{i0} d W_0+\bar{q}_{i0} d \overline{W}_0 ,\\
			\bar{x}_0^{u_0}(0)&= \xi_0,\quad \bar{x}_i^{u_0,u_i^*}(0)=\ \xi,\\
			p_i^*(T)&=G\bar{x}_i^{u_0,u_i^*}(T)+\left[\tilde{\Gamma}_4  +\left(\bar{\Gamma}_4^\top-I_n\right)G\Gamma_4\right] \mathbb{E}\bar{x}_i^{u_0,u_i^*}(T) -G\Gamma_5 \bar{x}_0^{u_0}(T) \\
			&\quad +\tilde{\Gamma}_5 \mathbb{E}\bar{x}_0^{u_0}(T)+\left(\bar{\Gamma}_4^\top-I_n\right)G\eta_4,
		\end{aligned}\right.
	\end{equation}
	where we replace $z(\cdot)$ with $\mathbb{E}\bar{x}_i^{u_0,u_i^*}(\cdot)$.
	Noticing that the open-loop decentralized optimal strategy (\ref{follower open loop optimal control}) of $i$-th follower $\mathcal{A}_i$ contains the optimal filtering estimate $\hat{p}_i$, we need to derive the filtering equation that $\hat{p}_i$ satisfies. However, $p_i$ is influenced by $\bar{x}_0^{u_0}$ and $\bar{x}^{u_0,u_i^*}_i$ in equation (\ref{CC system}), so we also need to obtain the filtering equation for $\bar{x}_0^{u_0}$ and $\bar{x}^{u_0,u_i^*}_i$. Then, we get the following filtering result (see Liptser and Shiryayev \cite{Liptser-Shiryayev-1977}).
	\begin{mypro}
		Let Assumption \ref{A1} hold. For any $u_0(\cdot) \in \mathcal{U}_0^d$, the optimal filtering of the solution to (\ref{CC system}) with respect to $\mathcal{F}_t^{\bar{Y}_i^{u_0,u_i^*}}$ satisfies
		\begin{equation}\label{filtering of CC system}
			\left\{\begin{aligned}
				d \mathbb{E}\bar{x}_0^{u_0}  &= \left[\left(A_0+\bar{A}_0\right) \mathbb{E}\bar{x}_0^{u_0}+B_0 \mathbb{E}u_0+C_0 \mathbb{E}\bar{x}_i^{u_0,u_i^*}+b_0\right] d t,  \\
				d \hat{\bar{x}}^{u_0,u_i^*}_i  &= \Big[\left(A-BR^{-1}S\right) \hat{\bar{x}}^{u_0,u_i^*}_i-BR^{-1}B^\top \hat{p}_i+\left(\bar{A}+C+BR^{-1}S\bar{\Gamma}_2\right) \mathbb{E}\bar{x}_i^{u_0,u_i^*}\\
				&\quad +F \mathbb{E}\bar{x}_0^{u_0}+b\Big] d t+\left(\bar{\sigma}+\Pi H^\top (f^{\top})^{-1}\right) d \widetilde{W}_i, \\
				d \hat{p}_i&=-\left\{\left(A^\top-S^\top R^{-1}B^\top\right) \hat{p}_i+\left(\bar{A}^\top+\bar{\Gamma}_2^\top S^\top R^{-1}B^\top\right)\mathbb{E}p_i  \right.\\
				&\qquad +\left(Q-S^\top R^{-1}S\right)\hat{\bar{x}}_i^{u_0,u_i^*}+\left[\tilde{\Gamma}_2+S^\top R^{-1}S\bar{\Gamma}_2+\bar{\Gamma}_2^\top S^\top R^{-1}S\left(I_n-\bar{\Gamma}_2\right)\right.  \\
				&\qquad +\left.\left.\left(\bar{\Gamma}_2^\top-I_n\right)Q\Gamma_2\right] \mathbb{E}\bar{x}_i^{u_0,u_i^*}+\left(\tilde{\Gamma}_3-Q\Gamma_3\right) \mathbb{E}\bar{x}_0^{u_0}+\left(\bar{\Gamma}_2^\top-I_n\right)Q\eta_2\right\} d t\\
                &\qquad +\tilde{q}_i d \widetilde{W}_i ,\\
				\mathbb{E}\bar{x}_0^{u_0}(0)&= \xi_0,\quad \hat{\bar{x}}_i^{u_0,u_i^*}(0)=\ \xi,\\
				\hat{p}_i(T)&=G\hat{\bar{x}}_i^{u_0,u_i^*}(T)+\left[\tilde{\Gamma}_4  +\left(\bar{\Gamma}_4^\top-I_n\right)G\Gamma_4\right] \mathbb{E}\bar{x}_i^{u_0,u_i^*}(T)  \\
				&\quad +\left(\tilde{\Gamma}_5-G\Gamma_5\right)\mathbb{E} \bar{x}_0^{u_0}(T)+\left(\bar{\Gamma}_4^\top-I_n\right)G\eta_4,
			\end{aligned}\right.
		\end{equation}
		where $\widetilde{W}_i$ is an $R^n$-valued standard Brownian motion satisfying
		\begin{equation}
			\widetilde{W}_i:=\int_0^t f^{-1} H\left(\bar{x}_i^{u_0,u_i^*}-\hat{\bar{x}}_i^{u_0,u_i^*}\right)ds+\overline{W}_i,
		\end{equation}
		and the mean square error $\Pi:=\mathbb{E}\left[\left(\bar{x}_i^{u_0,u_i^*}-\hat{\bar{x}}_i^{u_0,u_i^*}\right)\left(\bar{x}_i^{u_0,u_i^*}-\hat{\bar{x}}_i^{u_0,u_i^*}\right)^\top\right]$ is given by the following Riccati equation:
		\begin{equation}
			\left\{\begin{aligned}
				\dot{\Pi}&=A\Pi+\Pi A^\top -\left(\bar{\sigma}+\Pi H^\top (f^{\top})^{-1}\right)\left(\bar{\sigma}+\Pi H^\top (f^{\top})^{-1}\right)^\top+\sigma\sigma^\top+\bar{\sigma}\bar{\sigma}^\top ,\\
				\Pi(0)&=0.
			\end{aligned}\right.
		\end{equation}
	\end{mypro}
	
	\subsection{State feedback decentralized strategies of the followers}
	
	In this subsection, we derive the state feedback representation of the decentralized optimal strategies (\ref{follower open loop optimal control}) of the followers, through decoupling technique.
	
	Noting the terminal condition and structure of (\ref{CC system}), for each $i=1, \cdots, N$, we suppose
	\begin{equation}\label{follower decouple form}
		\hat{p}_i(t)=P_1(t)\left[\hat{\bar{x}}_i^{u_0,u_i^*}(t)-\mathbb{E}\bar{x}_i^{u_0,u_i^*}(t)\right]+P_2(t)\mathbb{E}\bar{x}_i^{u_0,u_i^*}(t)+P_3(t)\mathbb{E}\bar{x}^{u_0}_0(t) +\varphi^{u_0}(t),\quad t\in[0,T],
	\end{equation}
	with $P_1(T)=G$, $P_2(T)=G+\tilde{\Gamma}_4-\left(I_n+\bar{\Gamma}_4^\top\right)G\Gamma_4$, $P_3(T)=\tilde{\Gamma}_5-G\Gamma_5$, $\varphi^{u_0}(T)=\left(\bar{\Gamma}_4^\top-I_n\right)G\eta_4$ for four deterministic differentiable functions $P_1(\cdot),P_2(\cdot),P_3(\cdot),\varphi^{u_0}(\cdot)$. Applying It\^o's formula to (\ref{follower decouple form}), we achieve
	\begin{equation}\label{32}
		\begin{aligned}
			d\hat{p}_i
			=&-\left\{\left(A^\top-S^\top R^{-1}B^\top\right) \hat{p}_i+\left(\bar{A}^\top+\bar{\Gamma}_2^\top S^\top R^{-1}B^\top\right)\mathbb{E}p_i +\left(Q-S^\top R^{-1}S\right)\hat{\bar{x}}_i^* \right.\\
			&\quad +\left[\tilde{\Gamma}_2+
			S^\top R^{-1}S\bar{\Gamma}_2+\bar{\Gamma}_2^\top S^\top R^{-1}S\left(I_n-\bar{\Gamma}_2\right)+\left(\bar{\Gamma}_2^\top-I_n\right)Q\Gamma_2\right] \mathbb{E}\bar{x}_i^{u_0,u_i^*}  \\
			&\quad \left.+\left(\tilde{\Gamma}_3-Q\Gamma_3\right) \mathbb{E}\bar{x}_0^{u_0}+\left(\bar{\Gamma}_2^\top-I_n\right)Q\eta_2\right\} d t+\tilde{q}_i d \widetilde{W}_i ,
		\end{aligned}
	\end{equation}
	By comparing the coefficients of the diffusion terms, we get for $i=1, \cdots, N$,
	$$
	\tilde{q}_i=P_1\left[\bar{\sigma}+\Pi H^\top (f^{\top})^{-1}\right].
	$$
	By comparing the coefficients of the drift terms, we get for $i=1, \cdots, N$,
	$$
	\begin{aligned}
		&\left[\dot{P}_1+P_1A+A^\top P_1-\left(P_1B+S^\top\right)R^{-1}\left(B^\top P_1+S\right)+Q\right]\left(\hat{\bar{x}}_i^{u_0,u_i^*}-\mathbb{E}\bar{x}_i^{u_0,u_i^*}\right)\\
		&+\left[\dot{P}_2+P_2\left(A+\bar{A}+C\right)+\left(A+\bar{A}\right)^\top P_2 +P_3C_0+Q+\tilde{\Gamma}_2+\left(\bar{\Gamma}_2^\top-I_n\right)Q\Gamma_2 \right.\\
		&\left.-\left(P_2B+\left(I_n-\bar{\Gamma}_2^\top\right)S^\top\right)R^{-1}\left(B^\top P_2+S\left(I_n-\bar{\Gamma}_2\right)\right)\right]\mathbb{E}\bar{x}_i^{u_0,u_i^*}+\left[\dot{P}_3+P_3\left(A_0+\bar{A}_0\right)\right.\\
		&\left.+\left(A+\bar{A}-BR^{-1}S\left(I_n-\bar{\Gamma}_2\right)\right)^\top P_3 -P_2BR^{-1}B^\top P_3+P_2F+\tilde{\Gamma}_3-Q\Gamma_3\right]\mathbb{E}\bar{x}_0^{u_0}\\
		&+\dot{\varphi}^{u_0}+\left(A^\top+\bar{A}^\top-P_2BR^{-1}B^\top-\left(I_n-\bar{\Gamma}_2^\top\right)S^\top R^{-1}B^\top\right)\varphi^{u_0}+P_3B_0\mathbb{E}u_0+P_2b\\
		&+P_3b_0+\left(\bar{\Gamma}_2^\top-I_n\right)Q\eta_2=0.
	\end{aligned}
	$$
	
	Thus, we introduce the Riccati equations as follows:
	\begin{equation}\label{RE P1}
		\left\{\begin{aligned}
			&\dot{P}_1+P_1A+A^\top P_1-\left(P_1B+S^\top\right)R^{-1}\left(B^\top P_1+S\right)+Q=0, \\
			&P_1(T)=G\,,
		\end{aligned}\right.
	\end{equation}
	\begin{equation}\label{RE P2}
		\left\{\begin{aligned}
			\dot{P}_2&+P_2\left(A+\bar{A}+C-BR^{-1}S\left(I_n-\bar{\Gamma}_2\right)\right)+\left(A+\bar{A}-BR^{-1}S\left(I_n-\bar{\Gamma}_2\right)\right)^\top P_2+P_3C_0\\
			&-P_2BR^{-1}B^\top P_2+Q+\tilde{\Gamma}_2+\left(\bar{\Gamma}_2^\top-I_n\right)Q\Gamma_2-\left(I_n-\bar{\Gamma}_2^\top\right)S^\top R^{-1}S\left(I_n-\bar{\Gamma}_2\right)=0,\\
			P_2&(T)=G+\tilde{\Gamma}_4+\left(\bar{\Gamma}_4^\top-I_n\right)G\Gamma_4\,,
		\end{aligned}\right.
	\end{equation}
	\begin{equation}\label{RE P3}
		\left\{\begin{aligned}
			\dot{P}_3&+P_3\left(A_0+\bar{A}_0\right)+\left(A+\bar{A}-BR^{-1}S\left(I_n-\bar{\Gamma}_2\right)\right)^\top P_3 -P_2BR^{-1}B^\top P_3+P_2F\\
			&+\tilde{\Gamma}_3-Q\Gamma_3=0, \\
			 P_3&(T)=\tilde{\Gamma}_5-G\Gamma_5\,,
		\end{aligned}\right.
	\end{equation}
	and the equation of $\varphi^{u_0}(\cdot)\in L^2(0,T;\mathbb{R}^n)$ as
	\begin{equation}\label{varphi equation}
		\left\{\begin{aligned}
			\dot{\varphi}^{u_0}&+\left(A^\top+\bar{A}^\top-P_2BR^{-1}B^\top-\left(I_n-\bar{\Gamma}_2^\top\right)S^\top R^{-1}B^\top\right)\varphi^{u_0}+P_3B_0\mathbb{E}u_0+P_2b\\
			&+P_3b_0+\left(\bar{\Gamma}_2^\top-I_n\right)Q\eta_2=0,\\
			\varphi^{u_0}&(T)= \left(\bar{\Gamma}_4^\top-I_n\right)G\eta_4\,.
		\end{aligned}\right.
	\end{equation}
	
	\begin{myremark}
		(i) The equation \eqref{RE P1} is a standard Riccati equation, we can deduce that there exists a unique solution for it by the Theorem 7.2, Chapter 6 of Yong and Zhou \cite{Yong-Zhou-1999}.
		
		(ii) Noting the equations \eqref{RE P2} and \eqref{RE P3} are fully coupled {\it ordinary differential equations} (ODEs), the solvability of this equation is not immediately obvious and has not been studied in existing literature. We have given the solvability of the coupled Riccati equations under certain conditions in the following Proposition \ref{RE solvability}.
		
		(iii) If the equations \eqref{RE P1}, \eqref{RE P2} and \eqref{RE P3} admit solutions, then the equation \eqref{varphi equation} admits unique solution by the classical theory of linear ODE.
	\end{myremark}
	
	Next, we will give conditions under which equations (\ref{RE P2}) and (\ref{RE P3}) are solvable. Let $P:=[P_2, P_3]$, then the equations (\ref{RE P2}) and (\ref{RE P3}) can be rewritten as follows:
	\begin{equation}
		\left\{\begin{aligned}
			\dot{P}&+PH_1+H_2 P -PI_1 P+Q_1=0,\\
			P&(T)=G_1\,,
		\end{aligned}\right.
	\end{equation}
	where 
	\begin{equation*}
    \begin{aligned}
	H_1&:=\left(\begin{array}{cc}
		A+\bar{A}-BR^{-1}S\left(I_n-\bar{\Gamma}_2\right) & F  \\
		C_0 & A_0+\bar{A}_0
	\end{array}\right),\quad H_2:=\left(A+\bar{A}-BR^{-1}S\left(I_n-\bar{\Gamma}_2\right)\right)^\top, \\ 
	Q_1&:=\left(\begin{array}{cc}
		Q+\tilde{\Gamma}_2+\left(\bar{\Gamma}_2^\top-I_n\right)Q\Gamma_2-\left(I_n-\bar{\Gamma}_2^\top\right)S^\top R^{-1}S\left(I_n-\bar{\Gamma}_2\right)  &
		\tilde{\Gamma}_3-Q\Gamma_3
	\end{array}\right),\\
	I_1&:=\left(\begin{array}{c}
		BR^{-1}B^\top \\  
		0
	\end{array}\right),\quad G_1:=\left(\begin{array}{cc}
		G+\tilde{\Gamma}_4+\left(\bar{\Gamma}_4^\top-I_n\right)G\Gamma_4 & \tilde{\Gamma}_5-G\Gamma_5
	\end{array}\right).
	\end{aligned}
    \end{equation*}
	Then, we introduce the following linear equation
	$
	\begin{pmatrix}
		\dot{u} \\
		\dot{v}
	\end{pmatrix}
	=H\begin{pmatrix}
		u \\
		v
	\end{pmatrix}dt,
	\begin{pmatrix}
		u \\
		v
	\end{pmatrix}(T)=\begin{pmatrix}
		I_n \\
		I_n
	\end{pmatrix},
	$
	where $H:=\begin{pmatrix}
		H_1 & -I_1  \\
		-Q_1 & -H_2
	\end{pmatrix}$. Let $\Theta(\cdot)$ be the fundamental solution matrix of above linear ODE, then we have the following result by the result of Page 11 in Reid \cite{Reid-1972}.
	\begin{mypro}\label{RE solvability}
		If $
		\begin{pmatrix}
			I_n & 0
		\end{pmatrix} \Theta(t) \Theta^{-1}(T)\begin{pmatrix}
			I_n \\
			G_1
		\end{pmatrix}
		$ is non-singular, then the equations (\ref{RE P2}) and (\ref{RE P3}) have a solution in $C([0, T ], \mathbb{R}^{n\times n})$.
	\end{mypro}
	
	Then, the state feedback representation of the decentralized optimal strategies of the followers can be obtained in the following theorem.
	\begin{mythm}\label{theorem feedback of follower}
		Let Assumption \ref{A1} and the assumption of Proposition \ref{RE solvability} hold, for given $u_0(\cdot) \in \mathcal{U}_{0}^{d}$, the state feedback optimal strategies of the followers $\mathcal{A}_i,i=1,\cdots,N$, can be represented as
		\begin{equation}\label{follower feedback strategy}
			\begin{aligned}
				u_i^*=&-R^{-1}\Big[(B^\top P_1+S)\hat{\bar{x}}_i^{u_0,u_i^*}+(B^\top\left(P_2-P_1\right)-S\bar{\Gamma}_2)\mathbb{E}\bar{x}_i^{u_0,u_i^*}+B^\top P_3\mathbb{E}\bar{x}_0^{u_0} +B^\top \varphi^{u_0} \big],
			\end{aligned}
		\end{equation}
		where $\left(\hat{\bar{x}}_i^{u_0,u_i^*}(\cdot),\mathbb{E}\bar{x}_i^{u_0,u_i^*}(\cdot),\mathbb{E}\bar{x}_0^{u_0}(\cdot)\right)\in L^2_{\mathcal{F}_t^{\bar{Y}^{u_0,u_i^*}_i}}(0,T;\mathbb{R})\times L^2(0,T;\mathbb{R}^n)\times L^2(0,T;\mathbb{R}^n)$ satisfies
		\begin{equation}\label{hat bar x_i}
			\left\{\begin{aligned}
				d \hat{\bar{x}}^{u_0,u_i^*}_i &= \left\{\left(A-BR^{-1}(B^\top P_1+S)\right)\hat{\bar{x}}^{u_0,u_i^*}_i+\left(F-BR^{-1}B^\top P_3\right) \mathbb{E}\bar{x}_0^{u_0} \right.\\
				&\qquad -BR^{-1}B^\top\varphi^{u_0}+\big[\bar{A}+C-BR^{-1}\big(B^\top\left(P_2-P_1\right)-S\bar{\Gamma}_2\big)\big]\mathbb{E}\bar{x}_i^{u_0,u_i^*}\\
                &\qquad +b\Big\} d t +\left(\bar{\sigma}+\Pi H^\top (f^{\top})^{-1}\right) d \widetilde{W}_i, \\
				\hat{\bar{x}}_i^{u_0,u_i^*} (0)&= \xi,
			\end{aligned}\right.
		\end{equation}
		\begin{equation}\label{E bar x_i}
			\left\{\begin{aligned}
				d \mathbb{E}\bar{x}_i^{u_0,u_i^*}  &= \left\{\left(A+\bar{A}+C-BR^{-1}B^\top P_2-BR^{-1}S\left(I_n-\bar{\Gamma}_2\right)\right)\mathbb{E}\bar{x}_i^{u_0,u_i^*} \right.\\
				&\quad \left.+\left(F-BR^{-1}B^\top P_3\right) \mathbb{E}\bar{x}_0^{u_0}-BR^{-1}B^\top\varphi^{u_0}+b\right\} d t, \\
				\mathbb{E}\bar{x}_i^{u_0,u_i^*}(0)& = \xi,
			\end{aligned}\right.
		\end{equation}
		\begin{equation}\label{E x_0}
			\left\{\begin{aligned}
				d \mathbb{E}\bar{x}_0^{u_0}  &= \left[\left(A_0+\bar{A}_0\right)\mathbb{E}\bar{x}_0^{u_0}+B_0\mathbb{E}u_0+C_0\mathbb{E}\bar{x}_i^{u_0,u_i^*}+b_0\right]dt \\
				\mathbb{E}\bar{x}_0^{u_0} (0) &= \xi_0,
			\end{aligned}\right.
		\end{equation}
		and $\mathbb{R}^n$-valued function $\varphi^{u_0}(\cdot)$ satisfies (\ref{varphi equation}).
	\end{mythm}
	
	\begin{proof}
		Noticing that (\ref{E bar x_i}) and (\ref{E x_0}) are coupled ODEs, we apply dimension expansion technique (see Yong \cite{Yong-2002}).
		Let $X:=\left(\mathbb{E}\bar{x}_0^{{u_0}\top},\mathbb{E}\bar{x}_i^{{u_0,u_i^*}\top}\right)^\top$,  $b_1:=\left(\begin{array}{c}
			B_0\mathbb{E}u_0+b_0  \\
			-BR^{-1}B^\top \varphi^{u_0}+b
		\end{array}\right), $ $\xi_1:=\left(\xi_0^\top,\xi^\top\right)^\top,$ 
		$
		A_1:=\left(\begin{array}{cc}
			A_0+\bar{A}_0 & C_0  \\
			F-BR^{-1}B^\top P_3 & A+\bar{A}+C-BR^{-1}B^\top P_2-BR^{-1}S\left(I_n-\bar{\Gamma}_2\right)
		\end{array}\right),
		$ we have $d X  = \left[A_1X+b_1\right]dt$, $X (0)= \xi_1.$
		Then, the equation that $X$ satisfies admits unique solution by using classical linear ODE theory and equation (\ref{hat bar x_i}) admits unique solution  by Proposition 2.6 of Yong \cite{Yong-2013}. The feedback form (\ref{follower feedback strategy}) of optimal strategy can be obtained from (\ref{follower open loop optimal control}) and (\ref{follower decouple form}). The proof is complete.
	\end{proof}
	
	To conclude this subsection, we discuss the well-posedness of the CC system (\ref{CC system}). Similar as Li et al. {\cite{Li-Nie-Wu-2023}}, we employ the Riccati equation method.
	
	Firstly, we consider the equivalence between the solvability of (\ref{CC system}) and the solvability of (\ref{filtering of CC system}). On the one hand, if (\ref{CC system}) admit a solution $(\bar{x}_0^{u_0},\bar{x}_i^{u_0,u_i^*},p_i)$, then taking filtering to $(\bar{x}_0^{u_0},\bar{x}_i^{u_0,u_i^*},p_i)$ with respect to $\mathcal{F}^{\bar{Y}^{u_0,u_i^*}_i}$ will yield a solution $(\mathbb{E}\bar{x}_0^{u_0},\hat{\bar{x}}_i^{u_0,u_i^*},\hat{p}_i)$ of (\ref{filtering of CC system}). On the other hand, if (\ref{filtering of CC system}) admit a solution $(\mathbb{E}\bar{x}_0^{u_0},\hat{\bar{x}}_i^{u_0,u_i^*},\hat{p}_i)$, then we can get a solution $(\bar{x}_0^{u_0},\bar{x}_i^{u_0,u_i^*})$ of the first and second equation of (\ref{CC system}) by Proposition 2.6 of Yong \cite{Yong-2013}. And then, the third equation of (\ref{CC system}) admit a solution by Theorem 2.1 of Li et al. \cite{Li-Sun-Xiong-2019}. 
	
	Under the assumptions of Theorem \ref{theorem feedback of follower}, the equations (\ref{hat bar x_i}) and (\ref{E x_0})  admit the solutions. Then, the third equation of (\ref{filtering of CC system}) admit a solution by Theorem 2.1 of \cite{Li-Sun-Xiong-2019}. In other word, under the assumptions of Theorem \ref{theorem feedback of follower}, (\ref{filtering of CC system}) admit a solution. Equivalently, we have (\ref{CC system}) admit a solution.
	
	\subsection{Open loop decentralized strategy of the leader}

	Noting that the followers take their optimal strategies $u^*(\cdot;u_0)\equiv\left(u_1^*(\cdot;u_0), \ldots, u_N^*(\cdot;u_0)\right)$, the decentralized state equation of the leader $\mathcal{A}_0$ becomes
	\begin{equation}\label{leader state decentralized new }
		\left\{\begin{aligned}
			d \bar{x}_0^{u_0} &= \left(A_0 \bar{x}_0^{u_0}+B_0 u_0+\bar{A}_0\mathbb{E}\bar{x}_0+C_0 \mathbb{E}\bar{x}_i^{u_0,u_i^*}+b_0\right) d t +\sigma_0 d W_0+\bar{\sigma}_0 d \overline{W}_0, \\
			d \mathbb{E}\bar{x}_i^{u_0,u_i^*}&= \left[\left(A+\bar{A}+C-BR^{-1}B^\top P_2-BR^{-1}S\left(I_n-\bar{\Gamma}_2\right)\right)\mathbb{E}\bar{x}_i^{u_0,u_i^*} \right.\\
			&\quad \left.+\left(F-BR^{-1}B^\top P_3\right) \mathbb{E}\bar{x}_0^{u_0}-BR^{-1}B^\top\varphi^{u_0}+b\right] d t,\\
			d\varphi^{u_0}&=-\left[\left(A+\bar{A}-BR^{-1}B P_2-B R^{-1}S\left(I_n-\bar{\Gamma}_2\right)\right)^\top\varphi^{u_0}+P_3B_0\mathbb{E}u_0\right.\\
			&\quad \left.+P_2b+P_3b_0+\left(\bar{\Gamma}_2^\top-I_n\right)Q\eta_2\right]dt,\\
			\bar{x}^{u_0}_0(0)&= \xi_0, \quad \mathbb{E}\bar{x}_i^{u_0,u_i^*}(0)=\xi,\quad \varphi^{u_0}(T)=\left(\bar{\Gamma}_4^\top-I_n\right)G\eta_4\,.
		\end{aligned}\right.
	\end{equation}
	The corresponding observation process satisfies
	\begin{equation}\label{leader limiting observation 1}
		\left\{\begin{aligned}
			d \bar{Y}_0^{u_0}  &=\left[H_0 \bar{x}_0^{u_0}+\bar{H}_0 \mathbb{E}\bar{x}_0^{u_0}+I_0 \mathbb{E}\bar{x}_i^{u_0,u_i^*}+h_0\right] d t +f_0 d \overline{W}_0 , \\
			\bar{Y}^{u_0}_0(0) &=0,
		\end{aligned}\right.
	\end{equation}
	and the decentralized cost functional is
	\begin{equation}\label{leader decentralized cost functional new}
		\begin{aligned}
			J_0&\left(u_0(\cdot)\right)=\frac{1}{2} \mathbb{E} \bigg\{\int_0^T \left(\|\bar{x}_0^{u_0}-\Gamma_0  \mathbb{E}\bar{x}_i^{u_0,u_i^*}-\bar{\Gamma}_0  \mathbb{E}\bar{x}_0^{u_0}-\eta_0\|_{Q_0}^2 +\|u_0\|^2_{R_0}+2\left\langle S_0(\bar{x}_0^{u_0}\right.\right.\\
			&\left.\left.-\Gamma_0\mathbb{E}\bar{x}_i^{u_0,u_i^*}-\bar{\Gamma}_0\mathbb{E}\bar{x}_0^{u_0}),u_0\right\rangle \right) dt+ \|\bar{x}_0^{u_0}(T)-\Gamma_1 \mathbb{E}\bar{x}_i^{u_0,u_i^*}(T)-\bar{\Gamma}_1  \mathbb{E}\bar{x}_0^{u_0}(T)-\eta_1\|_{G_0}^2 \bigg\}\,.
		\end{aligned}
	\end{equation}
	
	By applying the standard variational method and dual technique, we can obtain the open-loop decentralized optimal strategies of the leader $\mathcal{A}_0$ for sub-problem (ii) of \textbf{Problem \ref{problem decentralized}}.
	
	\begin{mythm}\label{open loop thm1 of leader}
		Let Assumption \ref{A1} and the assumption of Proposition \ref{RE solvability} hold. Then the open-loop decentralized strategy of the leader $\mathcal{A}_0$ is given by
		\begin{equation}\label{leader open loop optimal strategy}
			u_0^*=-R_0^{-1}B_0^\top \mathbb{E}\left[\left. y_0\right\rvert \mathcal{F}^{\bar{Y}^{u_0^*}}_t\right]+R_0^{-1}B_0^\top P_3^\top \gamma-R_0^{-1}S_0\left(\mathbb{E}\left[\left. \bar{x}_0^{u_0^*}\right\rvert \mathcal{F}^{\bar{Y}^{u_0^*}_0}_t\right]-\Gamma_0\mathbb{E}\bar{x}_i^{u_0^*,u_i^*}-\bar{\Gamma}_0\mathbb{E}\bar{x}_0^{u_0^*}\right),
		\end{equation}
		where $\left(\bar{x}_0^{u_0^*}(\cdot),\mathbb{E}\bar{x}_i^{u_0^*,u_i^*}(\cdot),\varphi^{u_0^*}(\cdot),y_0(\cdot),z_0(\cdot),\bar{z}_0(\cdot),y(\cdot),\gamma(\cdot)\right)$ satisfies the following stochastic Hamiltonian system:
		\begin{equation}\label{leader Hamiltonian system}
			\left\{\begin{aligned}
				d \bar{x}^{u_0^*}_0= & \left(A_0 \bar{x}^{u_0^*}_0+B_0 u^*_0+\bar{A}_0\mathbb{E}\bar{x}^{u_0^*}_0+C_0 \mathbb{E}\bar{x}_i^{u_0^*,u_i^*}+b_0\right) d t +\sigma_0 d W_0+\bar{\sigma}_0 d \overline{W}_0, \\
				d \mathbb{E}\bar{x}_i^{u_0^*,u_i^*}=& \left[\left(A+\bar{A}+C-BR^{-1}B^\top P_2-BR^{-1}S\left(I_n-\bar{\Gamma}_2\right)\right)\mathbb{E}\bar{x}_i^{u_0^*,u_i^*} \right.\\
				&\quad\left.+\left(F-BR^{-1}B^\top P_3\right) \mathbb{E}\bar{x}_0^{u_0^*}-BR^{-1}B^\top\varphi^{u_0^*}+b\right] d t,\\
				d\varphi^{u_0^*}=&-\left[\left(A+\bar{A}-BR^{-1}B P_2-B R^{-1}S\left(I_n-\bar{\Gamma}_2\right)\right)^\top\varphi^{u_0^*}+P_3B_0\mathbb{E}u_0^*+P_2b\right.\\
				&\quad\left.+P_3b_0+\left(\bar{\Gamma}_2^\top-I_n\right)Q\eta_2\right]dt,\\
				dy_0=&-\left\{A_0^\top y_0+\bar{A}_0^\top\mathbb{E}y_0+\left(F^\top-P_3^\top BR^{-1}B^\top\right)y+Q_0\bar{x}_0^{u_0^*}+\tilde{\Gamma}_0\mathbb{E}\bar{x}_0^{u_0^*} \right.\\
				&\quad\left.+\left(\bar{\Gamma}_0^\top-I_n\right)Q_0\Gamma_0\mathbb{E}\bar{x}_i^{u_0^*,u_i^*}+S_0^\top u^*_0-\bar{\Gamma}_0^\top S_0^\top\mathbb{E}u^*_0+\left(\bar{\Gamma}_0^\top-I_n\right)Q_0\eta_0\right\}dt\\
				&\quad +z_0dW_0+\bar{z}_0d\overline{W}_0,\\
				dy=&-\left\{\left(A+\bar{A}+C-BR^{-1}B P_2-BR^{-1}S\left(I_n-\bar{\Gamma}_2\right)\right)^\top y-\Gamma_0^\top S_0\top\mathbb{E}u_0^* \right.\\
				&\quad\left.+ C_0^\top \mathbb{E}y_0+\Gamma_0^\top Q_0\left(\bar{\Gamma}_0-I_n\right)\mathbb{E}\bar{x}_0^{u_0^*}+\Gamma_0^\top Q\Gamma_0\mathbb{E}\bar{x}_i^{u_0^*,u_i^*}+\Gamma_0^\top Q_0\eta_0\right\}dt,\\
				d\gamma=&\left\{\left(A+\bar{A}-BR^{-1}B P_2-B R^{-1}S\left(I_n-\bar{\Gamma}_2\right)\right)\gamma+BR^{-1}B^\top y\right\}dt,\\
				\bar{x}^{u_0^*}_0(0) =& \xi_0, \quad \mathbb{E}\bar{x}_i^{u_0^*,u_i^*}(0)=\xi,\quad \varphi^{u_0^*}(T)=\left(\bar{\Gamma}_4^\top-I_n\right)G\eta_4,\quad \gamma(0)=0,\\
				y_0(T)=&G_0\bar{x}^{u_0^*}_0(T)+\tilde{\Gamma}_1\mathbb{E}\bar{x}_0^{u_0^*}(T)+\left(\bar{\Gamma}_1-I_n\right)G_0\Gamma_1\mathbb{E}\bar{x}_i^{u_0^*,u_i^*}(T)+\left(\bar{\Gamma}_1-I_n\right)G_0\eta_1,\\
				y(T)=&\Gamma_1^\top G_0\left(\bar{\Gamma}_1-I_n\right)\mathbb{E}\bar{x}_0^{u_0^*}(T)+\Gamma_1^\top G_0\Gamma_1 \mathbb{E}\bar{x}_i^{u_0^*,u_i^*}(T)+\Gamma_1^\top G_0\eta_1,
			\end{aligned}\right.
		\end{equation}
		where
		$$
		\begin{aligned}
			\tilde{\Gamma}_0:=\bar{\Gamma}_0^\top Q_0\bar{\Gamma}_0-Q_0\bar{\Gamma}_0-\bar{\Gamma}_0^\top Q_0, \quad \tilde{\Gamma}_1:=\bar{\Gamma}_1^\top G_0\bar{\Gamma}_1-G_0\bar{\Gamma}_1-\bar{\Gamma}_1^\top G_0,
		\end{aligned}
		$$
		and $\bar{Y}^{u_0^*}_0(\cdot)$ satisfies
		\begin{equation*}
			\left\{\begin{aligned}
				d \bar{Y}^{u_0^*}_0  &=\left[H_0 \bar{x}^{u_0^*}_0+\bar{H}_0 \mathbb{E}\bar{x}^{u_0^*}_0+I_0 \mathbb{E}\bar{x}_i^{u_0^*,u_i^*}+h_0\right] d t +f_0 d \overline{W}_0 , \\
				\bar{Y}^{u_0^*}_0(0) &=0.
			\end{aligned}\right.
		\end{equation*}
	\end{mythm}

    The proof is left to the Appendix. 
	
	\begin{myremark}
		For simplicity of notation, in the following text, we denote $\check{\theta}_0:=\mathbb{E}\Big[\left.\theta_0\right\rvert \mathcal{F}_t^{\bar{Y}^{u_0^*}_{0}}\Big]$ as the conditional expectation of an arbitrary stochastic process $\theta_0$ with respect to $\mathcal{F}_t^{\bar{Y}^{u_0^*}_{0}}$.
	\end{myremark}
	
	Noticing that the open-loop decentralized optimal strategy (\ref{leader open loop optimal strategy}) of leader $\mathcal{A}_0$ contains the optimal filtering estimate $\check{y}_0$, we need to derive the filtering equation that $\check{y}_0$ satisfies. However, $y_0$ is influenced by $\bar{x}^{u_0^*}_0$ in equation (\ref{leader Hamiltonian system}), so we also need to obtain the filtering equation for $\bar{x}_0^{u_0^*}$. Then, we get the following filtering result (\cite{Liptser-Shiryayev-1977}).
	\begin{mypro}
		Let Assumption \ref{A1} and the assumption of Proposition \ref{RE solvability} hold. The optimal filtering of the solution $(\bar{x}^{u_0^*}_0(\cdot), \mathbb{E}\bar{x}_i^{u_0^*,u_i^*}(\cdot), \varphi^{u_0^*}(\cdot), y_0(\cdot), y(\cdot), \gamma(\cdot))$ to (\ref{leader Hamiltonian system}) with respect to $\mathcal{F}_t^{\bar{Y}_0^{u_0^*}}$ satisfies
		\begin{equation}\label{filtering of leader hamiltionian}
			\left\{\begin{aligned}
				d \check{\bar{x}}^{u_0^*}_0&=\left(A_0 \check{\bar{x}}^{u_0^*}_0+B_0 u^*_0+\bar{A}_0\mathbb{E}\bar{x}^{u_0^*}_0+C_0 \mathbb{E}\bar{x}_i^{u_0^*,u_i^*}+b_0\right) d t
                  +\left(\bar{\sigma}_0+\Pi_0H_0^\top f_0^{\top-1}\right) d \widetilde{W}_0, \\
				d \mathbb{E}\bar{x}_i^{u_0^*,u_i^*}&= \left[\left(A+\bar{A}+C-BR^{-1}B^\top P_2-BR^{-1}S\left(I_n-\bar{\Gamma}_2\right)\right)\mathbb{E}\bar{x}_i^{u_0^*,u_i^*}\right.\\
				&\quad +\left.\left(F-BR^{-1}B^\top P_3\right) \mathbb{E}\bar{x}_0^{u_0^*}-BR^{-1}B^\top\varphi^{u_0^*}+b\right] d t,\\
				d\varphi^{u_0^*}&= -\left[\left(A+\bar{A}-BR^{-1}B P_2-BR^{-1}S\left(I_n-\bar{\Gamma}_2\right)\right)^\top\varphi^{u_0^*}+P_3B_0\mathbb{E}u_0^*+P_2b\right.\\
				&\qquad +\left.P_3b_0+\left(\bar{\Gamma}_2^\top-I_n\right)Q\eta_2\right]dt,\\
				d\check{y}_0&= -\left\{A_0^\top \check{y}_0+\bar{A}_0^\top\mathbb{E}y_0+\left(F^\top-P_3^\top BR^{-1}B^\top\right)y+Q_0\check{\bar{x}}_0^{u_0^*}+\tilde{\Gamma}_0\mathbb{E}\bar{x}_0^{u_0^*}+S_0^\top u^*_0\right.\\
				&\qquad -\left.\bar{\Gamma}_0^\top S_0^\top\mathbb{E}u^*_0 +\left(\bar{\Gamma}_0^\top-I_n\right)Q_0\Gamma_0\mathbb{E}\bar{x}_i^{u_0^*,u_i^*}+\left(\bar{\Gamma}_0^\top-I_n\right)Q_0\eta_0\right\}dt+\tilde{z}_0d\widetilde{W}_0,\\
				dy=& -\left\{\left(A+\bar{A}+C-BR^{-1}B P_2-BR^{-1}S\left(I_n-\bar{\Gamma}_2\right)\right)^\top y+ C_0^\top \mathbb{E}y_0 \right.\\
				&\quad -\left.\Gamma_0^\top S_0^\top\mathbb{E}u^*_0+\Gamma_0^\top Q_0\left(\bar{\Gamma}_0-I_n\right)\mathbb{E}\bar{x}_0^{u_0^*}+\Gamma_0^\top Q\Gamma_0\mathbb{E}\bar{x}_i^{u_0^*,u_i^*}+\Gamma_0^\top Q_0\eta_0\right\}dt,\\
				d\gamma&= \left\{\left(A+\bar{A}-BR^{-1}B P_2-BR^{-1}S\left(I_n-\bar{\Gamma}_2\right)\right)\gamma+BR^{-1}B^\top y\right\}dt,\\
				\check{\bar{x}}^{u_0^*}_0(0)&= \xi_0, \quad \mathbb{E}\bar{x}_i^{u_0^*,u_i^*}(0)=\xi,\quad \varphi^{u_0^*}(T)=\left(\bar{\Gamma}_4^\top-I_n\right)G\eta_4,\quad \gamma(0)=0,\\
				\check{y}_0(T)&= G_0\check{\bar{x}}^*_0(T)+\tilde{\Gamma}_1\mathbb{E}\bar{x}_0^{u_0^*}(T)+\left(\bar{\Gamma}_1-I_n\right)G_0\Gamma_1\mathbb{E}\bar{x}_i^{u_0^*,u_i^*}(T)+\left(\bar{\Gamma}_1-I_n\right)G_0\eta_1,\\
				y(T)&=\Gamma_1^\top G_0\left(\bar{\Gamma}_1-I_n\right)\mathbb{E}\bar{x}_0^{u_0^*}(T)+\Gamma_1^\top G_0\Gamma_1 \mathbb{E}\bar{x}_i^{u_0^*,u_i^*}(T)+\Gamma_1^\top G_0\eta_1,
			\end{aligned}\right.
		\end{equation}
		where $\widetilde{W}_0$ is an $R^n$-valued standard Brownian motion satisfying
		\begin{equation}
			\widetilde{W}_0:=\int_0^t f_0^{-1} H_0\left(\bar{x}_0^{u_0^*}-\check{\bar{x}}_0^{u_0^*}\right)ds+\overline{W}_0,
		\end{equation}
		and the mean square error $\Pi_0:=\mathbb{E}\left[\left(\bar{x}_0^{u_0^*}-\hat{\bar{x}}_0^{u_0^*}\right)\left(\bar{x}_0^{u_0^*}-\hat{\bar{x}}_0^{u_0^*}\right)^\top\right]$ is given by the following Riccati equation:
		\begin{equation}
			\left\{\begin{aligned}
				\dot{\Pi}_0&= A_0\Pi_0+\Pi_0 A_0^\top -\left(\bar{\sigma}_0+\Pi_0 H_0^\top (f_0^{\top})^{-1}\right)\left(\bar{\sigma}_0+\Pi_0 H_0^\top(f_0^\top)^{-1}\right)^\top\\
                &\quad +\sigma_0\sigma_0^\top+\bar{\sigma}_0\bar{\sigma}_0^\top,\\
				\Pi_0(0)&=0\,.
			\end{aligned}\right.
		\end{equation}
	\end{mypro}
	
	\subsection{Feedback decentralized strategy of the leader}
	
	Now, we derive the state feedback form of the open loop strategy of the leader $\mathcal{A}_0$, by the dimension expansion technique of Yong \cite{Yong-2002}. Let
	$$
	\mathbb{X}_0:=\left[\begin{array}{c}
		\bar{x}_0^{u_0^*} \\
		\mathbb{E}\bar{x}_i^{u_0^*,u_i^*} \\
		\gamma
	\end{array}\right],\quad \Phi:=\left[\begin{array}{c}
		y_0 \\
		y \\
		\varphi^{u_0^*}
	\end{array}\right],\quad \widetilde{\mathbb{Z}}_0:=\left[\begin{array}{c}
		\tilde{z}_0 \\
		0 \\
		0
	\end{array}\right].
	$$
	Then, the Hamiltonian system (\ref{leader Hamiltonian system}) of the leader $\mathcal{A}_0$ can be rewritten as
	\begin{equation}\label{dimension expansion}
		\left\{\begin{aligned}
			d \check{\mathbb{X}}_0= & \left(\mathbb{A}_{0} \check{\mathbb{X}}_0+\bar{\mathbb{A}}_{0} \mathbb{E} \mathbb{X}_0+\mathbb{C}_{0} \mathbb{E} \Phi
			+\mathbb{B}_0u_0^*+\mathbb{D}_0 \right) d t+\tilde{\sigma}_0d\widetilde{W}_0, \\
			d \check{\Phi}= & -\left(\mathbb{A}_0^\top \check{\Phi}+\bar{\mathbb{A}}^\top_0 \mathbb{E} \Phi+\mathbb{Q}_0\check{\mathbb{X}}_0
			+\bar{\mathbb{Q}}_0 \mathbb{E} \mathbb{X}_0+\mathbb{S}_0u_0^*+\bar{\mathbb{B}}_0\mathbb{E}u^*_0+\bar{\mathbb{D}}_0\right) d t+\widetilde{\mathbb{Z}}_0 d \widetilde{W}_0 ,\\
			\check{\mathbb{X}}_0(0)= &\ \Xi_0,\quad \check{\Phi}(T)=\mathbb{G}_0 \check{\mathbb{X}}_0(T)+\mathbb{G}_1\mathbb{E}\mathbb{X}_0(T)+\mathbb{G}_2,
		\end{aligned}\right.
	\end{equation}
	where we have denoted
	$$
	\begin{aligned}
		& \mathbb{A}_{0}:=\left[\begin{array}{ccc}
			A_0 & 0 & 0 \\
			0 & 0 & 0 \\
			0 & 0 & 0
		\end{array}\right], \mathbb{C}_0=\left[\begin{array}{ccc}
		0 & 0 & 0 \\
		0 & 0 & -BR^{-1}B^\top \\
		0 & BR^{-1}B^\top & 0
		\end{array}\right], 
		\mathbb{B}_0:=\left[\begin{array}{c}
		B_0  \\
		0 \\
		0
		\end{array}\right], \mathbb{D}_0:=\left[\begin{array}{c}
		b_0  \\
		b \\
		0
		\end{array}\right],
		  \\
		&\bar{\mathbb{A}}_{0}:=\left[\begin{array}{ccc}
			\bar{A}_0 & C_0 & 0 \\
			F-BR^{-1}B^\top P_3 & A_2+C & 0 \\
			0 & 0 & A_2
		\end{array}\right],\tilde{\sigma}_0:=\left[\begin{array}{c}
		\bar{\sigma}_0+\Pi_0H_0^\top f_0^{\top-1}  \\
		0 \\
		0
		\end{array}\right],\bar{\mathbb{B}}_0:=\left[\begin{array}{c}
		-\bar{\Gamma}_0^\top S_0^\top\\
		-\Gamma_0^\top S_0^\top  \\
		P_3B_0
		\end{array}\right] ,
		 \\
		&
		\mathbb{Q}_0:=\left[\begin{array}{ccc}
			Q_0 & 0 & 0 \\
			0 & 0 & 0 \\
			0 & 0 & 0
		\end{array}\right] ,\bar{\mathbb{Q}}_0:=\left[\begin{array}{ccc}
		\tilde{\Gamma}_0 & \left(\bar{\Gamma}_0^\top -I_n\right)Q_0\Gamma_0 & 0 \\
		\Gamma_0^\top Q_0\left(\bar{\Gamma}_0 -I_n\right) & \Gamma_0^\top Q_0 \Gamma_0 & 0 \\
		0 & 0 & 0
		\end{array}\right] ,
		\Xi_0:=\left[\begin{array}{c}
		\xi_0 \\
		\xi  \\
		0
		\end{array}\right] ,
		\\
		&\bar{\mathbb{D}}_0:=\left[\begin{array}{c}
			\left(\bar{\Gamma}_0^\top -I_n\right)Q_0\eta_0 \\
			\bar{\Gamma}_0^\top Q_0\eta_0 \\
			P_2b+P_3b_0+\left(\bar{\Gamma}_2^\top-I_n\right)Q\eta_2
		\end{array}\right],
		\mathbb{G}_0:=\left[\begin{array}{ccc}
			G_0 & 0 & 0 \\
			0 & 0 & 0 \\
			0 & 0 & 0
		\end{array}\right],\mathbb{S}_0:=\left[\begin{array}{c}
		S^\top_0 \\
		0  \\
		0
		\end{array}\right],\\
		& \mathbb{G}_1:=\left[\begin{array}{ccc}
			\tilde{\Gamma}_1 & \left(\bar{\Gamma}_1^\top -I_n\right)G_0\Gamma_1 & 0 \\
			\Gamma_1^\top G_0\left(\bar{\Gamma}_1 -I_n\right) & \Gamma_1^\top G_0 \Gamma_1 & 0 \\
			0 & 0 & 0
		\end{array}\right] ,
		\mathbb{G}_2:=\left[\begin{array}{c}
			\left(\bar{\Gamma}_1^\top -I_n\right)G_0\eta_1 \\
			\bar{\Gamma}_1^\top G_0\eta_1 \\
			\left(\bar{\Gamma}_4^\top-I_n\right)G\eta_4
		\end{array}\right] ,
	\end{aligned}
	$$
	with $A_2:=A+\bar{A}-BR^{-1}B^\top P_2-BR^{-1}S\left(I_n-\bar{\Gamma}_2\right)$. Then the optimal strategy of leader $\mathcal{A}_0$ can be represented as
	\begin{equation}\label{leader optimal strategy open loop  new}
		u^*_0=\mathcal{B}_0\check{\mathbb{X}}_0+\bar{\mathcal{B}}_0\mathbb{E}\mathbb{X}_0+\mathcal{B}_1\check{\Phi},
	\end{equation}
	where $\mathcal{B}_0:=\left(-R_0^{-1}S,0,0\right)$, $\bar{\mathcal{B}}_0:=\left(-R_0^{-1}S_0\bar{\Gamma}_0,-R_0^{-1}S_0\Gamma_0,R_0^{-1}B_0^\top P_3^\top\right)$,  $\mathcal{B}_1:=\left(-R_0^{-1}B_0^\top,0,0\right)$.
	Noticing the terminal condition of (\ref{dimension expansion}), we can assume that
	\begin{equation}\label{leader decouple equation}
		\check{\Phi}(t)=\Sigma_1(t) \left(\check{\mathbb{X}}_0(t)-\mathbb{E} \mathbb{X}_0(t)\right)+\Sigma_2(t) \mathbb{E} \mathbb{X}_0(t)+\psi(t),\quad t\in[0,T],
	\end{equation}
	where $\Sigma_1(\cdot)$, $\Sigma_2(\cdot)$ and $\psi(\cdot)$ are deterministic matrix-valued functions satisfying $\Sigma_1(T)=\mathbb{G}_0$, $\Sigma_2(T)=\mathbb{G}_1$ and $\psi(T)=\mathbb{G}_2$. In addition, from the first equation of (\ref{dimension expansion}), we have
	\begin{equation}\label{EX equation}
		\left\{\begin{aligned}
			d\mathbb{E}\mathbb{X}_0&=\left[\left(\mathbb{A}_0+\bar{\mathbb{A}}_0\right)\mathbb{E}\mathbb{X}_0+\mathbb{C}_{0} \mathbb{E} \Phi+\mathbb{B}_0 \mathbb{E} u_0^*+\mathbb{D}_0\right] d t, \\
			\mathbb{E}\mathbb{X}_0(0)&=\Xi_0.
		\end{aligned}\right.
	\end{equation}
	Applying It\^{o}'s formula, we have
	$$
	\begin{aligned}
		d \check{\Phi}
		&= -\left(\mathbb{A}_0^\top \check{\Phi}+\bar{\mathbb{A}}_0^\top \mathbb{E} \Phi+\mathbb{Q}_0\check{\mathbb{X}}_0 +\bar{\mathbb{Q}}_0 \mathbb{E} \mathbb{X}_0+\mathbb{S}_0u_0^*
		+\bar{\mathbb{B}}_0\mathbb{E}u^*_0+\bar{\mathbb{D}}_0\right) d t+\widetilde{\mathbb{Z}}_0 d \widetilde{W}_0.
	\end{aligned}
	$$
	By comparing the coefficients of the diffusion terms, we obtain
	$$
	\widetilde{\mathbb{Z}}_0=\Sigma_1\tilde{\sigma}_0.
	$$
	By comparing the coefficients of the drift terms and noting (\ref{leader optimal strategy open loop  new}), we obtain
	\begin{equation}\label{compare the coefficients of leader }
		\begin{aligned}
			&\left(\dot{\Sigma}_1+\Sigma_1\left(\mathbb{A}_0+\mathbb{B}_0\mathcal{B}_0\right)+\left(\mathbb{A}_0^\top+\mathbb{S}_0\mathcal{B}_1\right)\Sigma_1+\Sigma_1\mathbb{B}_0\mathcal{B}_1\Sigma_1+\mathbb{Q}_0+\mathbb{S}_0\mathcal{B}_0\right)\left(\check{\mathbb{X}}_0-\mathbb{E}\mathbb{X}_0\right)\\
			&+\left[\dot{\Sigma}_2+\Sigma_2\left(\mathbb{A}_0+\bar{\mathbb{A}}_0+\mathbb{B}_0\left(\mathcal{B}_0+\bar{\mathcal{B}}_0\right)\right)+\left(\mathbb{A}^\top_0+\bar{\mathbb{A}}^\top_0+\left(\bar{\mathbb{B}}_0+\mathbb{S}_0\right)\mathcal{B}_1\right)\Sigma_2\right.\\	&\left.+\Sigma_2\left(\mathbb{C}_0+\mathbb{B}_0\mathcal{B}_1\right)\Sigma_2+\left(\mathbb{Q}_0+\bar{\mathbb{Q}}_0\right)+\left(\bar{\mathbb{B}}_0+\mathbb{S}_0\right)\left(\mathcal{B}_0+\bar{\mathcal{B}}_0\right)\right]\mathbb{E}\mathbb{X}_0\\
			&+\dot{\psi}+\left(\mathbb{A}_0^\top+\bar{\mathbb{A}}_0^\top +\Sigma_2\mathbb{C}_0+\Sigma_2\mathbb{B}_0\mathcal{B}_1+\left(\bar{\mathbb{B}}_0+\mathbb{S}_0\right)\mathcal{B}_1\right)\psi+\Sigma_2\mathbb{D}_0+\bar{\mathbb{D}}_0=0.\\
		\end{aligned}
	\end{equation}
	Then, we can get the equation of $\Sigma_1(\cdot)$:
	\begin{equation}\label{Sigma1 equation}
		\left\{\begin{aligned}
			&\dot{\Sigma}_1+\Sigma_1\left(\mathbb{A}_0+\mathbb{B}_0\mathcal{B}_0\right)+\left(\mathbb{A}_0^\top+\mathbb{S}_0\mathcal{B}_1\right)\Sigma_1+\Sigma_1\mathbb{B}_0\mathcal{B}_1\Sigma_1+\mathbb{Q}_0+\mathbb{S}_0\mathcal{B}_0=0, \\
			&\Sigma_1(T)=\mathbb{G}_0,
		\end{aligned}\right.
	\end{equation}
	the equation of $\Sigma_2(\cdot)$:
	\begin{equation}\label{Sigma2 equation}
		\left\{\begin{aligned}
			&\dot{\Sigma}_2+\Sigma_2\left(\mathbb{A}_0+\bar{\mathbb{A}}_0+\mathbb{B}_0\left(\mathcal{B}_0+\bar{\mathcal{B}}_0\right)\right)+\left(\mathbb{A}^\top_0+\bar{\mathbb{A}}^\top_0+\left(\bar{\mathbb{B}}_0+\mathbb{S}_0\right)\mathcal{B}_1\right)\Sigma_2\\	&+\Sigma_2\left(\mathbb{C}_0+\mathbb{B}_0\mathcal{B}_1\right)\Sigma_2+\left(\mathbb{Q}_0+\bar{\mathbb{Q}}_0\right)+\left(\bar{\mathbb{B}}_0+\mathbb{S}_0\right)\left(\mathcal{B}_0+\bar{\mathcal{B}}_0\right)=0,\\
			&\Sigma_2(T)=\mathbb{G}_1,
		\end{aligned}\right.
	\end{equation}
	and
	\begin{equation}\label{psi equation}
		\left\{\begin{aligned}
			&\dot{\psi}+\left(\mathbb{A}_0^\top+\bar{\mathbb{A}}_0^\top +\Sigma_2\mathbb{C}_0+\Sigma_2\mathbb{B}_0\mathcal{B}_1+\left(\bar{\mathbb{B}}_0+\mathbb{S}_0\right)\mathcal{B}_1\right)\psi+\Sigma_2\mathbb{D}_0+\bar{\mathbb{D}}_0 =0,\\
			&\psi(T)=\mathbb{G}_2\,.
		\end{aligned}\right.
	\end{equation}
	
	\begin{myremark}
		We note that (\ref{Sigma1 equation}) and (\ref{Sigma2 equation}) are asymmetric Riccati equations, which are hard to solve in general. But, by using the result of Page 11 in Reid \cite{Reid-1972}, we can derive conditions for the solvability of equations (\ref{Sigma1 equation}) and (\ref{Sigma2 equation}) (similar to Proposition \ref{RE solvability}).  For convenience in the subsequent analysis, we assume the solvability of (\ref{Sigma1 equation}) and (\ref{Sigma2 equation}). Then noticing the equation (\ref{psi equation}) is linear ODE, which obviously admits a unique solution.
	\end{myremark}
	
	Therefore, the feedback decentralized strategy of the leader $\mathcal{A}_0$ can be given in the following.
	
	\begin{mythm}\label{theorem feedback of leader}
		Let Assumption \ref{A1} and the assumption of Proposition \ref{RE solvability} hold. Let $\Sigma_1(\cdot)$ and $\Sigma_2(\cdot)$ be the solution to (\ref{Sigma1 equation}) and (\ref{Sigma2 equation}), then the feedback decentralized strategy of the leader $\mathcal{A}_0$ can be represented as
		\begin{equation}\label{leader feedback strategy}
			\begin{aligned}
				u_0^*=&\left(\mathcal{B}_0+\mathcal{B}_1\Sigma_1\right)\check{\mathbb{X}}_0+\left(\bar{\mathcal{B}}_0+\mathcal{B}_1\left(\Sigma_2-\Sigma_1\right)\right)\mathbb{E}\mathbb{X}_0+\mathcal{B}_1\psi,
			\end{aligned}
		\end{equation}
		where $\check{\mathbb{X}}_0(\cdot)\in L_{\mathcal{F}_t^{\bar{Y}_0^*}}^2\left(0, T ; \mathbb{R}^{3n}\right)$ satisfies the following SDE:
		\begin{equation}\label{check X_0}
			\left\{\begin{aligned}
				d \check{\mathbb{X}}_0= & \left\{\left[\mathbb{A}_{0}+\mathbb{B}_0\left(\mathcal{B}_0+\mathcal{B}_1\Sigma_1\right)\right] \check{\mathbb{X}}_0 +\left[\bar{\mathbb{A}}_0+\mathbb{C}_0\Sigma_2+\mathbb{B}_0\left(\bar{\mathcal{B}}_0+\mathcal{B}_1\left(\Sigma_2-\Sigma_1\right)\right)\right] \mathbb{E} \mathbb{X}_0\right.\\
				&\left.+\left(\mathbb{C}_0+\mathbb{B}_0\mathcal{B}_1\right)\psi+\mathbb{D}_0 \right\} d t+\tilde{\sigma}_0d\widetilde{W}_0, \\
				\check{\mathbb{X}}_0(0)= &\ \Xi_0,
			\end{aligned}\right.
		\end{equation}
		and $\mathbb{E}{\mathbb{X}}_0(\cdot)\in L^2\left(0, T ; \mathbb{R}^{3n}\right)$ satisfies the following ODE:
		\begin{equation}\label{E X_0}
			\left\{\begin{aligned}
				d \mathbb{E}\mathbb{X}_0= & \left\{\left[\mathbb{A}_{0}+\bar{\mathbb{A}}_0+\mathbb{C}_0\Sigma_2
				+\mathbb{B}_0\left(\mathcal{B}_0+\bar{\mathcal{B}}_0+\mathcal{B}_1\Sigma_2\right)\right] \mathbb{E} \mathbb{X}_0\right.\\
				&\left.+\left(\mathbb{C}_0+\mathbb{B}_0\mathcal{B}_1\right)\psi+\mathbb{D}_0 \right\} d t, \\
				\mathbb{E}\mathbb{X}_0(0)= &\ \Xi_0,
			\end{aligned}\right.
		\end{equation}
		and $\psi(\cdot) \in L^2\left(0, T ; \mathbb{R}^{3n}\right)$ satisfies the {\it backward ODE} (BODE) (\ref{psi equation}), respectively.
	\end{mythm}
	
	\begin{proof}
		The equation (\ref{leader feedback strategy}) can be obtained by equation (\ref{leader optimal strategy open loop new}) and equation(\ref{leader decouple equation}). The equation(\ref{check X_0}) can be achieved from the first equation (\ref{dimension expansion}) and (\ref{leader feedback strategy}). Moreover, (\ref{E X_0}) can be got from (\ref{check X_0}) by taking expectation directly. The proof is complete.
	\end{proof}
	
	Similar to the subsection 3.2, in brief, we discuss the well-posedness of (\ref{leader Hamiltonian system}), which is equivalent to (\ref{filtering of leader hamiltionian}) and also equivalent to (\ref{dimension expansion}). Under the assumptions of Theorem \ref{theorem feedback of leader}, the equation (\ref{check X_0}) admits a solution by the classical result of SDEs. Then, the second equation of (\ref{dimension expansion}) admits a solution by Theorem 2.1 of Li et al. \cite{Li-Sun-Xiong-2019}. In other word, under the assumptions of Theorem \ref{theorem feedback of leader}, (\ref{dimension expansion}) admits a solution. Equivalently, (\ref{leader Hamiltonian system}) admits a solution.
	
	\section{$\varepsilon$-Stackelberg-Nash equilibria analysis}
	
	In Section 3, we derived the decentralized Stackelberg-Nash equilibrium, denoted as $\left(u_0^*(\cdot),u^*(\cdot)\right.$ $\left.\equiv\left(u^*_1(\cdot), u^*_2(\cdot), \ldots, u^*_N(\cdot)\right)\right)$, for the \textbf{Problem \ref{problem decentralized}}. In this section, we will prove that the decentralized strategies $\left(u_0^*(\cdot),u^*(\cdot)\right)$ satisfy the $\varepsilon$-Stackelberg-Nash equilibrium property. To begin with, we define what an $\varepsilon$-Stackelberg-Nash equilibrium entails.
	
	\begin{mydef}\label{varipsilon Stackelberg-Nash equalibrium}
		A strategy $\left(u_0^*(\cdot), u_1^*(\cdot), \cdots, u_N^*(\cdot)\right)$ is said to be an $\varepsilon$-Stackelberg-Nash equilibrium for \textbf{Problem \ref{problem centralized}}, if there exists a non-negative value $\varepsilon = \varepsilon(N)$ such that $\lim_{N \rightarrow \infty} \varepsilon(N) = 0$, which satisfies:
		
		(i) Given $u_0 \in \mathcal{U}_0^d$, for $i=1, \cdots, N$, finding the control strategy $u^*(\cdot;u_0)\equiv\left(u^*_1(\cdot;u_0), \ldots,\right.$ $\left. u^*_N(\cdot;u_0)\right)$ such that for any $u_i(\cdot) \in \mathcal{U}_i^d$,
		\begin{equation*}
			\mathcal{J}_i\left(u^*_i(\cdot;u_0),u^*_{-i}(\cdot;u_0), u_0\right) \leq  \mathcal{J}_i\left(u_i(\cdot); u^*_{-i}(\cdot;u_0), u_0\right)+\varepsilon;
		\end{equation*}
		
		(ii) Finding the control strategy  $u^*_0(\cdot)$ such that for any $u_0(\cdot) \in \mathcal{U}_0^d$,
		\begin{equation*}
			\mathcal{J}_0\left(u^*_0(\cdot),u^*(\cdot;u^*_0)\right)\leq  \mathcal{J}_0\left(u_0(\cdot ),u^*(\cdot;u_0)\right)+\varepsilon.
		\end{equation*}
	\end{mydef}
	
	The main result of this section is the following theorem.
	\begin{mythm}\label{thm4.1}
		Let Assumptions \ref{A1} and the assumption of Proposition \ref{RE solvability} hold. Let $\Sigma_1(\cdot)$ and $\Sigma_2(\cdot)$ be the solution to \eqref{Sigma1 equation} and \eqref{Sigma2 equation}, then the set of strategies $\left(u_0^*(\cdot), u_1^*(\cdot), u_2^*(\cdot),\right.$ $\left. \ldots, u_N^*(\cdot)\right)$ constitutes an $\varepsilon$-Stackelberg-Nash equilibrium for \textbf{Problem \ref{problem centralized}}, where for $i=1,2,\ldots, N$, 
		\begin{equation}\label{varipsilon Stackelberg-Nash equalibrium ui* and u0*}
			\left\{\begin{aligned}
				u_i^*=&-R^{-1}\left[\left(B^\top P_1+S\right)\hat{\bar{x}}_i^{u_0^*,u_i^*}+\left(B^\top\left(P_2-P_1\right)-S\bar{\Gamma}_2\right)\mathbb{E}\bar{x}_i^{u_0^*,u_i^*}+B^\top P_3\mathbb{E}\bar{x}_0^{u_0^*} +B^\top \varphi^{u_0^*} \right],\\
				u_0^*=&\left(\mathcal{B}_0+\mathcal{B}_1\Sigma_1\right)\check{\mathbb{X}}_0+\left(\bar{\mathcal{B}}_0+\mathcal{B}_1\left(\Sigma_2-\Sigma_1\right)\right)\mathbb{E}\mathbb{X}_0+\mathcal{B}_1\psi,
			\end{aligned}\right.
		\end{equation}
		with $P_1(\cdot)$, $\varphi^{u_0^*}(\cdot)$, $\hat{\bar{x}}_i^*(\cdot)$, $\mathbb{E}\bar{x}_i^{u_0^*,u_i^*}(\cdot)$, $\mathbb{E}\bar{x}_0^{u_0^*}(\cdot)$ satisfy \eqref{RE P1}, \eqref{varphi equation}, \eqref{hat bar x_i}, \eqref{E bar x_i}, \eqref{E x_0} and $\psi(\cdot)$, $\check{\mathbb{X}}_0(\cdot)$, $\mathbb{E}{\mathbb{X}}_0(\cdot)$ satisfy \eqref{psi equation}, \eqref{check X_0}, \eqref{E X_0}, respectively.
	\end{mythm}
	
	To establish this theorem, we shall present several lemmas, whose proofs are left to the Appendix.
	
	\begin{mylem}\label{CC lemma}
		Let Assumption \ref{A1} and the assumption of Proposition \ref{RE solvability} hold. Let $\Sigma_1(\cdot)$ and $\Sigma_2(\cdot)$ be the solutions to \eqref{Sigma1 equation} and \eqref{Sigma2 equation}, respectively. Then for any $u_0(\cdot) \in \mathcal{U}_{0}^{d}$, we have the following estimations:
		\begin{equation}\label{CC1}
			\sup _{0 \leq t \leq T} \mathbb{E}\left[\left|{\hat{\bar{x}}}^{u_0,u^*(N)}(t)-\mathbb{E}\bar{x}_i^{u_0,u_i^*}(t)\right|^2\right] = O\left(\frac{1}{N}\right),
		\end{equation}
		\begin{equation}\label{CC2}
			\sup _{0 \leq t \leq T} \mathbb{E}\left[\left|\bar{x}^{u_0,u^*(N)}(t)-\mathbb{E}\bar{x}_i^{u_0,u^*(N)}(t)\right|^2\right] = O\left(\frac{1}{N}\right),
		\end{equation}
		\begin{equation}\label{CC3}
			\sup _{0 \leq t \leq T} \mathbb{E}\left[\left|{x}^{u_0,u^*(N)}(t)-\mathbb{E}\bar{x}_i^{u_0,u_i^*}(t)\right|^2\right] = O\left(\frac{1}{N}\right),
		\end{equation}
		\begin{equation}\label{leader approximation}
			\sup _{0 \leq t \leq T}\mathbb{E} \left[\left|x_0^{u_0,u^*}(t)-\bar{x}_0^{u_0}(t)\right|^2\right] = O\left(\frac{1}{N}\right).
		\end{equation}
		\begin{equation}\label{follower approximation}
			\sup _{1 \leq i \leq N}\left[\sup _{0 \leq t \leq T} \mathbb{E} \left|x_i^{u_0,u^*}(t)-\bar{x}_i^{u_0,u_i^*}(t)\right|^2\right] = O\left(\frac{1}{N}\right),
		\end{equation}
		where $\hat{\bar{x}}^{u_0,u^*(N)}(t):=\frac{1}{N}\sum_{i=1}^{N}\hat{\bar{x}}_i^{u_0,u^*_i}(t)$, $\bar{x}^{u_0,u^*(N)}(t):=\frac{1}{N}\sum_{i=1}^{N}\bar{x}_i^{u_0,u^*_i}(t)$,  $x^{u_0,u^*(N)}(t):=\frac{1}{N}\sum_{i=1}^{N}\\x_i^{u_0,u^*}(t)$, for $t\in[0,T]$.
	\end{mylem}
	
	\begin{mylem}\label{left half of follower NE}
	Let Assumption \ref{A1} and the assumption of Proposition \ref{RE solvability} hold. Let $\Sigma_1(\cdot)$ and $\Sigma_2(\cdot)$ be the solutions to \eqref{Sigma1 equation} and \eqref{Sigma2 equation}, respectively. Then, we have the following estimation
		\begin{equation}\label{estimate followers cost}
			\left|\mathcal{J}_i\left(u_i^*(\cdot), u_{-i}^*(\cdot),u_0(\cdot)\right)-J_i\left(u_i^*(\cdot),u_0(\cdot)\right) \right|=O\left(\frac{1}{\sqrt{N}}\right),\quad \text{for any } u_0(\cdot) \in \mathcal{U}_{0}^{d},
		\end{equation}
		and
		\begin{equation}\label{estimate leader cost}
			\left|\mathcal{J}_0\left(u_0^*(\cdot),u^*(\cdot)\right)-J_0\left(u_0^*(\cdot)\right) \right|=O\left(\frac{1}{\sqrt{N}}\right).
		\end{equation}
	\end{mylem}
	
	Now, let us consider the perturbed controls of the followers. For any $u_0(\cdot) \in \mathcal{U}_0^d$, suppose the $i$-th follower $\mathcal{A}_i$ takes the value of $v_i(\cdot)\in \mathcal{U}_i^d$, while $\mathcal{A}_j(j \neq i)$ still takes $u^*_j(\cdot)$. Then, the centralized state equation for the perturbation of $\mathcal{A}_0$ is
	\begin{equation}\label{m_0 leader state perturbation}
		\left\{\begin{aligned}
			d m_0(t)= & \left[A_0(t) m_0(t)+B_0(t) u_0(t)+\bar{A}_0(t) \mathbb{E}m_0(t)+C_0(t) m^{(N)}(t)+b_0(t)\right] d t \\
			&+\sigma_0(t) d W_0(t)+\bar{\sigma}_0(t) d \overline{W}_0(t), \\
			m_0(0)= &\ \xi_0,
		\end{aligned}\right.
	\end{equation}
	and the centralized state $m_i(\cdot)$ of the follower $\mathcal{A}_i$ is given by the following linear $\mathrm{SDE}$:
	\begin{equation}\label{m_i follower state perturbation}
		\left\{\begin{aligned}
			d m_i(t) =& \left[A(t) m_i(t)+B(t) v_i(t)+\bar{A}(t) \mathbb{E}m_i(t)+C(t) m^{(N)}(t)+F(t) \mathbb{E}m_0(t)+b(t)\right] d t\\
			& +\sigma(t) d W_i(t)+\bar{\sigma}(t) d \overline{W}_i(t), \\
			m_i(0) =&\ \xi\,,
		\end{aligned}\right.
	\end{equation}
	and the centralized state equation for the $j$-th follower $\mathcal{A}_j$ is the following equation
	\begin{equation}\label{m_j perturbation}
		\left\{\begin{aligned}
			d m_j(t) =& \left[A(t) m_j(t)+B(t) u^*_j(t)+\bar{A}(t) \mathbb{E}m_j(t)+C(t) m^{(N)}(t)+F(t) \mathbb{E}m_0(t)+b(t)\right] d t\\
			& +\sigma(t) d W_j(t)+\bar{\sigma}(t) d \overline{W}_j(t), \\
			m_j(0) =&\ \xi,
		\end{aligned}\right.
	\end{equation}
	where $m_0(\cdot)=x_0^{u_0,u_{-i}^*,v_i}(\cdot)$, $m_i(\cdot)=x_i^{u_0,u_{-i}^*,v_i}(\cdot)$,  $m_j(\cdot)=x_j^{u_0,u_{-i}^*,v_i}(\cdot)(j\neq i)$,  $m^{(N)}(\cdot)=\frac{1}{N}\sum_{k=1}^{N}m_k(\cdot)$  for notational simplicity. 
	Then, the decentralized state equation of the follower $\mathcal{A}_i$ is given by the following linear $\mathrm{SDE}$ (Here, we set $\bar{m}_i(\cdot)=\bar{x}_i^{u_0,v_i}(\cdot)$ for notational simplicity):
	\begin{equation}\label{bar m_i follower state perturbation limiting}
		\left\{\begin{aligned}
			d \bar{m}_i(t) =& \big[A(t) \bar{m}_i(t)+B(t) v_i(t)+\bar{A}(t) \mathbb{E}\bar{m}_i(t)+C(t) \mathbb{E} \bar{x}^{u_0,u^*_i}(t)+F(t) \mathbb{E}\bar{x}_0^{u_0}(t)\\
			&+b(t)\big] d t +\sigma(t) d W_i(t)+\bar{\sigma}(t) d \overline{W}_i(t), \\
			\bar{m}_i(0) =&\ \xi\,.
		\end{aligned}\right.
	\end{equation}
	
	By the Definition \ref{varipsilon Stackelberg-Nash equalibrium} (i) of $\varepsilon$-Stackelberg-Nash equilibria, we need to show
	$$
	\mathcal{J}_i\left(u^*_i(\cdot), u^*_{-i}(\cdot)\right) \leq \inf _{u_i(\cdot) \in\ \mathcal{U}_i^d} \mathcal{J}_i\left(u_i(\cdot), u^*_{-i}(\cdot)\right)+\varepsilon .
	$$
	Therefore, we only need to consider $v_i(\cdot)$ satisfies $\mathcal{J}_i\left(v_i(\cdot), u^*_{-i}(\cdot)\right) \leq \mathcal{J}_i\left(u^*_i(\cdot), u^*_{-i}(\cdot)\right)$. And then,
	$$
	\mathbb{E} \int_0^T \left\|v_i(t)\right\|_{R}^2 d t \leq \mathcal{J}_i\left(v_i(\cdot), u^*_{-i}(\cdot)\right) \leq \mathcal{J}_i\left(u^*_i(\cdot), u^*_{-i}(\cdot)\right) \leq J_i\left(u^*_i(\cdot)\right)+O\left(\frac{1}{\sqrt{N}}\right),
	$$
	i.e.,
	$$
	\mathbb{E} \int_0^T\left|v_i(t)\right|^2 d t \leq K .
	$$
	
	\begin{mylem}\label{follower perturbed CC lemma}
		Let Assumption \ref{A1} and the assumption of Proposition \ref{RE solvability} hold. Let $\Sigma_1(\cdot)$ and $\Sigma_2(\cdot)$ be the solutions to \eqref{Sigma1 equation} and \eqref{Sigma2 equation}, respectively. Then, we have the following estimations:
		\begin{equation}\label{follower perturbed CC 1}
			\sup _{0 \leq t \leq T}\mathbb{E} \left[\left|m^{(N)}(t)- \mathbb{E}\bar{x}_i^{u_0,u^*_i}(t)\right|^2\right] =O\left(\frac{1}{N}\right),
		\end{equation}
		\begin{equation}\label{follower perturbed CC 2}
			\sup _{0 \leq t \leq T}\mathbb{E} \left[\left|m_0(t)-\bar{x}_0^{u_0}(t)\right|^2\right] =O\left(\frac{1}{N}\right),
		\end{equation}
		\begin{equation}\label{follower perturbed CC 3}
			\sup _{1 \leq i \leq N}\left[\sup _{0 \leq t \leq T} \mathbb{E} \left|m_{i}(t)-\bar{m}_{i}(t)\right|^2\right] =O\left(\frac{1}{N}\right).
		\end{equation}
	\end{mylem}

Its proof is postponed to the Appendix.
	
	\begin{mylem}\label{right half of follower NE}
		Let Assumption \ref{A1} and the assumption of Proposition \ref{RE solvability} hold. Let $\Sigma_1(\cdot)$ and $\Sigma_2(\cdot)$ be the solutions to \eqref{Sigma1 equation} and \eqref{Sigma2 equation}, respectively. Then, we have the following estimations:
		\begin{equation}
			\left|\mathcal{J}_i\left(v_i(\cdot), u_{-i}^*(\cdot)\right)-J_i\left(v_i(\cdot)\right) \right|=O\left(\frac{1}{\sqrt{N}}\right).
		\end{equation}
	\end{mylem}
	
    The proof is also left to the Appendix.

	The proof about (i) of the Definition \ref{varipsilon Stackelberg-Nash equalibrium} of $\epsilon$-Stackelberg-Nash equilibrium is given below. 
	\begin{proof}
		Applying Lemma \ref{left half of follower NE} and Lemma \ref{right half of follower NE}, we obtain
		\begin{equation*}
			\mathcal{J}_i\left(u^*_i(\cdot);u^*_{-i}(\cdot)\right) \leq J_i\left(u^*_i(\cdot)\right)+O\left(\frac{1}{\sqrt{N}}\right) \leq  J_i\left(v_i(\cdot)\right)+O\left(\frac{1}{\sqrt{N}}\right) \leq \mathcal{J}_i\left(v_i(\cdot);u^*_{-i}(\cdot)\right)+O\left(\frac{1}{\sqrt{N}}\right).
		\end{equation*}
   The proof is complete.
	\end{proof}
	
	Next, let us consider a perturbed control $v_0(\cdot)\in \mathcal{U}_0^d$ for the leader $\mathcal{A}_0$, the corresponding centralized state equation is
	\begin{equation}\label{tilde m_0 leader state perturbation}
		\left\{\begin{aligned}
			d \tilde{m}_0(t)= & \left[A_0(t) \tilde{m}_0(t)+B_0(t) v_0(t)+\bar{A}_0(t) \mathbb{E}\tilde{m}_0(t)+C_0(t) x^{v_0,u^*(N)}(t)+b_0(t)\right] d t  \\
			&+\sigma_0(t) d W_0(t)+\bar{\sigma}_0(t) d \overline{W}_0(t) , \\
			\tilde{m}_0(0)= &\ \xi_0,
		\end{aligned}\right.
	\end{equation}
	and the corresponding decentralized state equation is
	\begin{equation}\label{bar m_0 leader state perturbation}
		\left\{\begin{aligned}
			d \bar{m}_0(t)= & \left[A_0(t) \bar{m}_0(t)+B_0(t) v_0(t)+\bar{A}_0(t) \mathbb{E}\bar{m}_0(t)+C_0(t) \mathbb{E}\bar{x}_i^{v_0,u^*_i}(t)+b_0(t)\right] d t  \\
			&+\sigma_0(t) d W_0(t)+\bar{\sigma}_0(t) d \overline{W}_0(t)  , \\
			\bar{m}_0(0)= &\ \xi_0,
		\end{aligned}\right.
	\end{equation}
	where $\tilde{m}_0(\cdot)=x_0^{v_0,u^*}$, $\bar{m}_0(\cdot)=\bar{x}_0^{v_0}$. 
	By the Definition \ref{varipsilon Stackelberg-Nash equalibrium} (ii) of $\varepsilon$-Stackelberg-Nash equilibria, we need to show
	$$
	\mathcal{J}_0\left(u^*_0(\cdot)\right) \leq \inf _{u_0(\cdot) \in\, \mathcal{U}_{0}^{d}} \mathcal{J}_0\left(u_0(\cdot)\right)+\varepsilon.
	$$
	Therefore, we only need to consider $v_0(\cdot)$ satisfies $\mathcal{J}_0\left(v_0(\cdot)\right) \leq \mathcal{J}_0\left(u^*_0(\cdot)\right)$. And then,
	$$
	\mathbb{E} \int_0^T R_0 {v_0}^2(t) d t \leqslant \mathcal{J}_0\left(v_0(\cdot)\right) \leq \mathcal{J}_0\left(u^*_0(\cdot)\right) \leq J_0\left(u^*_0(\cdot)\right)+O\left(\frac{1}{\sqrt{N}}\right) ,
	$$
	i.e.,
	$$
	\mathbb{E} \int_0^T\left|v_0(t)\right|^2 d t \leqslant K .
	$$
	
	\begin{mylem}\label{leader perturbed CC lemma}		
		Let Assumption \ref{A1} and the assumption of Proposition \ref{RE solvability} hold. Let $\Sigma_1(\cdot)$ and $\Sigma_2(\cdot)$ be the solutions to \eqref{Sigma1 equation} and \eqref{Sigma2 equation}, respectively. Then, we have the following estimation:
		\begin{equation}\label{leader perturbed CC}
			\sup _{0 \leq t \leq T}\mathbb{E} \left[\left|\tilde{m}_{0}(t)-\bar{m}_{0}(t)\right|^2\right] =O\left(\frac{1}{N}\right) .
		\end{equation}
	\end{mylem}
	\begin{proof}
		Using Lemma \ref{CC lemma},  let $u_0=v_0$, and the proof is complete.
	\end{proof}
	
	\begin{mylem}\label{right half of leader NE}
		Let Assumption \ref{A1} and the assumption of Proposition \ref{RE solvability} hold. Let $\Sigma_1(\cdot)$ and $\Sigma_2(\cdot)$ be the solutions to \eqref{Sigma1 equation} and \eqref{Sigma2 equation}, respectively. Then, we have the following estimation:
		\begin{equation}\label{right half of leader NE_0}
			\left|\mathcal{J}_0\left(v_0(\cdot)\right)-J_0\left(v_0(\cdot)\right) \right|=O\left(\frac{1}{\sqrt{N}}\right).
		\end{equation}
	\end{mylem}
	\begin{proof}
		Using similar method as Lemma \ref{left half of follower NE}, we can prove \eqref{right half of leader NE_0} by (\ref{CC2}) and (\ref{leader perturbed CC}). The proof is complete.
	\end{proof}
	
	The proof about (ii) of the Definition \ref{varipsilon Stackelberg-Nash equalibrium} of $\epsilon$-Stackelberg-Nash equilibrium is given below. 
	
	\begin{proof}
		Applying Lemma \ref{left half of follower NE} and Lemma \ref{right half of leader NE}, we obtain
		\begin{equation*}
			\mathcal{J}_0\left(u^*_0(\cdot)\right) \leq J_0\left(u^*_0(\cdot)\right)+O\left(\frac{1}{\sqrt{N}}\right) \leq  J_0\left(v_0(\cdot)\right)+O\left(\frac{1}{\sqrt{N}}\right) \leq \mathcal{J}_0\left(v_0(\cdot)\right)+O\left(\frac{1}{\sqrt{N}}\right).
		\end{equation*}
        The proof is complete.
	\end{proof}
	Up to this point, we have proven that the strategy (\ref{varipsilon Stackelberg-Nash equalibrium ui* and u0*}) satisfies Definition \ref{varipsilon Stackelberg-Nash equalibrium} of $\epsilon$-Stackelberg-Nash equilibrium, and thus we have established Theorem 4.1.
	
	\section{Applications to a product planning problem with sticky prices}  

	In this section, we apply the theoretical results obtained in the previous sections, to a product planning problem with sticy prices. This model is based on those presented by \cite{Engwerda-2005}, \cite{Fershtman-Kamien-2005}, \cite{Li-Marelli-2021} and \cite{Wang-2025}. Here we assume that $\left\{W_0(s),\overline{W}_0(s), W_i(s),\overline{W}_i(s),0 \leq s \leq t, 1 \leq i \leq N\right\}$ is a standard $(2N+2)$-dimensional Brownian motion.
	We consider a scenario where there is a headquarters manufacturer $\mathcal{B}_0$ and $N$ local manufacturers $\mathcal{B}_i$, $i=1,2,\cdots,N$ producing homogeneous products in a large-scale market within the time interval $[0, T]$.  We further  hypothesize that market prices do not adjust immediately to the levels specified by the demand function. Given the existence of lags in market price adjustments, such prices are termed ``sticky prices".  Since the product price in a region is affected by prices in other regions, we propose the following model. The price $p_0(\cdot)$ in the region  near the headquarters manufacturer $\mathcal{B}_0$ satisfies
	\begin{equation}\label{price of leader}
		\left\{\begin{aligned}
			d p_0(t)= &\ s_0(t)\left[\lambda_0\left(a-q_0(t)\right)+\mu_0 p^{(N)}(t)-p_0(t)\right] d t +\sigma_0(t) d W_0(t)+\bar{\sigma}_0(t) d \overline{W}_0(t) , \\
			 p_0(0)= &\ \bar{p}_0,
		\end{aligned}\right.
	\end{equation}
	and the price $p_i(\cdot)$ in the region  near the local manufacturer $\mathcal{B}_i$ satisfies
	\begin{equation}\label{price of follower}
		\left\{\begin{aligned}
			d p_i(t)= &\ s(t)\left[\lambda\left(a-q_i(t)\right)+\mu p^{(N)}(t)+\nu\mathbb{E}p_0(t)-p_i(t)\right] d t +\sigma(t) d W_i(t)+\bar{\sigma}(t) d \overline{W}_i(t) , \\
			p_i(0)= &\ \bar{p}_i,\quad i=1,2,\cdots,N,
		\end{aligned}\right.
	\end{equation}
	where $q_0(\cdot)$ is the output of the headquarters manufacturer $\mathcal{B}_0$, $q_i(\cdot)$ is the output of the local manufacturer $\mathcal{B}_i$, $i=1,2,\cdots,N$ and $\lambda_0$, $\mu_0$, $\lambda$, $\mu$, $\nu$ is positive constant satisfying  $\lambda_0+\mu_0=1$ and $\lambda+\mu +\nu=1$. Here, we assume that $a-q_0(t)$ ($a>0$ is a constant) represents the price effect on the demand function for the given level of output in the region surrounding the headquarters manufacturer $\mathcal{B}_0$ (see \cite{Engwerda-2005}). Similarly, $a-q_i(t)$ is the price effect on the demand function for the given level of output in the region surrounding the local manufacturer $\mathcal{B}_i$, $i=1,2,\cdots,N$. Then, the weighted averages $\lambda_0\left(a-q_0(\cdot)\right)+\mu_0 p^{(N)}(\cdot)$ and $\lambda\left(a-q_i(\cdot)\right)+\mu p^{(N)}(\cdot)+\nu\mathbb{E}p_0(\cdot)$ represent the instantaneous prices that the headquarters manufacturer and each local manufacturer desire to achieve, reflecting the mutual influence of prices across regions. In addition, $s_0(\cdot)$ and $s(\cdot)$ denote the adjustment speed parameter of the headquarters manufacturer and each local manufacturer, respectively. Moreover, we assume that each manufacturer cannot directly observe the price; instead, they can only observe an observation process correlated with the real-time price. Then, the observation process of the headquarters manufacturer $\mathcal{B}_0$ satisfies
	\begin{equation}\label{obsevation of leader}
		\left\{\begin{aligned}
			d z_0(t)= & \left(m_0(t) p_0(t)+n_0(t)\right) d t +\delta_0(t)  d \overline{W}_0(t) , \\
			z_0(0)= &\ 0,
		\end{aligned}\right.
	\end{equation}
	and the observation process of the local manufacturer $\mathcal{B}_i$ satisfies
	\begin{equation}\label{obsevation of follower}
		\left\{\begin{aligned}
			d z_i(t)= & \left(m(t) p_i(t)+n(t)\right) d t +\delta(t)  d \overline{W}_i(t) , \\
			z_i(0)= &\ 0,\quad i=1,2,\cdots,N.
		\end{aligned}\right.
	\end{equation}
	For simplicity, the manufacturers are assumed to have quadratic production costs, then we assume that the cost functional of the headquarters manufacturer $\mathcal{B}_0$ is 
	\begin{equation}\label{headquarter cost}
		\begin{aligned}
			\mathcal{J}_0&\left(q_0(\cdot),q(\cdot)\right)=\frac{1}{2} \mathbb{E} \bigg\{\int_0^T \Big(\|p_0(t)-\bar{\eta}(t)\|_{Q_0}^2-p_0(t)q_0(t)+\|q_0(t)\|^2_{R_0}\Big)dt\bigg\}.
		\end{aligned}
	\end{equation}
	Here, $\|q_0(t)-\bar{\eta}(t)\|_{Q_0}^2$ is a penalty term that can represent government price regulation of certain daily necessities. Furthermore, $p_0(t)q_0(t)$ is written as a negative value which denotes headquarters manufacturer’s profit, since the equation (\ref{headquarter cost}) stands for the cost. Then  $\|q_0(t)\|^2_{R_0}$ represents the production cost of the headquarter manufacturer $\mathcal{B}_0$. Similarly, the cost functional of the local manufacturer $\mathcal{B}_i$ is 
	\begin{equation}\label{local cost}
		\begin{aligned}
			\mathcal{J}_i\left(q_i(\cdot),q_{-i}(\cdot),q_0(\cdot)\right)&=\frac{1}{2} \mathbb{E} \bigg\{\int_0^T \Big(\|p_i(t)-wp^{(N)}(t)-\left(1-w\right)p_0\|_{Q}^2\\
                         &\qquad\qquad -p_i(t)q_i(t)+\|q_i(t)\|^2_{R}\Big)dt\bigg\},\quad i=1,2,\cdots,N.
		\end{aligned}
	\end{equation}
	Here, $\|q_i(t)-wq^{(N)}(t)-\left(1-w\right)q_0\|_{Q}^2$ is a penalty term that captures the price gap between the products in the region near local manufacturer $\mathcal{B}_i$ and those in other regions, serving to prevent excessive price dispersion across areas. $w\in[0,1]$ is a constant. We then formulate the following problem.
	\begin{myprob}
		To find a decentralized $\varepsilon$-Stackelberg-Nash equilibrium production planning strategy set $(q_0^*(\cdot),q^*(\cdot))$, where $q^*(\cdot)\equiv ( q_1^*(\cdot), \cdots, q_N^*(\cdot))$ and  $q_i^*(\cdot)\in\mathcal{U}^d_i, i=0,1,\cdots, N$,  such that
		
		(i) Given $q_0 \in \mathcal{U}_0^d$, for $i=1, \cdots, N$, the strategy $q^*(\cdot;q_0)\equiv\left(q^*_1(\cdot;q_0), \ldots,\right.$ $\left. q^*_N(\cdot;q_0)\right)$ satisfies that for any $q_i(\cdot) \in \mathcal{U}_i^d$,
		\begin{equation*}
			\mathcal{J}_i\left(q^*_i(\cdot;q_0),q^*_{-i}(\cdot;q_0), q_0\right) \leq  \mathcal{J}_i\left(q_i(\cdot); q^*_{-i}(\cdot;q_0), q_0\right)+\varepsilon;
		\end{equation*}
		
		(ii) the strategy  $q^*_0(\cdot)$ satisfies that for any $q_0(\cdot) \in \mathcal{U}_0^d$,
		\begin{equation*}
			\mathcal{J}_0\left(q^*_0(\cdot),q^*(\cdot;q^*_0)\right)\leq  \mathcal{J}_0\left(q_0(\cdot ),q^*(\cdot;q_0)\right)+\varepsilon,
		\end{equation*}
		where $\varepsilon = \varepsilon(N)$ such that $\lim_{N \rightarrow \infty} \varepsilon(N) = 0$.
	\end{myprob}
	Obviously, this model established above is a special case of the problem introduced in Section 2. Therefore,  we apply the main results from Section 4 to solve this problem. For the simplicity of the calculations, we set the number of manufacturers $N=300$ and $T=5$. We additionally assume the following system parameters: $s_0=3$, $\lambda_0=0.5$, $a=20$, $\mu_0=0.5$, $\sigma_0=0.4$, $\bar{\sigma}_0=0.4$, $s=3$,  $\lambda=0.3$, $\mu=0.3$, $\nu=0.3$, $\sigma=0.4$, $\bar{\sigma}=0.4$, $m_0=0.8$, $n_0=0.8$, $\delta_0=0.8$, $m=0.8$, $n=0.8$, $\delta=0.8$, $\bar{\eta}=11$, $Q_0=3$, $R_0=1$, $w=0.5$, $Q=5$, $R=1$. To provide more intuitive insight, we use the specific coefficients above to generate several plots that corroborate the validity of our practical findings. Taking into account that the headquarters manufacturer is located in a relatively prosperous area, we have set a higher initial price for the vicinity of the headquarters' manufacturer, while a lower initial price has been set for areas near the local manufacturers. Therefore, we set $\bar{p}_i(0)=8$ and $\bar{p}_0(0)=14$.
	
	By the  Euler's method, the trajectories of the optimal decentralized strategies of $q^*_i(\cdot),i=1,\cdots,N$ of $N$ local manufacturers $\mathcal{B}_i,i=1,\cdots,N$ are plotted by Figure \ref{fig:q^*_i}. 
\begin{figure}[H]
		\centering  
			\includegraphics[width=0.6\textwidth]{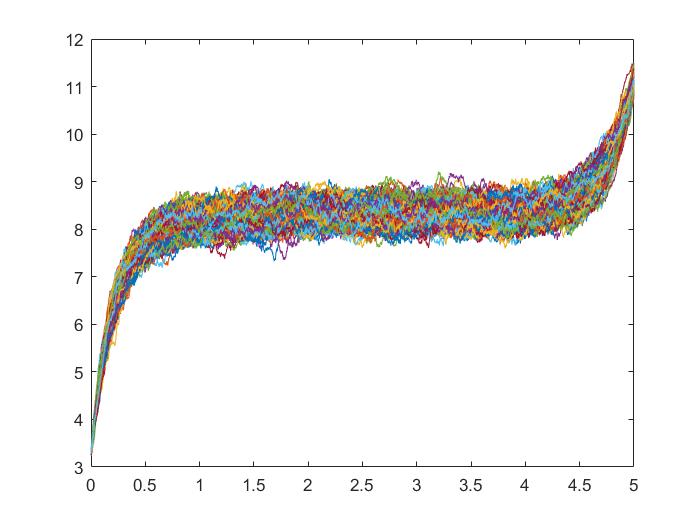}
			\caption{The trajectories of $q^*_i(\cdot),i=1,\cdots,N$. }
			\label{fig:q^*_i}
	\end{figure}

After implementing the optimal decentralized strategies, the trajectories of prices $p^*_i(\cdot), i=1,\cdots,N$ of $N$ local manufacturers $\mathcal{B}_i,i=1,\cdots,N$ are shown in Figure \ref{fig:p^*_i}. 

    \begin{figure}[H]	
        \centering
			\includegraphics[width=0.6\textwidth]{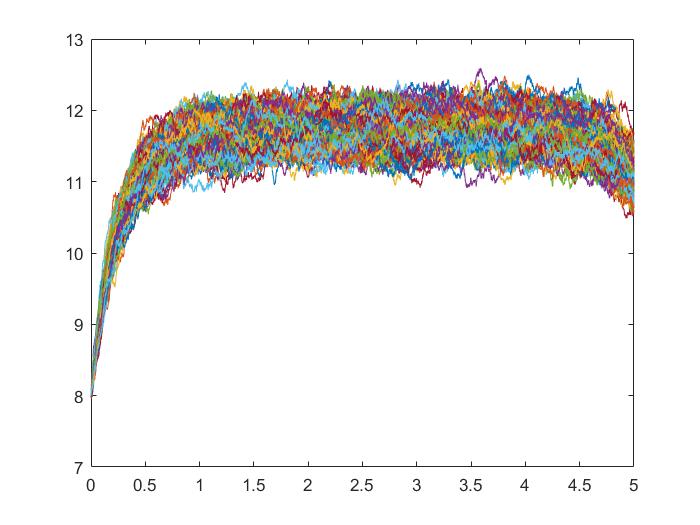}
			\caption{The trajectories of $p^*_i(\cdot)$, $i=1,\cdots,N$}
			\label{fig:p^*_i}
	\end{figure}

We can observe that the initial output of the product is relatively low. This is because people living near local manufacturers have lower purchasing power and weaker acceptance of new things, which results in initially low demand in those regions. As people gain a deeper understanding of the product, demand increases, driving prices upward. When supply and demand reach equilibrium, both price and output maintain stability. When the production of the product is scheduled to be phased out, the product may have been superseded by newer versions. Therefore, local manufacturers eager to avoid excess inventory and reduced prices, which leads to the increases in demand and output. In addition, Figure \ref{fig:p^{*(N)}} illustrates the convergence of the price-average term $p^{*(N)}(\cdot)$ to the limiting process $\mathbb{E}\bar{p}_i^*(\cdot)$. 

    \begin{figure}[H]
		\centering
		\includegraphics[width=0.6\textwidth]{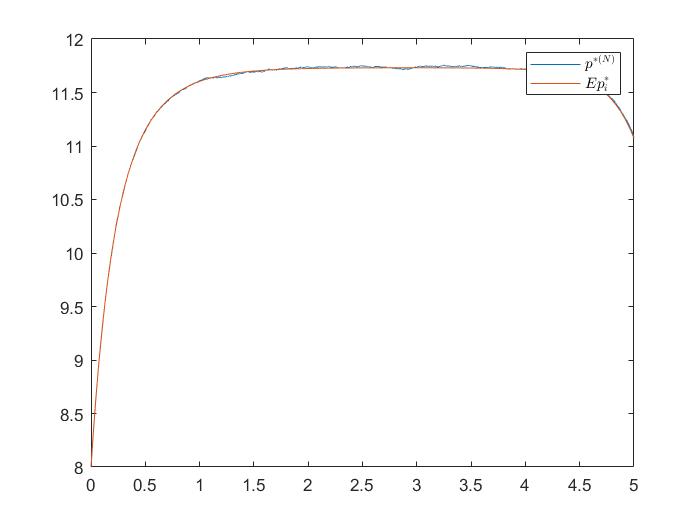}
		\caption{The solutions curves of $p^{*(N)}(\cdot)$ and $\mathbb{E}p^*_i(\cdot)$.}
		\label{fig:p^{*(N)}} 
	\end{figure}
	
	Similarly, the trajectory of the optimal decentralized strategy of  $q^*_0(\cdot)$ of the headquarters manufacturer $\mathcal{B}_0$ and its expectation are given by Figure \ref{fig:q^*_0}. After implementing the optimal decentralized strategy $q^*_0(\cdot)$, the prices $p^*_0(\cdot)$ of the headquarters manufacturer $\mathcal{B}_0$ and their expectation are shown in Figure \ref{fig:p^*_0}. Since people near the headquarters manufacturer have better economic conditions, there is high demand when the product is put on sale. Subsequently, the demand for the product weakens, causing prices to drop until both demand and price stabilized. As the product is nearing the end of its production cycle, the region near the headquarters manufacturer also encounters a situation similar to that of the local manufacturers: the price is lowered to promote sales, which in turn leads to increases in demand and output.

    \begin{figure}[H]
		\centering  
			\includegraphics[width=0.6\textwidth]{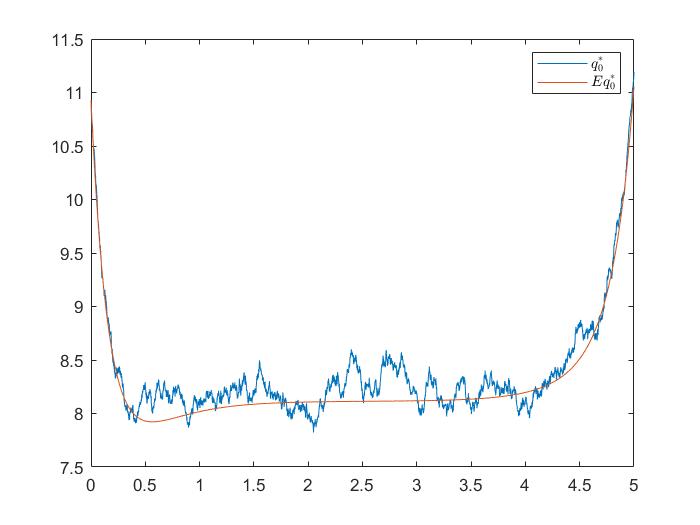}
			\caption{The trajectories of $q^*_0(\cdot)$ and $\mathbb{E}q^*_0(\cdot)$.}
			\label{fig:q^*_0}
	\end{figure}

	\begin{figure}[H]	
        \centering 
			\includegraphics[width=0.6\textwidth]{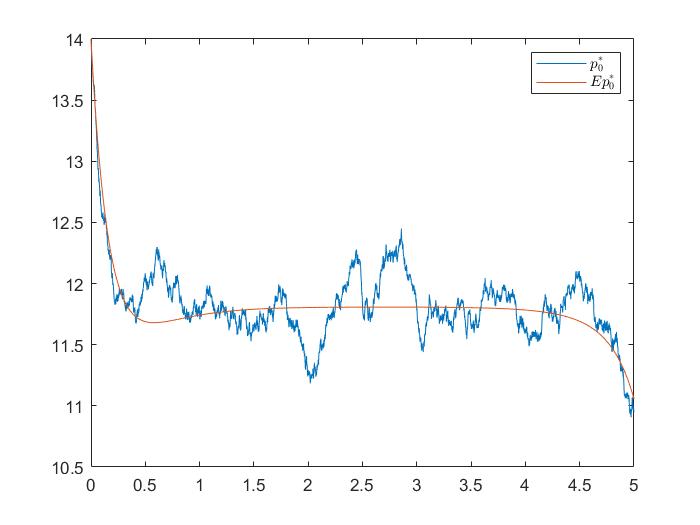}
			\caption{The trajectories of $p^*_0(\cdot)$ and $\mathbb{E}p^*_0(\cdot)$}
			\label{fig:p^*_0} 
	\end{figure}

	\section{Conclusion}
	
	In this paper, we have studied a linear-quadratic partially observed mean field Stackelberg stochastic differential game. The state equations of the single leader and multiple followers are general. By utilizing the techniques of state decomposition and backward separation principle, we overcome the difficulty of circular dependency arising from the partial observation framework. Then the strategies are given by the stochastic maximum principle with partial information and optimal filter technique, and the feedback form of decentralized strategies have been obtained by using some high-dimensional Riccati equations. Moreover, the decentralized  strategies obtained has been verified as an $\varepsilon$-Stackelberg-Nash equilibrium of the original game (\textbf{Problem \ref{problem centralized}}).  Finally, we apply the theoretical result to a product planning problem with sticky prices. 
	
	The general solvability of the high-dimensional Riccati equation \eqref{RE P2}, \eqref{RE P3}, \eqref{Sigma1 equation}, \eqref{Sigma2 equation}, \eqref{RE P2} is rather challenging. Problems with partial observation are the interesting but difficult ones. We will consider these topics in our future research. We will try to extend our results to cases with delays, nonlinearities, and random coefficients. Also, a continuous-time partially observed large-population Stackelberg game with heterogeneous followers is worthy to be investigated, where the heterogeneity can be modeled with $K$-distinct types or continuum parameter sets and different type partial observability can be described with corresponding distinct observation processes. We will consider these topics in the near future.

\section*{Appendix}\setcounter{section}{0}

\subsection{Proof of Theorem \ref{open loop thm1 of follower}}

\begin{proof}
		Due to Lemma \ref{follower lemma2}, we only need to minimize $J_i\left(u_i\right)$ over $\mathcal{U}_i$. Suppose $u_i^{u_0,u_i^*}$ is the optimal strategy of the sub-problem (i) in \textbf{Problem (\ref{problem decentralized})} and $\bar{x}^*_i$ is the corresponding optimal trajectory. For any $u_i \in \mathcal{U}_i$ and $\forall$ $\varepsilon>0$, we denote
		$$
		u_i^\varepsilon=u^*_i+\varepsilon v_i \in \mathcal{U}_i,
		$$
		where $v_i=u_i-u^*_i$.
		
		Let $\bar{x}^{u_0,u_i^{\epsilon}}_i$ be the solution of the following perturbed state equation:
		$$
		\left\{\begin{aligned}
			\bar{x}_i^{u_0,u_i^{\epsilon}}&=\left[A \bar{x}^{u_0,u_i^{\epsilon}}_i+B u^\varepsilon_i+\bar{A} \mathbb{E}\bar{x}^{u_0,u_i^{\epsilon}}_i+C z+F \mathbb{E}\bar{x}_0^{u_0}+b\right] d t +\sigma d W_i+\bar{\sigma} d \overline{W}_i, \\
			\bar{x}_i^{u_0,u_i^{\epsilon}}(0)&=\xi.
		\end{aligned}\right.
		$$
		Let $\Delta \bar{x}_i:=\frac{\bar{x}_i^{u_0,u_i^{\epsilon}}-\bar{x}^{u_0,u_i^*}_i}{\varepsilon}$. It can be verified that $\Delta \bar{x}_i$ satisfies
		$$
		\left\{\begin{aligned}
			d \Delta \bar{x}_i&=\left(A \Delta \bar{x}_i+B v_i+\bar{A} \mathbb{E}\Delta\bar{x}_i\right) d t \\
			\Delta \bar{x}_i(0)&=0.
		\end{aligned}\right.
		$$
		Applying It\^o's formula to $\left\langle \Delta \bar{x}_i, p_i\right\rangle$, we derive
		\begin{equation}\label{Ito of open loop follower}
			\begin{aligned}
				&\mathbb{E}\left[\left\langle \Delta\bar{x}_i(T), p_i(T) \right\rangle\right]=\mathbb{E} \int_0^T \bigg[-\left\langle \Delta \bar{x}_i,Q\bar{x}_i^{u_0,u_i^*}
				+\tilde{\Gamma}_2 \mathbb{E}\bar{x}^{u_0,u_i^*}_i +\left(\bar{\Gamma}_2^\top-I_n\right)Q\Gamma_2 z+S^\top u_i^* \right.\\
				&\qquad -\bar{\Gamma}_2^\top S^\top\mathbb{E}u_i^* -Q\Gamma_3 \bar{x}_0^{u_0}  \left.+\tilde{\Gamma}_3\mathbb{E}\bar{x}_0^{u_0}+\left(\bar{\Gamma}_2^\top-I_n\right)Q\eta_2\right\rangle+\left\langle B v_i, p_i\right\rangle \bigg]d t.
			\end{aligned}
		\end{equation}
		Then
		$$
		\begin{aligned}
			& J_i\left(u_i^\varepsilon\right)-J_i\left(u^*_{i}\right)=\frac{\varepsilon^2}{2} X_1+\varepsilon X_2,
		\end{aligned}
		$$
		where
		$$
		\begin{aligned}
			X_1&:=\mathbb{E}\bigg[\int_0^T\bigg( \left\langle Q\left(\Delta\bar{x}_i-\bar{\Gamma}_2\mathbb{E}\Delta\bar{x}_i\right),\Delta\bar{x}_i-\bar{\Gamma}_2\mathbb{E}\Delta\bar{x}_i\right\rangle
			+2\left\langle S\left(\Delta\bar{x}_i-\bar{\Gamma}_2\mathbb{E}\Delta\bar{x}_i\right),v_i\right\rangle  \\
			&\qquad \quad+ \left\langle Rv_i,v_i\right\rangle\bigg)d t+\left\langle G\left(\Delta\bar{x}_i(T)-\bar{\Gamma}_4\mathbb{E}\Delta\bar{x}_i(T)\right),\Delta\bar{x}_i(T)-\bar{\Gamma}_4\mathbb{E}\Delta\bar{x}_i(T)\right\rangle \bigg],
		\end{aligned}
		$$
		\begin{equation}\label{X_2 of follower}
			\begin{aligned}
				X_2&:=\mathbb{E}\left[\int_0^T \bigg(\left\langle Q\left(\Delta\bar{x}_i-\bar{\Gamma}_2\mathbb{E}\Delta\bar{x}_i\right),\bar{x}^{u_0,u_i^*}_i
				-\Gamma_2  z-\bar{\Gamma}_2  \mathbb{E}\bar{x}^{u_0,u_i^*}_i-\Gamma_3\bar{x}_0-\bar{\Gamma}_3\mathbb{E}\bar{x}_0 -\eta_2\right\rangle  \right.\\
				&\qquad +\left\langle R v_i, u^*_i\right\rangle+\left\langle S\left(\Delta\bar{x}_i-\bar{\Gamma}_2\mathbb{E}\Delta\bar{x}_i\right),u_i^*\right\rangle+\left\langle S\left(\bar{x}_i^{u_0,u_i^*}-\bar{\Gamma}_2\mathbb{E}\bar{x}_i^{u_0,u_i^*}\right),v_i\right\rangle\bigg)d t\\
				&\qquad +\left\langle G\big(\Delta\bar{x}_i(T)-\bar{\Gamma}_4\mathbb{E}\Delta\bar{x}_i(T)\big),\bar{x}^{u_0,u_i^*}_i(T)-\Gamma_4  z(T)-\bar{\Gamma}_4  \mathbb{E}\bar{x}^{u_0,u_i^*}_i(T)\right.\\
                &\qquad -\left.\Gamma_5 \bar{x}_0^{u_0}(T)-\bar{\Gamma}_5\mathbb{E}\bar{x}_0^{u_0}(T)-\eta_4\right\rangle \bigg].
			\end{aligned}
		\end{equation}
		Due to the optimality of $u^*_i$, we have $J_i\left(u_i^\varepsilon \right)-J_i\left(u^*_{i} \right) \geq 0$. Using the completing the square method and $Q-S^\top R^{-1}S\geq 0$ in Assumption \ref{A1}, we obtain $X_1\geq 0$. Noticing the arbitrariness of $\epsilon$, we have $X_2=0$. Then, simplifying (\ref{X_2 of follower}) with (\ref{Ito of open loop follower}), we have
		$$
		X_2=\mathbb{E} \int_0^T\left\langle B^{\top} \mathbb{E}\left[\left.p_i\right\rvert \mathcal{F}_t^{\bar{Y}_{i1}}\right]+R u^*_i+S(\mathbb{E}\left[\left.\bar{x}_i^{u_0,u_i^*}\right\rvert \mathcal{F}_t^{\bar{Y}_{i1}}\right]-\bar{\Gamma}_2\mathbb{E}\bar{x}^{u_0,u_i^*}_i), v_i\right\rangle d t\,.
		$$
		Due to the arbitrariness of $v_i\in\mathcal{U}_i$, we have $B^{\top} \mathbb{E}\left[\left.p_i\right\rvert \mathcal{F}_t^{\bar{Y}_{i1}}\right]+R u^*_i+S(\mathbb{E}\left[\left.\bar{x}_i^{u_0,u_i^*}\right\rvert \mathcal{F}_t^{\bar{Y}_{i1}}\right]-\bar{\Gamma}_2\mathbb{E}\bar{x}^{u_0,u_i^*}_i)=0$. Noting the result $\mathcal{F}^{\bar{Y}_{i1}}=\mathcal{F}^{\bar{Y}^{u_0,u_i^*}_{i}}$ of Lemma \ref{follower lemma1}, we obtain the optimal conditions (\ref{follower open loop optimal control}). The proof is complete.
	\end{proof}

\subsection{Proof of Theorem \ref{open loop thm1 of leader}}

\begin{proof}
		Due to Lemma \ref{leader lemma2}, we only need to minimize $J_0\left(u_0\right)$ over $\mathcal{U}_0$. Suppose $u_0^*$ is optimal strategy of the sub-problem (ii) in \textbf{Problem (\ref{problem decentralized})} and $\left(\bar{x}^{u_0^*}_0,\mathbb{E}\bar{x}_i^{u_0^*},\varphi^{u_0^*}\right)$ are the corresponding optimal trajectories. For any $u_0 \in \mathcal{U}_0$ and $\forall$ $\varepsilon>0$, we denote
		$$
		u_0^\varepsilon=u^*_0+\varepsilon v_0 \in \mathcal{U}_0,
		$$
		where $v_0=u_0-u^*_0$.
		Let $\left(\bar{x}^{u_0^{\epsilon}}_0,\mathbb{E}\bar{x}_i^{u_0^{\epsilon},u_i^*},\varphi^\epsilon\right)$ be the solution to the following perturbed state equation:
		$$
		\left\{\begin{aligned}
			d \bar{x}^{u_0^{\epsilon}}_0=& \left(A_0 \bar{x}^\epsilon_0+B_0 u^{\epsilon}_0+\bar{A}_0\mathbb{E}\bar{x}^{u_0^{\epsilon}}_0+C_0 \mathbb{E}\bar{x}_i^{u_0^{\epsilon},u_i^*}+b_0\right) d t +\sigma_0 d W_0+\bar{\sigma}_0 d \overline{W}_0, \\
			d \mathbb{E}\bar{x}_i^{u_0^{\epsilon},u_i^*}=& \left[\left(A+\bar{A}+C-BR^{-1}B^\top P_2-BR^{-1}S\left(I_n-\bar{\Gamma}_2\right)\right)\mathbb{E}\bar{x}_i^{u_0^{\epsilon},u_i^*} \right.\\
			&\left.+\left(F-BR^{-1}B^\top P_3\right) \mathbb{E}\bar{x}^{u_0^{\epsilon}}_0-BR^{-1}B^\top\varphi^{u_0^{\epsilon}}+b\right] d t,\\
			d\varphi^{u_0^{\epsilon}}=&-\left[\left(A+\bar{A}-BR^{-1}B P_2-BR^{-1}S\left(I_n-\bar{\Gamma}_2\right)\right)^\top\varphi^{u_0^{\epsilon}}+P_3B_0\mathbb{E}u^\epsilon_0+P_2b+P_3b_0\right.\\
			&\quad\left.+\left(\bar{\Gamma}_2^\top-I_n\right)Q\eta_2\right]dt,\\
			\bar{x}^{u_0^{\epsilon}}_0(0) =&\ \xi_0, \quad \bar{x}_i^{u_0^{\epsilon},u_i^*}(0)=\xi,\quad \varphi^{u_0^{\epsilon}}(T)=\left(\bar{\Gamma}_4^\top-I_n\right)G\eta_4\,.
		\end{aligned}\right.
		$$
		Let $\Delta \bar{x}_0=\frac{\bar{x}_0^{u_0^{\epsilon}}-\bar{x}^{u_0^*}_0}{\varepsilon}$, $\Delta \mathbb{E}\bar{x}_i=\frac{\mathbb{E}\bar{x}_i^{u_0^{\epsilon},u_i^*}-\mathbb{E}\bar{x}_i^{u_0^*,u_i^*}}{\varepsilon}$, $\Delta \varphi=\frac{\varphi^{u_0^{\epsilon}}-\varphi^{u_0^*}}{\varepsilon}$. We can verify that $\left(\Delta\bar{x}_0,\Delta\mathbb{E}\bar{x}_i,\Delta\varphi\right)$ satisfies
		$$
		\left\{\begin{aligned}
			d \Delta\bar{x}_0= & \left(A_0 \Delta\bar{x}_0+B_0 v_0+\bar{A}_0\Delta\mathbb{E}\bar{x}_0+C_0 \Delta\mathbb{E}\bar{x}_i\right) d t, \\
			d \Delta\mathbb{E}\bar{x}_i=& \left[\left(A+\bar{A}+C-BR^{-1}B^\top P_2-BR^{-1}S\left(I_n-\bar{\Gamma}_2\right)\right)\Delta\mathbb{E}\bar{x}_i \right.\\
			&\left.+\left(F-BR^{-1}B^\top P_3\right) \Delta\mathbb{E}\bar{x}_0-BR^{-1}B^\top\Delta\varphi\right] d t,\\
			d\Delta\varphi=&-\left[\left(A+\bar{A}-BR^{-1}B P_2-BR^{-1}S\left(I_n-\bar{\Gamma}_2\right)\right)^\top\Delta\varphi+P_3B_0\mathbb{E}v_0\right]dt,\\
			\Delta\bar{x}_0(0) =& 0, \quad \Delta\mathbb{E}\bar{x}_i(0)=0,\quad \Delta\varphi(T)=0.
		\end{aligned}\right.
		$$
		Applying It\^o's formula to $\left\langle \Delta \bar{x}_0, y_0\right\rangle+\left\langle \Delta \mathbb{E}\bar{x}_i, y\right\rangle+\left\langle \Delta \varphi, \gamma\right\rangle$, we derive
		\begin{equation}\label{Ito of open loop leader}
			\begin{aligned}
				\mathbb{E}&\big[\left\langle \Delta \bar{x}_0(T), y_0(T)\right\rangle+\left\langle \Delta \mathbb{E}\bar{x}_i(T), y(T)\right\rangle\big]=\mathbb{E} \int_0^T \bigg[\left\langle \Delta \bar{x}_0,-Q_0\bar{x}_0^{u_0^*}-\tilde{\Gamma}_0 \mathbb{E}\bar{x}^{u_0^*}_0+S_0^\top u^*_0\right.\\
				&\left.-\bar{\Gamma}_0^\top S_0^\top\mathbb{E}u^*_0 -\left(\bar{\Gamma}_0^\top-I_n\right)Q_0\Gamma_0\mathbb{E}\bar{x}_i^{u_0^*,u_i^*}-\left(\bar{\Gamma}_0^\top-I_n\right)Q_0\eta_0\right\rangle+\left\langle B v_0, y_0\right\rangle+\left\langle -P_3B_0\mathbb{E} v_0, \gamma\right\rangle\\
				& +\left\langle \Delta\mathbb{E} \bar{x}_i,-\Gamma_0^\top S_0^\top\mathbb{E}u^*_0-\Gamma_0^\top Q_0\left(\bar{\Gamma}_0-I_n\right) \mathbb{E}\bar{x}_0^{u_0^*}-\Gamma_0^\top Q_0\Gamma_0\mathbb{E}\bar{x}_i^{u_0^*,u_i^*}-\Gamma_0^\top Q_0\eta_0\right\rangle \bigg]d t.
			\end{aligned}
		\end{equation}
		Then
		$$
		\begin{aligned}
			& J_0\left(u_0^\varepsilon\right)-J_0\left(u^*_{0}\right)=\frac{\varepsilon^2}{2} \bar{X}_1+\varepsilon \bar{X}_2,
		\end{aligned}
		$$
		where
		$$
		\begin{aligned}
			\bar{X}_1&:=\mathbb{E}\bigg[\int_0^T \left\langle Q_0\left(\Delta\bar{x}_0-\Gamma_0\Delta\mathbb{E}\bar{x}_i-\bar{\Gamma}_0\mathbb{E}\Delta\bar{x}_0\right),\Delta\bar{x}_0
			 -\Gamma_0\Delta\mathbb{E}\bar{x}_i-\bar{\Gamma}_0\mathbb{E}\Delta\bar{x}_0\right\rangle + \left\langle Rv_0,v_0\right\rangle  \\
			&\qquad +\left\langle S_0\left(\Delta\bar{x}_0-\Gamma_0\Delta\mathbb{E}\bar{x}_i-\bar{\Gamma}_0\mathbb{E}\Delta\bar{x}_0\right),v_0\right\rangle d t
             +\left\langle G_0\left(\Delta\bar{x}_0(T)-\Gamma_1\Delta\mathbb{E}\bar{x}_i(T)\right.\right.\\
			&\qquad -\left.\left. \bar{\Gamma}_1\mathbb{E}\Delta\bar{x}_0(T)\right),\Delta\bar{x}_0(T)-\Gamma_1\Delta\mathbb{E}\bar{x}_i(T)-\bar{\Gamma}_1\mathbb{E}\Delta\bar{x}_0(T)\right\rangle \bigg],
		\end{aligned}
		$$
		\begin{equation}\label{X_2 of leader}
			\begin{aligned}
				\bar{X}_2&:=\mathbb{E}\bigg[\int_0^T \left\langle Q_0\left(\Delta\bar{x}_0-\Gamma_0\Delta\mathbb{E}\bar{x}_i-\bar{\Gamma}_0\mathbb{E}\Delta\bar{x}_0\right), \bar{x}_0^{u_0^*}
				-\Gamma_0\mathbb{E}\bar{x}_i^{u_0^*,u_i^*}-\bar{\Gamma}_0\mathbb{E}\bar{x}_0^{u_0^*}-\eta_0 \right\rangle  \\
				&\qquad +\left\langle R_0 v_0, u^*_0\right\rangle+\left\langle S_0\left(\bar{x}_0^{u_0^*}
				-\Gamma_0\mathbb{E}\bar{x}_i^{u_0^*,u_i^*}-\bar{\Gamma}_0\mathbb{E}\bar{x}_0^{u_0^*}\right),v_0\right\rangle\\
                &\qquad +\left\langle S_0\left(\Delta\bar{x}_0-\Gamma_0\Delta\mathbb{E}\bar{x}_i-\bar{\Gamma}_0\mathbb{E}\Delta\bar{x}_0\right),u^*_0\right\rangle d t+\Big\langle G_0\left(\Delta\bar{x}_0(T)-\Gamma_1\Delta\mathbb{E}\bar{x}_i(T)\right.\\
				&\qquad -\left.\bar{\Gamma}_1\mathbb{E}\Delta\bar{x}_0(T)\right),\bar{x}^{u_0^*}_0(T)-\Gamma_1\mathbb{E}\bar{x}_i^{u_0^*,u_i^*}(T)-\bar{\Gamma}_1\mathbb{E}\bar{x}_0^{u_0^*}(T)-\eta_1\Big\rangle \bigg].
			\end{aligned}
		\end{equation}
		Due to the optimality of $u^*_0$, we have $J_0\left(u_0^\varepsilon \right)-J_0\left(u^*_{0} \right) \geq 0$. By using the completing the square method and $Q_0-S_0^\top R_0^{-1}S_0\geq 0$ in Assumption \ref{A1},  we can get $\bar{X}_1 \geq 0$. Noticing the arbitrariness of $\epsilon$, we have $\bar{X}_2=0$. Then, simplifying (\ref{X_2 of leader}) with (\ref{Ito of open loop leader}), we have
		$$
		\begin{aligned}
			\bar{X}_2=&\mathbb{E} \int_0^T\Big\langle B_0^{\top} \mathbb{E}\left[y_0|\mathcal{F}^{\bar{Y}_{01}}_t\right]-B_0^\top P_3^\top\gamma+R u^*_0\\
			&+S_0\left(\mathbb{E}\left[\left. \bar{x}^{u_0^*}_0\right\rvert \mathcal{F}^{\bar{Y}_{01}}_t\right]-\Gamma_0\mathbb{E}\bar{x}_i^{u_0^*,u_i^*}-\bar{\Gamma}_0\mathbb{E}\bar{x}^{u_0^*}_0\right), v_0\Big\rangle d t.
		\end{aligned}
		$$
		Due to the arbitrariness of $v_0\in\mathcal{U}_0$, we have $B_0^{\top} \mathbb{E}\left[y_0|\mathcal{F}^{\bar{Y}_{01}}_t\right]-B_0^\top P_3^\top\gamma+R u^*_0+S_0\big(\mathbb{E}\left[\left. \bar{x}^{u_0^*}_0\right\rvert \mathcal{F}^{\bar{Y}_{01}}_t\right] -\Gamma_0\mathbb{E}\bar{x}_i^{u_0^*,u_i^*}-\bar{\Gamma}_0\mathbb{E}\bar{x}^{u_0^*}_0\big)=0$. Noting the result $\mathcal{F}^{\bar{Y}^{u_0^*}_0}_t=\mathcal{F}^{\bar{Y}_{01}}_t$ of Lemma \ref{leader lemma1}, we obtain the optimal conditions (\ref{leader open loop optimal strategy}). The proof is complete.
	\end{proof}

\subsection{Proof of Lemma \ref{CC lemma}}

\begin{proof}
		Due to (\ref{hat bar x_i}), we have
		\begin{equation}\label{hat bar x^*N}
			\left\{\begin{aligned}
				d \hat{\bar{x}}^{u_0,u^*(N)}&= \left\{\left(A-B R^{-1} B^\top P_1-BR^{-1}S\right) \hat{\bar{x}}^{u_0,u^*(N)}+\left[\bar{A}+C-B R^{-1} B^\top\right. \right.\\
				&\qquad \times(P_2-P_1)+BR^{-1}S\bar{\Gamma}_2\Big]\mathbb{E} \bar{x}_i^{u_0,u^*_i}+\left(F-B R^{-1} B^{\top} P_3\right) \mathbb{E}\bar{x}^{u_0}_0\\
				&\qquad +b-B R^{-1} B^\top \varphi^{u_0}+\left.\left[\bar{\sigma}+\Pi H^\top (f^\top)^{-1}\right]f^{-1}H\left(\bar{x}^{u_0,u^*(N)}\right.\right.\\
                &\qquad -\left.\left.\hat{\bar{x}}^{u_0,u^*(N)}\right)\right\} d t+\frac{1}{N}\sum_{i=1}^{N}\left[\bar{\sigma}+\Pi H^\top (f^\top)^{-1}\right] d \overline{W}_i, \\
				\hat{\bar{x}}^{u_0,u^*(N)}(0) &=\xi.
			\end{aligned}\right.
		\end{equation}
		From (\ref{hat bar x_i}) and (\ref{E bar x_i}), we obtain
		\begin{equation}\label{hatx-Ebarx}
			\left\{\begin{aligned}
				d\big( \hat{\bar{x}}^{u_0,u^*(N)}-\mathbb{E}\bar{x}_i^{u_0,u^*_i}\big) &=\bigg\{\left(A-B R^{-1} B^\top P_1-BR^{-1}S\right) \left( \hat{\bar{x}}^{u_0,u^*(N)}
				-\mathbb{E}\bar{x}_i^{u_0,u^*_i}\right)\\
				&\qquad +\left[\bar{\sigma}+\Pi H^\top (f^\top)^{-1}\right]f^{-1}H\left(\bar{x}^{u_0,u^*(N)}-\hat{\bar{x}}^{u_0,u^*(N)}\right)\bigg\} d t\\
                &\qquad +\frac{1}{N}\sum_{i=1}^{N}\left[\bar{\sigma}+\Pi H^\top (f^\top)^{-1}\right] d \overline{W}_i,\\
				\big(\hat{\bar{x}}^{u_0,u^*(N)}-\mathbb{E}\bar{x}_i^{u_0,u^*_i}\big)(0) &= 0.
			\end{aligned}\right.
		\end{equation}
		Similarly, we have 
		\begin{equation}\label{bar x^*N}
			\left\{\begin{aligned}
				d\bar{x}^{u_0,u^*(N)}= & \bigg\{A\bar{x}^{u_0,u^*(N)}-B R^{-1}\left( B^\top P_1+S\right) \hat{\bar{x}}^{u_0,u^*(N)}+\left[\bar{A}+C-B R^{-1} B^\top\right.\\
				& \left.\times\left(P_2-P_1\right)+BR^{-1}S\bar{\Gamma}_2\right]\mathbb{E} \bar{x}_i^{u_0,u^*_i}+\left(F-B R^{-1} B^{\top} P_3\right) \mathbb{E}\bar{x}_0^{u_0}\\
				& +b-B R^{-1} B^\top \varphi^{u_0}\bigg\} d t+\frac{1}{N}\sum_{i=1}^{N}\sigma d W_i+\frac{1}{N}\sum_{i=1}^{N}\bar{\sigma} d \overline{W}_i, \\
				\bar{x}^{u_0,u^*(N)}(0) & =\xi.
			\end{aligned}\right.
		\end{equation}
		\begin{equation*}
			\left\{\begin{aligned}
				d\big( \bar{x}&^{u_0,u^*(N)}-\mathbb{E}\bar{x}_i^{u_0,u^*_i}\big)  =\bigg\{A\big( \bar{x}^{u_0,u^*(N)}-\mathbb{E}\bar{x}_i^{u_0,u^*_i}\big) \\
				&-B R^{-1}\left( B^\top P_1+S\right) \big( \hat{\bar{x}}^{u_0,u^*(N)}-\mathbb{E}\bar{x}_i^{u_0,u^*_i}\big)\bigg\} d t+\frac{1}{N}\sum_{i=1}^{N}\sigma d W_i+\frac{1}{N}\sum_{i=1}^{N}\bar{\sigma} d \overline{W}_i,\\
				\big(\bar{x}&^{u_0,u^*}-\mathbb{E}\bar{x}_i^{u_0,u^*_i}\big)(0) = 0.
			\end{aligned}\right.
		\end{equation*}
		By integrating and taking expectation on both sides of the equation (\ref{hatx-Ebarx}), we have for some positive constant $K$ (may be different line by line),
		\begin{equation*}
			\begin{aligned}
				& \mathbb{E} \left[\left| \hat{\bar{x}}^{u_0,u^*(N)}(t)-\mathbb{E}\bar{x}_i^{u_0,u^*_i}(t)\right|^2\right]\leq  K \mathbb{E} \int_0^t\left|\hat{\bar{x}}^{u_0,u^*(N)}(s)-\mathbb{E}\bar{x}_i^{u_0,u^*_i}(s)\right|^2\\
				& +\left|\bar{x}^{u_0,u^*(N)}(s)-\mathbb{E}\bar{x}_i^{u_0,u^*_i}(s)\right|^2 d s+\frac{K}{N^2} \mathbb{E} \sum_{i=1}^N \left[\int_0^t \left(\bar{\sigma}
				+\Pi H^\top (f^\top)^{-1}\right) d \overline{W}_i \right]^2 ,
			\end{aligned}
		\end{equation*}
		where
		\begin{equation*}
			\begin{aligned}
				&\frac{1}{N^2} \mathbb{E} \left[ \sum_{i=1}^N \int_0^t \left(\bar{\sigma}+\Pi H^\top (f^\top)^{-1}\right) d \overline{W}_i \right]^2
				= \frac{1}{N^2} \mathbb{E}  \left[ \sum_{i=1}^N \int_0^t \left(\bar{\sigma}+\Pi H^\top (f^\top)^{-1}\right)^2 d t \right. \\
				&\quad \left.+2\sum_{ i < j} \int_0^t \left(\bar{\sigma}+\Pi H^\top (f^\top)^{-1}\right) d \overline{W}_i  \int_0^t \left(\bar{\sigma}+\Pi H^\top (f^\top)^{-1}\right) d \overline{W}_j \right]
				\leq \frac{K}{N}.
			\end{aligned}
		\end{equation*}
		Therefore,
		\begin{equation}\label{Gronwall1}
			\begin{aligned}
				& \mathbb{E} \left[\left|\hat{\bar{x}}^{u_0,u^*(N)}(t)-\mathbb{E}\bar{x}_i^{u_0,u^*_i}(t)\right|^2\right]
				\leq K \mathbb{E} \int_0^t\Big( \left|\hat{\bar{x}}^{u_0,u^*(N)}(t)-\mathbb{E}\bar{x}_i^{u_0,u^*_i}(t)\right|^2\\
				& +\left|\bar{x}^{u_0,u^*(N)}(s)-\mathbb{E}\bar{x}_i^{u_0,u^*_i}(s)\right|^2\Big) d s+O\left(\frac{1}{N}\right).
			\end{aligned}
		\end{equation}
		Similarly
		\begin{equation}\label{Gronwall2}
			\begin{aligned}
				& \mathbb{E} \left[\left|\bar{x}^{u_0,u^*(N)}(t)-\mathbb{E}\bar{x}_i^{u_0,u^*_i}(t)\right|^2\right]
				\leq K \mathbb{E} \int_0^t\Big( \left|\bar{x}^{u_0,u^*(N)}(t)-\mathbb{E}\bar{x}_i^{u_0,u^*_i}(t)\right|^2\\
				& +\left|\hat{\bar{x}}^{u_0,u^*(N)}(s)-\mathbb{E}\bar{x}_i^{u_0,u^*_i}(s)\right|^2\Big) d s+O\left(\frac{1}{N}\right).
			\end{aligned}
		\end{equation}
		From (\ref{Gronwall1}) and (\ref{Gronwall2}), we have  (\ref{CC1}) and (\ref{CC2}) by applying Gronwall's inequality.
		
		Then, the equation of $x^{u_0,u^*}_i(\cdot)$ is
		\begin{equation}\label{x^*_i}
			\left\{\begin{aligned}
				dx^{u_0,u^*}_i=&  \bigg\{Ax^{u_0,u^*}_i-B R^{-1}\left( B^\top P_1+S\right) \hat{\bar{x}}^{u_0,u^*_i}_i+\bar{A}\mathbb{E}x^{u_0,u^*_i}_i+Cx^{u_0,u^*(N)}-B R^{-1} B^{\top} P_3 \mathbb{E}\bar{x}_0^{u_0}\\
				&+F\mathbb{E}x_0^{u_0,u^*}+\left[BR^{-1}S\bar{\Gamma}_2-B R^{-1} B^\top\left(P_2-P_1\right)\right]\mathbb{E} \bar{x}_i^{u_0,u^*_i}+b-B R^{-1} B^\top \varphi^{u_0}\bigg\} d t\\
				&+\sigma d W_i+\bar{\sigma} d \overline{W}_i\\
				 x^{u_0,u^*}_i(0) =&\xi.
			\end{aligned}\right.
		\end{equation}
		And we have the equation of $x^{u_0,u^*(N)}(\cdot)$ is 
		\begin{equation}\label{x^*N}
			\left\{\begin{aligned}
				dx^{u_0,u^*(N)}= & \bigg\{\left(A+C\right)x^{u_0,u^*(N)}-B R^{-1}\left( B^\top P_1+S\right) \hat{\bar{x}}^{u_0,u^*(N)}+\bar{A}\mathbb{E}x^{u_0,u^*(N)}\\
				& +\left[-B R^{-1} B^\top\left(P_2-P_1\right)+BR^{-1}S\bar{\Gamma}_2\right]\mathbb{E} \bar{x}_i^{u_0,u^*_i}+F\mathbb{E}x_0^{u_0,u^*}-B R^{-1} B^{\top} P_3 \mathbb{E}\bar{x}_0^{u_0}\\
				&+b-B R^{-1} B^\top \varphi^{u_0}\bigg\} d t+\frac{1}{N}\sum_{i=1}^{N}\sigma d W_i+\frac{1}{N}\sum_{i=1}^{N}\bar{\sigma} d \overline{W}_i, \\
				\bar{x}^{u_0,u^*(N)}(0) & =\xi.
			\end{aligned}\right.
		\end{equation}
		Therefore, by (\ref{x^*_i}) and (\ref{E bar x_i}), we obtain
		\begin{equation}\label{x*N-Ebarx}
			\left\{\begin{aligned}
				d	\big(x^{u_0,u^*(N)}&-\mathbb{E}\bar{x}_i^{u_0,u^*_i}\big)  =\left[\left(A+C\right) \big(x^{u_0,u^*(N)}-\mathbb{E}\bar{x}_i^{u_0,u^*_i}\big)+\bar{A}\big(\mathbb{E}x^{u_0,u^*(N)}-\mathbb{E}\bar{x}_i^{u_0,u^*_i}\big)
				\right.\\
				&\left.-B R^{-1}\left( B^\top P_1+S\right)\big(\hat{\bar{x}}^{u_0,u^*(N)}-\mathbb{E}\bar{x}_i^{u_0,u^*_i}\big)+F\big(\mathbb{E}x^{u_0,u^*}_0-\mathbb{E}\bar{x}^{u_0}_0\big)\right]dt\\
				&+\frac{1}{N}\sum_{i=1}^{N} \sigma dW_i+\frac{1}{N}\sum_{i=1}^{N} \bar{\sigma}d\bar{W}_i,\\
				\big(x^{u_0,u^*(N)}&-\mathbb{E}\bar{x}_i^{u_0,u^*_i}\big)(0) = 0.
			\end{aligned}\right.
		\end{equation}
		And, from (\ref{leader state}) and the first equation of (\ref{CC system}), we have
		\begin{equation}\label{x_0-barx_0}
			\left\{\begin{aligned}
				d	\left(x_0^{u_0,u^*}-\bar{x}_0^{u_0}\right)  =&\left[A_0\left(x_0^{u_0,u^*}-\bar{x}_0^{u_0}\right)+\bar{A}_0\mathbb{E}\left(x_0^{u_0,u^*}-\bar{x}_0^{u_0}\right) +C_0\left(x^{u_0,u^*(N)}-\mathbb{E}\bar{x}_i^{u_0}\right)\right]dt,\\
				\left(x_0^{u_0,u^*}-\bar{x}_0^{u_0}\right) (0) &=\ 0.
			\end{aligned}\right.
		\end{equation}
		By integrating and taking expectation on both sides of the
		equation (\ref{x*N-Ebarx}), we have
		\begin{equation}\label{estimation 1}
			\begin{aligned}
				& \mathbb{E}\left|x^{u_0,u^*(N)}(t)-\mathbb{E}\bar{x}_i^{u_0,u^*_i}(t)\right|^2 \leq K\mathbb{E}\left\{\int_0^t\left(\left|x^{u_0,u^*(N)}(s)-\mathbb{E}\bar{x}_i^{u_0,u^*_i}(s)\right|^2+\Big|\hat{\bar{x}}^{u_0,u^*(N)}(s)\right.\right. \\
				& \quad
				-\mathbb{E}\bar{x}_i^{u_0,u^*_i}(s)\Big|^2+\left|\mathbb{E} x_0^{u_0,u^*}(s)-\mathbb{E} \bar{x}_0^{u_0}(s)\right|^2\bigg) d s\bigg\} +\frac{1}{N^2}\sum_{i=1}^N \mathbb{E} \int_0^t (\sigma^2+\bar{\sigma}^2) d s \\
				&\leq K\mathbb{E}\left\{\int_0^t\left(\left|x^{u_0,u^*(N)}(s)-\mathbb{E}\bar{x}_i^{u_0,u^*_i}(s)\right|^2+\left| x_0^{u_0,u^*}(s)- \bar{x}_0^{u_0}(s)\right|^2\right) d s\right\}+O\left(\frac{1}{N}\right) .
			\end{aligned}
		\end{equation}
		Similarly, by integrating and taking expectation on both sides of the equation (\ref{x_0-barx_0}), we have
		\begin{equation}\label{estimation 2}
			\begin{aligned}
				& \mathbb{E} \left[\left| x_0^{u_0,u^*}(t)-\bar{x}_0^{u_0}(t)\right|^2\right]\leq K \mathbb{E} \int_0^t\left(\left| x_0^{u_0,u^*}(s)-\bar{x}_0^{u_0}(s)\right|^2+\left|x^{u_0,u^*(N)}(s)-\mathbb{E}\bar{x}_i^{u_0,u^*_i}(s)\right|^2\right) d s\,.
			\end{aligned}
		\end{equation}
		Adding (\ref{estimation 1}) and (\ref{estimation 2}) together, we obtain
		\begin{equation}
			\begin{aligned}
				& \mathbb{E} \left[\left| x_0^{u_0,u^*}(t)-\bar{x}_0^{u_0}(t)\right|^2\right]+\mathbb{E}\left|x^{u_0,u^*(N)}(t)-\mathbb{E}\bar{x}_i^{u_0,u^*_i}(t)\right|^2 \\
				&\leq K\mathbb{E}\int_0^t\left(\left|x_0^{u_0,u^*}(s)-\bar{x}_0^{u_0}(s)\right|^2+\left|x^{u_0,u^*(N)}(s)-\mathbb{E}\bar{x}_i^{u_0}(s)\right|^2\right)d s +O\left(\frac{1}{N}\right).
			\end{aligned}
		\end{equation}
		By applying Gronwall's inequality, (\ref{CC3}) and (\ref{leader approximation}) holds. And using similar method, we can obtain (\ref{follower approximation}) by (\ref{CC3}) and (\ref{leader approximation}).
		Then, the proof is complete.
	\end{proof}

\subsection{Proof of Lemma \ref{left half of follower NE}}

	\begin{proof}
		From (\ref{follower cost}) and (\ref{follower limiting cost}), we get
		\begin{equation}\label{follower cost estimation optimal}
			\begin{aligned}
				&\left|\mathcal{J}_i(u_i^*(\cdot), u_{-i}^*(\cdot),u_0(\cdot))-J_i(u_i^*(\cdot),u_0(\cdot))\right|\\
				&=\frac{1}{2} \mathbb{E}\bigg\{\int_0^T \left\{\Big\langle Q\left[x^{u_0,u^*}_i-\bar{x}^{u_0,u^*_i}_i-\Gamma_2(x^{u_0,u^*(N)}-\mathbb{E}\bar{x}_i^{u_0,u^*_i})
				-\bar{\Gamma}_2(\mathbb{E}x^{u_0,u^*}_i-\mathbb{E}\bar{x}^{u_0,u^*_i}_i)\right.\right.\\
				&\qquad\quad \left.-\Gamma_3(x^{u_0,u^*}_0-\bar{x}_0^{u_0})-\bar{\Gamma}_3(\mathbb{E}x_0^{u_0,u^*}-\mathbb{E}\bar{x}_0^{u_0})\right],x^{u_0,u^*}_i-\bar{x}^{u_0,u^*_i}_i-\Gamma_2(x^{u_0,u^*(N)}-\mathbb{E}\bar{x}_i^{u_0,u^*_i})
				\\
				&\qquad\quad -\bar{\Gamma}_2(\mathbb{E}x^{u_0,u^*}_i-\mathbb{E}\bar{x}_i^{u_0,u^*_i})-\Gamma_3(x_0^{u_0,u^*}-\bar{x}_0^{u_0})-\bar{\Gamma}_3(\mathbb{E}x_0^{u_0,u^*}-\mathbb{E}\bar{x}_0^{u_0}) \Big\rangle\\
				&\qquad\quad +2\Big\langle Q\Big[x^{u_0,u^*}_i-\bar{x}^{u_0,u^*_i}_i-\Gamma_2(x^{u_0,u^*(N)}-\mathbb{E}\bar{x}^{u_0,u^*_i}_i)
				-\bar{\Gamma}_2(\mathbb{E}x^{u_0,u^*(N)}-\mathbb{E}\bar{x}^{u_0,u^*_i}_i)\\
				&\qquad\quad-\Gamma_3(x_0^{u_0,u^*}-\bar{x}_0^{u_0})-\bar{\Gamma}_3(\mathbb{E}x_0^{u_0,u^*}-\mathbb{E}\bar{x}_0^{u_0})\Big],\bar{x}^{u_0,u^*_i}_i-\Gamma_2\mathbb{E}\bar{x}^{u_0,u^*_i}_i
				-\bar{\Gamma}_2\mathbb{E}\bar{x}^{u_0,u^*_i}_i-\Gamma_3 \bar{x}_0^{u_0}\\
				&\qquad\quad \left.-\bar{\Gamma}_3\mathbb{E}\bar{x}_0^{u_0}-\eta_2\Big\rangle+2\Big\langle S\left[x^{u_0,u^*}_i-\bar{x}^{u_0,u^*_i}_i-\bar{\Gamma}_2(\mathbb{E}x^{u_0,u^*}_i-\mathbb{E}\bar{x}_i^{u_0,u^*_i})\right],u^*_i\Big\rangle\right\} dt\\
				&\qquad\quad +\Big\langle G\left[x^{u_0,u^*}_i-\bar{x}^{u_0,u^*_i}_i-\Gamma_4(x^{u_0,u^*(N)}-\mathbb{E}\bar{x}_i^{u_0,u^*_i})-\bar{\Gamma}_4(\mathbb{E}x^{u_0,u^*}_i-\mathbb{E}\bar{x}^{u_0,u^*_i}_i)\right.\\
				&\qquad\quad -\Gamma_5(x_0^{u_0,u^*}-\bar{x}_0^{u_0})-\bar{\Gamma}_5(\mathbb{E}x_0^{u_0,u^*}-\mathbb{E}\bar{x}_0^{u_0})\Big],x^{u_0,u^*}_i-\bar{x}^{u_0,u^*_i}_i-\Gamma_4(x^{u_0,u^*(N)}-\mathbb{E}\bar{x}_i^{u_0,u^*_i})\\
				&\qquad\quad -\bar{\Gamma}_4(\mathbb{E}x^{u_0,u^*}_i-\mathbb{E}\bar{x}^{u_0,u^*_i}_i)-\Gamma_5(x_0^{u_0,u^*}-\bar{x}_0^{u_0})-\bar{\Gamma}_5(\mathbb{E}x_0^{u_0,u^*}-\mathbb{E}\bar{x}_0^{u_0})\Big\rangle(T) \\
				&\qquad\quad +2\Big\langle G\Big[(x^{u_0,u^*}_i-\bar{x}^{u_0,u^*_i}_i
				-\Gamma_4(x^{u_0,u^*(N)}-\mathbb{E}\bar{x}_i^{u_0,u^*_i})-\bar{\Gamma}_4(\mathbb{E}x^{u_0,u^*}_i-\mathbb{E}\bar{x}^{u_0,u^*_i}_i)\\
				&\qquad\quad-\Gamma_5(x_0^{u_0,u^*}-\bar{x}_0^{u_0})-\bar{\Gamma}_5(\mathbb{E}x_0^{u_0,u^*}-\mathbb{E}\bar{x}_0^{u_0})\Big],\bar{x}^{u_0,u^*}_i-\Gamma_4 \mathbb{E}\bar{x}_i^{u_0,u^*_i}-\bar{\Gamma}_4  \mathbb{E}\bar{x}^{u_0,u^*_i}_i-\Gamma_5 \bar{x}_0^{u_0}\\
				&\qquad\quad-\bar{\Gamma}_5\mathbb{E}\bar{x}_0^{u_0}-\eta_4\Big\rangle (T)\bigg\}\,.
			\end{aligned}
		\end{equation}
		Noting Lemma \ref{CC lemma}, we can obtain
		\begin{equation*}
			\begin{aligned}
				&\mathbb{E}\int_0^T\Big\{ \Big\langle Q\big[x^{u_0,u^*}_i-\bar{x}^{u_0,u^*_i}_i-\Gamma_2(x^{u_0,u^*(N)}-\mathbb{E}\bar{x}_i^{u_0,u^*_i})-\bar{\Gamma}_2(\mathbb{E}x^{u_0,u^*}_i-\mathbb{E}\bar{x}^{u_0,u^*_i}_i)-\Gamma_3(x_0^{u_0,u^*}-\bar{x}_0^{u_0})\\
				&\qquad -\bar{\Gamma}_3(\mathbb{E}x_0^{u_0,u^*}-\mathbb{E}\bar{x}_0^{u_0})\big],x^{u_0,u^*}_i-\bar{x}^{u_0,u^*_i}_i-\Gamma_2(x^{u_0,u^*(N)}-\mathbb{E}\bar{x}_i^{u_0,u^*_i})-\bar{\Gamma}_2(\mathbb{E}x^{u_0,u^*}_i-\mathbb{E}\bar{x}^{u_0,u^*_i}_i)\\
				&\qquad -\Gamma_3({x}_0^{u_0,u^*}-\bar{x}_0^{u_0,u^*_i})-\bar{\Gamma}_3(\mathbb{E}x_0^{u_0,u^*}-\mathbb{E}\bar{x}_0^{u_0}) \Big\rangle
				+2\Big\langle Q\big[x^{u_0,u^*}_i-\bar{x}^{u_0,u^*_i}_i-\Gamma_2(x^{u_0,u^*(N)}\\
				&\qquad -\mathbb{E}\bar{x}_i^{u_0,u^*_i})-\bar{\Gamma}_2(\mathbb{E}x^{u_0,u^*}_i-\mathbb{E}\bar{x}^{u_0,u^*_i}_i)-\Gamma_3(x_0^{u_0,u^*}-\bar{x}_0^{u_0})-\bar{\Gamma}_3(\mathbb{E}x_0^{u_0,u^*}-\mathbb{E}\bar{x}_0^{u_0})\Big],\bar{x}^{u_0,u^*_i}_i\\
				&\qquad-\Gamma_2 \mathbb{E}\bar{x}_i^{u_0,u^*_i}
				-\bar{\Gamma}_2 \mathbb{E}\bar{x}^{u_0,u^*_i}_i-\Gamma_3 \bar{x}_0^{u_0}-\bar{\Gamma}_3\mathbb{E}\bar{x}_0^{u_0}-\eta_2\Big\rangle \Big\}dt\\
&\leq K\mathbb{E}\bigg\{\int_0^T\Big|x^{u_0,u^*}_i-\bar{x}^{u_0,u^*_i}_i-\Gamma_2(x^{u_0,u^*(N)}-\mathbb{E}\bar{x}_i^{u_0,u^*_i})-\bar{\Gamma}_2(\mathbb{E}x^{u_0,u^*}_i-\mathbb{E}\bar{x}^{u_0,u^*_i}_i)-\Gamma_3({x}_0^{u_0,u^*}
				\\
				&\qquad -\bar{x}_0^{u_0})-\bar{\Gamma}_3(\mathbb{E}x_0^{u_0,u^*}-\mathbb{E}\bar{x}_0^{u_0})\Big|^2dt+\int_0^T\Big[\mathbb{E}\Big|(x^{u_0,u^*}_i-\bar{x}^{u_0,u^*_i}_i)-\Gamma_2(x^{u_0,u^*(N)}-\mathbb{E}\bar{x}_i^{u_0,u^*_i})\\
				&\qquad-\bar{\Gamma}_2(\mathbb{E}x^{u_0,u^*_i}_i-\mathbb{E}\bar{x}^{u_0,u^*_i}_i)-\Gamma_3({x}_0^{u_0,u^*}-\bar{x}_0^{u_0}) -\bar{\Gamma}_3(\mathbb{E}x_0^{u_0,u^*}-\mathbb{E}\bar{x}_0^{u_0})\Big|^2\Big]^{\frac{1}{2}}
				\Big[\mathbb{E}\Big|\bar{x}^{u_0,u^*_i}_i\\
&\qquad-\Gamma_2\mathbb{E}\bar{x}_i^{u_0,u^*_i}-\bar{\Gamma}_2\mathbb{E}\bar{x}_i^{u_0,u^*_i}-\Gamma_3\bar{x}_0^{u_0}-\bar{\Gamma}_3\mathbb{E}\bar{x}_0^{u_0}-\eta_2\Big|^2\Big]^{\frac{1}{2}}dt\bigg\}\\
				&\leq K \int_0^T \bigg(\sup _{0 \leq t \leq T}\mathbb{E}\Big|x_i^{u_0,u^*}-\bar{x}_i^{u_0,u^*_i}\Big|^2+\sup _{0 \leq t \leq T}\mathbb{E}\Big|x_0^{u_0,u^*}-\bar{x}^{u_0}_0\Big|^2
				\\
				&\qquad+\sup _{0 \leq t \leq T}\mathbb{E} \Big|x^{u_0,u^*(N)}-\mathbb{E}\bar{x}_i^{u_0,u^*_i}\Big|^2\bigg) d t+K\int_0^T\bigg[\sup _{0 \leq t \leq T} \mathbb{E}\Big|x_i^{u_0,u^*}-\bar{x}_i^{u_0,u^*_i}\Big|^2 \\
				&\qquad+\sup _{0 \leq t \leq T}\mathbb{E}\Big|x_0^{u_0,u^*}-\bar{x}_0^{u_0}\Big|^2
				+\sup _{0 \leq t \leq T}\mathbb{E} \Big|x^{u_0,u^*(N)}-\mathbb{E}\bar{x}_i^{u_0,u^*_i}\Big|^2\bigg]^{\frac{1}{2}} dt=O\left(\frac{1}{\sqrt{N}}\right).
			\end{aligned}
		\end{equation*}
		Similarly, the order of the terminal term of (\ref{follower cost estimation optimal}) is also $O\left(\frac{1}{\sqrt{N}}\right)$. Then, 
		$$
		\begin{aligned}
		&\mathbb{E}\int_0^T\Big\langle S\left[x^{u_0,u^*}_i-\bar{x}^{u_0,u^*_i}_i-\bar{\Gamma}_2(x^{u_0,u^*(N)}-\mathbb{E}\bar{x}_i^{u_0,u^*_i})\right],u^*_i\big\rangle dt\\
		&\leq K\int_0^T \left[\mathbb{E}\left(\left|x^{u_0,u^*}_i-\bar{x}^{u_0,u^*_i}_i\right|^2+\left|x^{u_0,u^*(N)}-\mathbb{E}\bar{x}_i^{u_0,u^*_i}\right|^2\right)\right]^{\frac{1}{2}} dt=O\left(\frac{1}{\sqrt{N}}\right).
	    \end{aligned} 
		$$
		Therefore, (\ref{estimate followers cost}) is proved. And then, when $u_0=u_0^*$, we apply (\ref{leader approximation}) of Lemma \ref{CC lemma} to get
		\begin{equation}\label{leader approximation *}
			\mathbb{E} \left[\sup _{0 \leq t \leq T}\left|x^{u_0,u^*}_0(t)-\bar{x}^{u_0}_0(t)\right|^2\right] = O\left(\frac{1}{N}\right).
		\end{equation}
		Using similar method, we can prove (\ref{estimate leader cost}) by (\ref{CC3}) and (\ref{leader approximation *}). The proof is complete.
	\end{proof}

\subsection{Proof of Lemma \ref{follower perturbed CC lemma}}

\begin{proof}
		Due to (\ref{m_i follower state perturbation}) and (\ref{m_j perturbation}), we have
		\begin{equation}\label{m^N equation}
			\left\{\begin{aligned}
				d m^{(N)}= & \left[\left(A+C\right) m^{(N)}+\frac{1}{N}B\left(v_i+\sum_{j=1,j \neq i}^{N}u_j^*\right)+\bar{A}\mathbb{E}m^{(N)}+F \mathbb{E}m_0+b\right] d t \\
				& +\frac{1}{N}\sum_{i=1}^{N}\sigma d W_i+\frac{1}{N}\sum_{i=1}^{N}\bar{\sigma} d \overline{W}_i, \\
				m^{(N)}(0) & =\xi.
			\end{aligned}\right.
		\end{equation}
		From (\ref{m_i follower state perturbation}) and (\ref{x^*N}), we obtain
		\begin{equation*}
			\left\{\begin{aligned}
				d	\left({m}^{(N)}-x^{u_0,u^*(N)}\right) =&\left[\left(A+C\right) \left(m^{(N)}-x^{u_0,u^*(N)}\right)+\frac{B}{N}\left(v_i-u_i^*\right) \right.\\
				&\left.+\bar{A}\left(\mathbb{E}m^{(N)}-\mathbb{E}x^{u_0,u^*(N)}\right)+F\left(\mathbb{E}m_0-\mathbb{E}x_0^{u_0,u^*}\right)\right]dt,\\
				\left({m}^{(N)}-x^{u_0,u^*(N)}\right)(0) =&\ 0.
			\end{aligned}\right.
		\end{equation*}
		By integrating and taking expectation on both sides of the above equation, we have
		\begin{equation}\label{estimation perturbation 1}
			\begin{aligned}
				& \mathbb{E}\left|{m}^{(N)}(t)-x^{u_0,u^*(N)}(t)\right|^2 \leq K\mathbb{E}\bigg\{\int_0^t\left|{m}^{(N)}(s)-x^{u_0,u^*(N)}(s)\right|^2+\left|\frac{B}{N}\left(v_i(s)-u_i^*(s)\right)\right|^2 \\
				&\qquad\quad +\left|\mathbb{E} m_0(s)-\mathbb{E} x_0^{u_0,u^*}(s)\right|^2 d s\bigg\}= K\mathbb{E}\bigg\{\int_0^t\left|{m}^{(N)}(s)-\mathbb{E}\bar{x}_i^{u_0,u^*_i}(s)\right|^2\\
				&\qquad\quad +\left| m_0(s)- \bar{x}_0^{u_0}(s)\right|^2 d s+O\left(\frac{1}{N}\right)\bigg\}.
			\end{aligned}
		\end{equation}
		Similarly, from (\ref{m_0 leader state perturbation}) and (\ref{leader state}), we have
		\begin{equation*}
			\left\{\begin{aligned}
				d\left(m_0-x_0^{u_0,u^*}\right)  =&\left[A_0\left(m_0-x_0^{u_0,u^*}\right)+\bar{A}_0\mathbb{E}\left(m_0-x_0^{u_0,u^*}\right) +C_0\left(m^{(N)}-x^{u_0,u^*(N)}\right)\right]dt,\\
				\left(m_0-x_0^{u_0,u^*}\right) (0) = &\ 0.
			\end{aligned}\right.
		\end{equation*}
		By integrating and taking expectation on both sides of the above equation, we have
		\begin{equation}\label{estimation perturbation 2}
			\begin{aligned}
				& \mathbb{E} \left[\left| m_0(t)-x_0^{u_0,u^*}(t)\right|^2\right]\\
				&\leq K \mathbb{E} \int_0^t\left(\left| m_0(s)-x_0^{u_0,u^*}(s)\right|^2+\left| \mathbb{E}m_0(s)-\mathbb{E}x_0^{u_0,u^*}(s)\right|^2+\left|m^{(N)}(s)-x^{u_0,u^*(N)}(s)\right|^2\right) d s \\
				&\leq  K \mathbb{E} \int_0^t\left(\left| m_0(s)-x_0^{u_0,u^*}(s)\right|^2+\left|m^{(N)}(s)-x^{u_0,u^*(N)}(s)\right|^2\right) d s.
			\end{aligned}
		\end{equation}
		Adding (\ref{estimation perturbation 1}) and (\ref{estimation perturbation 2}) together, we obtain
		\begin{equation}\label{Gronwall 1}
			\begin{aligned}
				& \mathbb{E} \left[\left| m_0(t)-x_0^{u_0,u^*}(t)\right|^2\right]+\mathbb{E}\left|m^{(N)}(t)-x^{u_0,u^*(N)}(t)\right|^2 \\
				&\leq K\mathbb{E}\int_0^t\left(\left|m_0(s)-x_0^{u_0,u^*}(s)\right|^2+\left|m^{(N)}(s)-x^{u_0,u^*(N)}(s)\right|^2\right)d s +O\left(\frac{1}{N}\right).
			\end{aligned}
		\end{equation}
		From (\ref{Gronwall 1}), by applying Gronwall's inequality, we have
		\begin{equation}\label{follower m-x}
			\sup _{0 \leq t \leq T}\mathbb{E} \left[\left|m^{(N)}(t)- x^{u_0,u^*(N)}(t)\right|^2\right] =O\left(\frac{1}{N}\right),
		\end{equation}
		\begin{equation}\label{leader m0-x0}
			\sup _{0 \leq t \leq T}\mathbb{E} \left[\left|m_0(t)-x_0^{u_0,u^*}(t)\right|^2\right] =O\left(\frac{1}{N}\right).
		\end{equation}
		Then, we can get 
		(\ref{follower perturbed CC 1}) from (\ref{CC3}) and (\ref{follower m-x}). Similarly, we can get (\ref{follower perturbed CC 2}) from (\ref{leader approximation}) and (\ref{leader m0-x0}). And using similar method, we can obtain (\ref{follower perturbed CC 3}) by (\ref{follower perturbed CC 1}) and (\ref{follower perturbed CC 2}).
		Then, the proof is complete.
	\end{proof}

\subsection{Proof of Lemma \ref{right half of follower NE}}

\begin{proof}
		Applying Lemma \ref{left half of follower NE} and Lemma \ref{follower perturbed CC lemma}, we obtain
		\begin{equation*}
			\begin{aligned}
				&\left|\mathcal{J}_i(v_i(\cdot), u_{-i}^*(\cdot))-J_i(v_i(\cdot))\right|
				=\frac{1}{2} \mathbb{E}\bigg\{\int_0^T\bigg\{ \Big\langle Q\Big[m_i-\bar{m}_i-\Gamma_2({m}^{(N)}-\mathbb{E}\bar{x}_i^{u_0,u^*_i})-\bar{\Gamma}_2(\mathbb{E}m_i-\mathbb{E}\bar{m}_i)\\
				& -\Gamma_3({m}_0-\bar{x}_0^{u_0,u^*})-\bar{\Gamma}_3(\mathbb{E}m_0-\mathbb{E}\bar{x}_0^{u_0})\Big],m_i-\bar{m}_i-\Gamma_2({m}^{(N)}-\mathbb{E}\bar{x}_i^{u_0,u^*_i})-\bar{\Gamma}_2(\mathbb{E}m_i-\mathbb{E}\bar{m}_i)\\
				& -\Gamma_3({m}_0-\bar{x}_0)-\bar{\Gamma}_3(\mathbb{E}m_0-\mathbb{E}\bar{x}_0^{u_0}) \Big\rangle
				+2\Big\langle Q\Big[m_i-\bar{m}_i-\Gamma_2({m}^{(N)}-\mathbb{E}\bar{x}_i^{u_0,u^*_i})-\bar{\Gamma}_2(\mathbb{E}m_i-\mathbb{E}\bar{m}_i)\\
				& -\Gamma_3({m}_0-\bar{x}_0^{u_0})-\bar{\Gamma}_3(\mathbb{E}m_0-\mathbb{E}\bar{x}_0^{u_0})\Big],\bar{m}_i-\Gamma_2 \mathbb{E}\bar{x}_i^{u_0,u^*_i}
				-\bar{\Gamma}_2 \mathbb{E}\bar{m}_i-\Gamma_3 \bar{x}_0^{u_0}-\bar{\Gamma}_3\mathbb{E}\bar{x}_0^{u_0}-\eta_2\Big\rangle
				\\
				& +2\Big\langle S\left[m_i-\bar{m}^*_i-\bar{\Gamma}_2({m}^{(N)}-\mathbb{E}\bar{x}_i^{u_0,u^*_i})\right],u^*_i\big\rangle \bigg\} dt\\
				&+\Big\langle G\Big[m_i-\bar{m}_i-\Gamma_4({m}^{(N)}-\mathbb{E}\bar{x}_i^{u_0,u^*_i})-\bar{\Gamma}_4(\mathbb{E}m_i-\mathbb{E}\bar{m}_i)-\Gamma_5({m}_0-\bar{x}_0^{u_0})-\bar{\Gamma}_5(\mathbb{E}m_0-\mathbb{E}\bar{x}_0^{u_0})\Big],\\
				&\quad m_i-\bar{m}_i-\Gamma_4({m}^{(N)}-\mathbb{E}\bar{x}_i^{u_0,u^*_i})-\bar{\Gamma}_4(\mathbb{E}m_i-\mathbb{E}\bar{m}_i)-\Gamma_5({m}_0-\bar{x}_0^{u_0})-\bar{\Gamma}_5(\mathbb{E}m_0-\mathbb{E}\bar{x}_0^{u_0}) \Big\rangle(T)\\
				&+2\Big\langle G\Big[m_i-\bar{m}_i-\Gamma_4({m}^{(N)}-\mathbb{E}\bar{x}_i^{u_0,u^*_i})-\bar{\Gamma}_4(\mathbb{E}m_i-\mathbb{E}\bar{m}_i)-\Gamma_5({m}_0-\bar{x}_0^{u_0})-\bar{\Gamma}_5(\mathbb{E}m_0-\mathbb{E}\bar{x}_0^{u_0})\Big],\\
				&\quad \bar{m}_i-\Gamma_4 \mathbb{E}\bar{x}_i^{u_0,u^*_i}-\bar{\Gamma}_4  \mathbb{E}\bar{m}_i-\Gamma_5 \bar{x}_0^{u_0}-\bar{\Gamma}_5\mathbb{E}\bar{x}_0^{u_0}-\eta_4\Big\rangle (T)\bigg\}\,.\\
			\end{aligned}
		\end{equation*}
		Applying the method similarly as Lemma \ref{left half of follower NE} and noting (\ref{follower perturbed CC 1}), (\ref{follower perturbed CC 2}), (\ref{follower perturbed CC 3}), we can get (\ref{right half of follower NE}). Then, we complete the proof.
	\end{proof}
	
\end{document}